\documentclass[review]{elsarticle}

\usepackage{lineno,hyperref}
\modulolinenumbers[5]
\usepackage{multirow,array}

\makeatletter
\def\ps@pprintTitle{%
 \let\@oddhead\@empty
 \let\@evenhead\@empty
 \def\@oddfoot{}%
 \let\@evenfoot\@oddfoot}
\makeatother
\usepackage{color}
\usepackage{float}
\usepackage{comment}
\usepackage[labelformat=simple]{subcaption}

\usepackage{amsfonts}
\usepackage{libertine}
\usepackage{comment}
\usepackage{enumerate}
\usepackage[letterpaper, top=2.5cm, bottom=2.5cm, left=2.2cm, right=2.2cm]{geometry}
\usepackage{times}
\usepackage{amsmath}
\usepackage{changepage}
\usepackage{amssymb}
\usepackage{epstopdf,comment}
\usepackage{wrapfig}

\usepackage[amsmath,thmmarks]{ntheorem}
\usepackage{graphicx}
\usepackage{todonotes}%
\newtheorem{theorem}{Theorem}[section]

\newtheorem{corollary}[theorem]{Corollary}

\newtheorem{definition}[theorem]{Definition}

\newtheorem{lemma}[theorem]{Lemma}

\newtheorem{proposition}[theorem]{Proposition}
\newtheorem{remark}[theorem]{Remark}

\newenvironment{proof}[1][Proof]{\textbf{#1.} }{\ \rule{0.5em}{0.5em}}

\newcommand{\erfc}{\text{erfc}}
\newcommand{\range}{\text{range}}
\DeclareMathOperator*{\argmin}{arg\,min}
\DeclareMathOperator*{\argmax}{arg\,max}
\newcommand{\esssup}{\operatorname{ess\,sup}}
\newcommand{\Var}{\operatorname{Var}}
\newcommand{\req}[1]{Eq.\,(\ref{#1})}

\begin{document}
\begin{frontmatter}
\title{Distributional Robustness and Uncertainty Quantification for Rare Events} 
\author[UMASS]{Jeremiah Birrell}\ead{birrell@math.umass.edu}
\author[BU]{Paul Dupuis}\ead{dupuis@dam.brown.edu}
\author[UMASS]{Markos A. Katsoulakis}\ead{markos@math.umass.edu}
\author[UMASS]{Luc Rey-Bellet}\ead{luc@math.umass.edu}
\author[UMASS]{Jie Wang}\ead{wang@math.umass.edu}
\address[UMASS]{Department of Mathematics and Statistics, University of Massachusetts Amherst, Amherst, MA 01003, USA} 
\address[BU]{Division of Applied Mathematics, Brown University, Providence, RI 02912, USA} 

\date{\today}
\begin{abstract}
Rare events, and more general risk-sensitive quantities-of-interest (QoIs), are significantly impacted by uncertainty in the tail behavior of a distribution.  Uncertainty in the tail can take many different forms, each of which leads to a particular ambiguity set of alternative models. Distributional robustness bounds over such an ambiguity set constitute a stress-test of the model.  In this paper we develop a method,  utilizing R{\'e}nyi-divergences,  of constructing the ambiguity set that captures a user-specified form of tail-perturbation.  We then obtain distributional  robustness bounds (performance guarantees) for risk-sensitive QoIs over these ambiguity sets, using the known connection between   R{\'e}nyi-divergences and robustness for risk-sensitive QoIs. We also  expand on this  connection in several ways, including a generalization of the Donsker-Varadhan variational formula to R{\'e}nyi divergences, and various tightness results. These ideas are illustrated through applications to uncertainty quantification  in a model of lithium-ion battery failure, robustness of large deviations rate functions, and  risk-sensitive distributionally robust optimization for option pricing.

\end{abstract}
\bigskip
\end{frontmatter}
%%%%%%%%%%%%%%%%%%%%%%%%%%%%%%%%%%%%%%%%%%
\section{Introduction}
Risk-sensitive quantities-of-interest (QoI), such as rare-events, partition functions, moment generating functions, and exit times, are of great interest in many applications from engineering, biology, chemistry, and finance.  Significant model uncertainty stemming  from uncertain parameter values, uncertain model form, or some approximation procedure (variational inference, dimension reduction, neglecting memory terms,  linearization, asymptotic approximation, etc.), is inherent in many such systems. Therefore, it is important to study the dependence of risk-sensitive QoIs on model perturbations. In \cite{atar2015robust,2018arXiv180506917D}, it was shown how robustness bounds for risk-sensitive QoIs can be obtained from  variational principles involving information-divergences (specifically, R{\'e}nyi-divergences). Such bounds can be viewed as performance guarantees or stress tests over a collection of alternative models, called an ambiguity set. In this work, we characterize the structure of R{\'e}nyi-divergence-based ambiguity sets by utilizing the connection between R{\'e}nyi-divergences and the moment-generating-function (MGF) of the log-likelihood. More specifically, we show how to construct ambiguity sets that correspond to various tail-behavior (stress-test) scenarios.  We also extend some of the foundational theory connecting risk-sensitive QoIs to information divergences and robust optimization, including a generalization of the Donsker-Varadhan variational formula to R{\'e}nyi divergences (see Theorem \ref{thm:gen_DV}). We show how these tools can be  combined  to obtain explicit uncertainty quantification (UQ) bounds, and illustrate this through several applications:  robustness of a data-driven model of  lithium-ion  battery failure, large deviations rate functions, and   robust option pricing.

\subsection{Risk-Sensitive QoIs, Ambiguity Sets, and Robust UQ}\label{sec:motivation}

\begin{wrapfigure}{r}{.5\textwidth}
\vspace{-1.1cm}
\begin{center}
  \includegraphics[width=.5\textwidth]{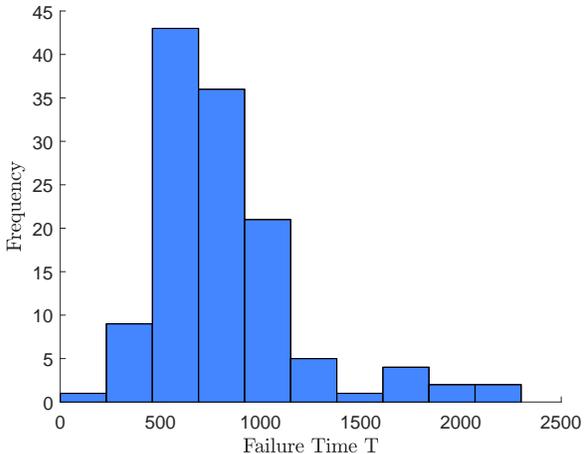}
\end{center}
\vspace{-.5cm}
\caption{Histogram of the battery failure data from \cite{Severson_battery}.  The small number of samples in the tails of the distribution highlights the need for robust UQ when dealing with rare-events and other risk-sensitive QoIs.   }\label{fig:hist_frequency}
\end{wrapfigure}

As motivation for the notion of ambiguity set and the problem of distributionally robust UQ, consider the data set of 124 lithium-ion-battery failure times from \cite{Severson_battery}. A  typical workflow for analyzing  QoIs (expectations of  random variables) might consist of:
\begin{align}\label{eq:workflow}
\text{Data}\xrightarrow[\text{fit/learn}]{}\text{ Model, $P$}\xrightarrow[\text{QoI $\tau$}]{} E_P[\tau],
\end{align}
i.e., the baseline model $P$ (a probability measure) is constructed from the data, and properties of QoIs are computed from $P$ ($E_P$ denotes the expectation under $P$).

A histogram of the data from \cite{Severson_battery}, shown in Figure \ref{fig:hist_frequency}, illustrates a pervasive, unsurprising, but crucial problem with the process (\ref{eq:workflow})  when $\tau$ is   sensitive to the tail(s) of the distribution, i.e., when $\tau$ is risk-sensitive: there is very  little data regarding the tails of the distribution! This paucity of data calls into question the computed quantity $E_P[\tau]$ for risk-sensitive $\tau$, e.g.,  early failure probabilities or the failure rate.  This motivates our consideration of the following distributionally robust UQ problem:
\begin{equation}\label{eq:general_goal}
{\bf Bound:}\,\,\,\,\,\,\, \log(E_Q[\tau]),\,\,\, Q\in\mathcal{U}.
\end{equation}
The robust UQ problem (\ref{eq:general_goal}) involves bounding the QoI over an ambiguity set, $\mathcal{U}$, of alternative models $Q$ (probability measures) that are `near' to $P$ in some sense.
\begin{remark}\label{remark:LDP}
One could equivalently formulate  (\ref{eq:general_goal}) without the logarithm, but in our approach it is generally convenient to include it.  This is further motivated by the following: if $\tau=1_A$ for some (rare) event $A$, then the QoI becomes $\log(P(A))$.  In this form, the connection with tail-behavior is apparent, as $\log(P(A))$ is of central importance in large deviations theory. Our new results, and those of \cite{atar2015robust,2018arXiv180506917D},  have many  connections to large deviations (see also Section \ref{sec:rate_func} below), though we are largely interested in obtaining non-asymptotic robustness bounds.
\end{remark}

The choice of $\mathcal{U}$ in (\ref{eq:general_goal}) is a modeling choice, and each such ambiguity set can be viewed  as a stress-test of $E_P[\tau]$ under tail-perturbations of $P$ of a particular type.  This motivates the stress-test design problem: how does one construct a $\mathcal{U}$ that realizes the desired  stress test?  A solution to this problem is given in Section \ref{sec:tail_behavior}, where we  derive a general method for constructing ambiguity sets that correspond to  stress-testing $E_P[\tau]$ under a prescribed type of tail-perturbation (i.e., prescribing $P(dQ/dP\geq r)$). For example,  $\mathcal{U}$ could consist of $Q$'s whose tails with respect to $P$ have power-law decay under $P$. A given modeling scenario might motivate a particularly  appropriate choice of $\mathcal{U}$; we will discuss this in the examples  in Section \ref{sec:applications}. The neighborhoods we construct will be defined in terms of R{\'e}nyi-divergences, and are thus capable of capturing non-parametric model perturbations.  We also describe several methods of constructing new members of a given ambiguity set, either out of $P$ or of other known members; these further highlight the  non-parametric nature of the results. 
\begin{remark}
The ambiguity set $\mathcal{U}$ could be chosen to be a subset of some finite-dimensional parametric family that includes $P$ (i.e., perturbing the parameters that define $P$), and our methods can be applied to such cases. However, our goal here is the development of methods that apply equally well to non-parametric (infinite dimensional) model neighborhoods.
\end{remark}

 Once  an ambiguity set  is chosen, the distributional robustness  problem of \req{eq:general_goal} can be addressed using tools from information theory.  Specifically, the following risk-sensitive robustness bounds  were proven in \cite{atar2015robust} (where a similar lower bound can also be found):
\begin{proposition}\label{prop:UQ1}
Let $P,Q$ be probability measures on $(\Omega,\mathcal{M})$ and $g:\Omega\to\mathbb{R}$ be measurable.  Then
\begin{align}\label{eq:UQ1_upper}
\log\left[\int e^{ g} dQ\right]\leq \inf_{c>1}\left\{\frac{1}{c} \log\left[\int e^{c g}dP \right]+\frac{1}{c-1}R_{c/(c-1)}(Q\|P)\right\}.
\end{align}
If $dQ=e^{\alpha g}dP/E_P[e^{\alpha g}]$ with $\alpha>0$ and $E_P[e^{\alpha g}]<\infty$ then equality holds in \req{eq:UQ1_upper} and the minimum is at $c=1+\alpha$.
\end{proposition}
Here the QoI is $\tau=e^g$ and  $R_\alpha(Q\|P)$ denotes the  R{\'e}nyi divergences \cite{renyi1961measures}, which are obtained from the MGF of the log-likelihood (when $Q\ll P$ and $\alpha>1$):
 \begin{align}\label{eq:renyi_formula}
R_\alpha(Q\|P)=       \frac{1}{\alpha(\alpha-1)}\log\left(\int e^{(\alpha-1)\log(dQ/dP)} dQ\right).
\end{align}

Proposition \ref{prop:UQ1}  is one of the key mathematical tools that we use and build upon in this paper, and when combined with \req{eq:renyi_formula}, it  illustrates a central message of this work: bounding risk-sensitive QoIs requires control on the MGF of the log-likelihood (i.e., control of all moments).  This contrasts with methods for non-risk-sensitive quantities, which can be expressed in terms of the relative entropy (i.e., the mean of the log-likelihood):
\begin{align}
R(Q\|P)=E_Q[\log(dQ/dP)]
\end{align}
(see \cite{dupuis2011uq,dupuis2015path,BLM,BREUER20131552,glasserman2014} and Proposition \ref{prop:info_ineq0} below). As we will show, combining  \req{eq:UQ1_upper}  with the methods for constructing and characterizing ambiguity sets are developed in   Section \ref{sec:tail_behavior}  results in a powerful toolset for  addressing the robustness problem (\ref{eq:general_goal}).

\subsection{Outline of Our Results}
We build upon the work in  \cite{atar2015robust} in two directions:
\begin{enumerate}
\item Our first innovation in this paper is a comprehensive study of the structure of R{\'e}nyi-divergence  ambiguity sets, the means by which they can be constructed, and their connection to the tail-behavior of $Q$, as compared to $P$. The general definition of R{\'e}yi-divergence ambiguity sets is found in  Section \ref{sec:ambiguity_sets} (see Definition \ref{def:U_Lambda}); that section begins with relevant background on  R{\'e}nyi divergences.  Section \ref{sec:tail_behavior} develops our general method for constructing the ambiguity set that captures a specified form of tail-behavior; see Definition \ref{def:tail_decay_ambiguity_set}.  Lemmas \ref{Lemma:KL_bound} and \ref{Lemma:U_inclusion} and Theorems \ref{Thm:construct_Q} and \ref{thm:tail_decay_ambiguity_set}   then provide further properties of these ambiguity sets and their members, including an algorithm for constructing a representative member. Examples that illustrate these constructions and theorems can be found in Section \ref{sec:U_examples}.

\item Secondly, we  prove several new theorems regarding  the connection  between R{\'e}nyi divergences and  robustness bounds for  risk-sensitive QoIs.  Our primary UQ bound, generalizing Proposition \ref{prop:UQ1}, is found in Theorem \ref{thm:main_UQ_result}. We also prove a generalization of the Donsker-Varadhan variational formula to R{\'e}nyi divergences (Theorem \ref{thm:gen_DV}), as well as  several  new divergence (Theorem \ref{thm:div_property}) and tightness properties (Theorems \ref{thm:tightness1} and \ref{thm:tightness2}).
\end{enumerate}

A number of examples and applications are treated in Section \ref{sec:applications}, illustrating how one can employ our results:
\begin{enumerate}
\item First, we consider robustness of a data-driven model of lithium-ion battery failure, obtained from the data in \cite{Severson_battery}.  Specifically, we  use the method of constructing ambiguity sets from Section \ref{sec:tail_behavior} together with our primary UQ bound, Theorem \ref{thm:main_UQ_result}, to study  robustness of the failure-rate  (see Section \ref{sec:Battery}).
\item The connection between our results and large deviations theory was motivated in Remark \ref{remark:LDP}.   In the second example, we study distributional robustness for large deviations rate functions. Specifically, we study the effect of adding some `roughness' to a model $P$, i.e., we consider alternative models of the form $dQ=Z^{-1}(1+\phi(x))P(dx)$ for an appropriate class of $\phi$'s. Our methods provides a bound on the rate function for IID averages from $Q$ in terms of that of $P$. We also use this example to illustrate the Bennett-ambiguity-sets,  an instance of the  classical  MGF bounds found at the end of Section \ref{sec:MGF_ambiguity} (see Section \ref{sec:rate_func}).
\item As a final example, we   show how   the tightness result, Theorem \ref{thm:tightness1}, can be used to solve certain distributionally robust optimization problems.  Specifically, we consider a distributionally robust optimal stopping problem for  option pricing (see Section \ref{ex:DRO}).
\end{enumerate}

\section{ Ambiguity Sets Based on R\'{e}nyi Divergences}\label{sec:ambiguity_sets}
Given a baseline model $P$, thought of as an approximate but tractable model, an ambiguity set for $P$ is simply a collection of alternative models $\mathcal{U}(P)$ containing $P$; we think of  $\mathcal{U}(P)$ as a `neighborhood' of models that contains the `true' model of the system.  As discussed above, different ambiguity sets capture different types of (possibly non-parametric) model uncertainty; i.e., different types of perturbations of $P$.    As we will see, ambiguity sets defined in terms of R{\'e}nyi divergences are appropriate for capturing uncertainty in the tail behavior of a distribution, thus making them useful tools for deriving distributional robustness bounds on rare events (and other risk-sensitive QoIs); the discussion of UQ will begin in Section \ref{sec:Renyi_UQ}.

In Section \ref{sec:background} we first provide some standard background on the R{\'e}nyi family of divergences.  Then, in Section \ref{sec:MGF_ambiguity},  we develop the connection between R{\'e}nyi divergences and the MGF of the log-likelihood, and use that to define a general notion of R{\'e}nyi-divergence ambiguity set.  We give several  examples stemming from classical MGF bounds in Section \ref{sec:classical_MGF_ex}, but also indicate their severe limitations.  This will serve to motivate the general method of constructing R{\'e}nyi-divergence ambiguity sets that we develop in  Section \ref{sec:tail_behavior}.

\subsection{Background on R\'{e}nyi Divergences}\label{sec:background}

The family of R\'{e}nyi divergences was firstly introduced in \cite{renyi1961measures} by R\'{e}nyi and provide a means of quantifying the discrepancy between two probability measures.  They can be defined in terms of a related family of $f$-divergences as follows \cite{1705001}:
\begin{definition}
Let $P,Q$ be  probability measures on  a measurable space $(\Omega,\mathcal{M})$.  The  Reyni divergence of order $\alpha\in(0,1)\cup(1,\infty)$ of $Q$ with respect to  $P$ is defined by
\begin{align}
R_\alpha(Q\|P)\equiv\frac{1}{\alpha(\alpha-1)}\log\left(\alpha(\alpha-1)D_{f_\alpha}(Q,P)+1\right),
\end{align}
where $D_f$ denotes an $f$-divergence and 
\begin{align}
f_\alpha(t)=      \frac{t^\alpha-1}{\alpha(\alpha-1)}.
\end{align}
\end{definition}
These satisfy the divergence property:
\begin{align}\label{eq:Renyi_div_property}
R_\alpha(Q\|P)\geq 0\text{ and } R_\alpha(Q\|P)=0 \text{ iff }Q=P.
\end{align}
and also
\begin{align}\label{eq:Renyi_nondec}
\alpha\to \alpha R_\alpha(Q\|P) \text{ is non-decreasing.} 
\end{align}

More explicitly, $R_\alpha(Q\|P)$ can be computed as follows: If $\nu$ is a sigma-finite positive measure with $dP=pd\nu$ and $dQ=qd\nu$ then
\begin{align}\label{eq:Renyi_formula}
R_\alpha(Q\|P)=\begin{cases} 
     \frac{1}{\alpha(\alpha-1)}\log\left(\int_{p>0} q^\alpha p^{1-\alpha} d\nu\right) & \text{ if }0<\alpha<1 \text{ or } (\alpha>1 \text{ and } Q\ll P)\\
       +\infty &\text{ if }\alpha>1\text{ and } Q\not\ll P.
         \end{cases}
\end{align}
One also has the property
\begin{align}\label{eq:R_neg_alpha}
R_\alpha(Q\|P)=R_{1-\alpha}(P\|Q),\,\,\,\alpha\in(0,1)
\end{align}
and we use this relation is used to extend the definition of $R_\alpha(Q\|P)$ to $\alpha<0$.

For certain values of $\alpha$, the R{\'e}nyi divergence is related to other commonly used divergences:
\renewcommand{\multirowsetup}{\centering}
\begin{table}[H]
  \caption{Special cases in the R\'{e}nyi divergence family }
  \centering
  \begin{tabular}{c c c}
    \hline\hline                   
    $R_\alpha(Q\|P)$ & Equivalent Formula & Note \\[0.5ex]  \hline
     $\alpha \nearrow 1$ & $R(Q\|P)$  & $R(Q\|P)=\begin{cases} 
    \int \log(dQ/dP) dQ  &\text{ if } Q\ll P\\
    +\infty & \text{ if }Q\not\ll P
         \end{cases}$\\ [1ex]
          $\alpha \searrow 1$ & $R(Q\|P)$  & if  $R(Q\|P)=\infty$ or  $R_\beta(Q\|P)<\infty$ for some $\beta>1$ \\ [1ex]           
          $\alpha \searrow 0$ & $R(P\|Q)$  & \\ [1ex]
     $\alpha=\frac{1}{2}$ & $-4\log(1-\frac{Hel^2(Q\|P)}{2})$  &$Hel^2(Q\|P)=\int(\sqrt{q}-\sqrt{p})^2d\mu$ \\ [1ex] 
     $\alpha=2$ & $\frac{1}{2}\log(1+\chi^2(Q\|P))$  & $\chi^2(Q\|P)=\begin{cases} 
     \int (dQ/dP-1)^2dP  &\text{ if } Q\ll P\\
    +\infty & \text{ if }Q\not\ll P
         \end{cases}$ \\ [1ex] 
     $\alpha \to +\infty $ & $\alpha R_\alpha(Q\|P) \to D_\infty(Q\|P)$  &   
$ D_\infty(Q\|P)=\begin{cases} 
     \log \left(\esssup_Q\frac{dQ}{dP}\right)  &\text{ if } Q\ll P\\
    +\infty & \text{ if }Q\not\ll P
         \end{cases}$     
       \\ [1ex] 
    \hline
    \hline
    
  \end{tabular}
  \label{Table:Renyi}
\end{table}
\noindent Here, $R(\cdot\|\cdot)$ denotes  the relative entropy (i.e., Kullback-Leibler divergence) and $D_\infty(\cdot\|\cdot)$ is called the worst case regret. The above limiting results motivate the definitions $R_1(Q\|P)=R(Q\|P)$ and $R_0(Q\|P)=R(P\|Q)$.

For proofs of the properties listed in this section, as well as many more, see \cite{van2014renyi1}. Note, however, that our definition of the R{\'e}nyi divergences is related to theirs by $D_\alpha(\cdot\|\cdot)=\alpha R_\alpha(\cdot\|\cdot)$.  Explicit formulas for the R{\'e}nyi divergence between members of many common parametric families can be found in \cite{GIL2013124}.

\subsection{R{\'e}nyi-Divergence Ambiguity Sets and MGF Bounds}\label{sec:MGF_ambiguity}

The robustness bounds in Proposition \ref{prop:UQ1} constrain risk-sensitive QoIs under $Q$ in terms of risk-sensitive QoIs under the baseline model,  $P$, along with the R{\'e}nyi divergence between $Q$ and $P$.  Each possible choice of model neighborhood (i.e., ambiguity set), $Q\in\mathcal{U}(P)$, encapsulates a certain level and form of uncertainty regarding the baseline model; specifically, we will see that ambiguity sets are closely related to the tail behavior of the log-likelihood, $\log(dQ/dP)$. As a first step,   the main goals of this subsection are the general definitions, \req{eq:h_def} and \req{eq:ambiguity_Lambda_Def}, of R{\'e}nyi-ambiguity sets, along with the intuition behind them.

The alternative model, $Q$, appears in the UQ upper bound \req{eq:UQ1_upper} only through $R_\alpha(Q\|P)$; this suggests that the natural notion of R{\'e}nyi-divergence ambiguity set is 
\begin{equation}\label{eq:h_def}
\mathcal{U}(P)=\{Q:R_\alpha(Q\|P)\leq h(\alpha), \,\,\,\alpha\in I\},
\end{equation}
 for some choice of function $h\geq 0$ and some  range of $\alpha\in I$. This is in contrast to the non-risk-sensitive UQ bound, \req{goal_oriented_bound}, for which the natural ambiguity sets has a  much simpler form, fixed by a choice of constant $\eta\geq 0$:  
 \begin{equation}\label{eq:KL_U}
 \mathcal{U}_{KL}^\eta(P)=\{Q:R(Q||P)\leq \eta\}.
 \end{equation}
This suggest the question:  how does the  choice of $h$ in \req{eq:h_def} relate to  properties of $Q\in\mathcal{U}$?

\begin{remark}
Ambiguity sets appropriate for the lower UQ bound, \req{eq:UQ1_upper}, can similarly be defined  by imposing an upper bound on $R_\alpha(P\|Q)$.  We focus on the  case of upper UQ  bounds, for which the definition (\ref{eq:h_def}) is appropriate.

Also, from \req{eq:Renyi_nondec} we see that even if $R_{\alpha}(Q\|P)$ is only explicitly constrained at only a single point (i.e., $I=\{\alpha_0\}$ in \req{eq:h_def}), this will generally imply a constraint on the R{\'e}nyi divergences over some range of $\alpha$'s.
\end{remark}

  The  crucial observation is that, by \req{eq:Renyi_formula},  for $\alpha>1$ and $Q\ll P$ the R{\'e}nyi divergence is related to the cumulant generating function of the log-likelihood:
\begin{align}\label{eq:Renyi_MGF}
R_\alpha(Q\|P)=       \frac{1}{\alpha(\alpha-1)}\log\left(\int e^{(\alpha-1)\log(dQ/dP)} dQ\right)= \frac{1}{\alpha(\alpha-1)}\Lambda_Q^{\log(dQ/dP)}(\alpha-1),
\end{align}
where we use the notation
\begin{align}
\Lambda_Q^f(\lambda)=\log E_Q[e^{\lambda f}]
\end{align}
for the cumulant generating function of $f:\Omega\to\overline{\mathbb{R}}$.  Intuitively, \req{eq:Renyi_MGF} implies that to control risk-sensitive QoIs, one needs control on all $Q$-moments of the log-likelihood.  This is in contrast to non-risk-sensitive QoIs (see \req {goal_oriented_bound}), where one only needs control on the first $Q$-moment: $E_Q[\log(dQ/dP)]=R(Q\|P)$.

Based on this observation, we are motivated to study the bounds on  R{\'e}nyi-divergences that arise from (one-sided) MGF bounds (which are commonly used in deriving concentration inequalities \cite{BLM}):
\begin{definition}\label{def:U_Lambda}
Let $\Lambda:[0,\infty)\to [0,\infty]$  satisfy $\Lambda(0)=0$ (a property held by any CGF) and be finite on a neighborhood of $0$.  We define the corresponding R{\'e}nyi-divergence ambiguity set by
\begin{align}\label{eq:ambiguity_Lambda_Def}
\mathcal{U}^\Lambda(P)\equiv&\left\{Q:R_\alpha(Q\|P)\leq \frac{1}{\alpha(\alpha-1)}\Lambda(\alpha-1),\,\,\,\alpha>1\right\}\\
=&\left\{Q:Q\ll P,\,\,\, \Lambda_Q^{\log(dQ/dP)}(\lambda)\leq \Lambda(\lambda),\,\,\,\lambda>0\right\}.\notag
\end{align}
If $\Lambda$ is only defined on some subset $I$ with $[0,\beta)\subset I\subset [0,\infty)$ then we simply extend $\Lambda$ to be  $+\infty$ on $[\beta,\infty)$.
\end{definition}

The following lemma shows that the $\mathcal{U}^\Lambda(P)$ represent more refined stress tests than the  non-risk-sensitive (i.e., relative entropy) ambiguity sets, \req{eq:KL_U}:
\begin{lemma}\label{lemma:U_def}
Let $\Lambda:[0,\infty)\to [0,\infty]$  satisfy $\Lambda(0)=0$, be finite on a neighborhood of $0$, and be differentiable from the right at $0$ with derivative $\eta$.  Then 
\begin{align}
\mathcal{U}^\Lambda(P)\subset\{Q:R(Q\|P)\leq \eta\}.
\end{align}
\end{lemma}
\begin{proof}
For $Q\in\mathcal{U}^\Lambda(P)$ we have
\begin{align}
R(Q\|P)=\lim_{\alpha\searrow 1} R_\alpha(Q\|P)\leq \lim_{\alpha\searrow 1}\frac{1}{\alpha(\alpha-1)}\Lambda(\alpha-1)=\eta.
\end{align}
 To obtain the first equality, use the second line of Table \ref{Table:Renyi}.
\end{proof}

\begin{remark}\label{remark:U_def}
One could instead write  \req{eq:Renyi_MGF} as $R_\alpha(Q\|P)=\frac{1}{\alpha(\alpha-1)}\Lambda_P^{\log(dQ/dP)}(\alpha)$ and attempt to define ambiguity sets by bounding $\Lambda_P$ by some other  CGF $\Lambda$, i.e., define
\begin{align}\label{eq:bad_ambiguity_def}
\widetilde{\mathcal{U}}^\Lambda=\left\{Q:R_\alpha(Q\|P)\leq \frac{1}{\alpha(\alpha-1)}\Lambda(\alpha),\,\,\alpha>1\right\}.
\end{align}
However, a general CGF will not vanish at $1$ and so  the upper bound will  generally diverge as $\alpha\searrow 1$, making \req{eq:bad_ambiguity_def} a poor way to define R{\'e}nyi-divergence ambiguity sets in general.  This is also why we require $\Lambda(0)=0$ in Definition \ref{def:U_Lambda}, a property that is satisfied by every CGF.
\end{remark}

\subsection{Examples of Ambiguity sets for Bounded Perturbations}\label{sec:classical_MGF_ex}
Several classical MGF bounds, which are commonly used in the derivation of concentration inequalities (again, see \cite{BLM}), can be used to define ambiguity sets via \req{eq:ambiguity_Lambda_Def}:
 \begin{enumerate}
\item Bennett-$(a,b)$ bound: If $a\leq \log(dQ/dP)\leq b$ (which implies $a\leq R(Q\|P)\leq b$) and $R(Q\|P)\leq \eta\leq b$ then  
\begin{align}
\Lambda_Q^{\log(dQ/dP)}(\lambda)\leq \lambda b+\log\left( \frac{b-\eta}{b-a}e^{-\lambda (b-a)}+\frac{\eta-a}{b-a}\right)\equiv \Lambda^{B(a,b)}_\eta(\lambda), \,\,\,\lambda>0.
\end{align}
This follows from Lemma 2.4.1 in  \cite{dembo2009large}.
\item Bennett bound: If $\log(dQ/dP)\in L^2(Q)$ with $\log(dQ/dP)\leq b$, $R(Q\|P)\leq \eta\leq b$, and $\Var_Q[\log(dQ/dP)]\leq \sigma^2$ then
\begin{align}\label{eq:bennett_bound}
\Lambda^{\log(dQ/dP)}_Q(\lambda)\leq \lambda b+\log\left( \frac{(b-\eta)^2}{(b-\eta)^2+\sigma^2}e^{-\lambda(\sigma^2/(b-\eta)+(b-\eta))}+\frac{\sigma^2}{(b-\eta)^2+\sigma^2}\right)\equiv \Lambda^{B}_{b,\eta(\lambda),\sigma}(\lambda),\,\,\,\lambda>0.
\end{align}
This follows from Corollary 2.4.5 in \cite{dembo2009large}.
\item Hoeffding bound: If  $a\leq \log(dQ/dP)\leq b$ and $R(Q\|P)\leq \eta\leq b$ then  
\begin{align}\label{eq:Hoeff_bound}
\Lambda^{\log(dQ/dP)}_Q(\lambda)\leq \eta \lambda+\frac{\sigma^2}{2}\lambda^2\equiv \Lambda^{SG}_{\eta,\sigma}(\lambda),\,\,\,\lambda> 0,
\end{align}
 where  $\sigma^2=(b-a)^2/4$ (the maximum possible variance of a random-variable bounded between $a$ and $b$). This follows from   Hoeffding's inequality (see Theorem 2.8 in \cite{BLM}).
\begin{remark}
This is a special case of a sub-Gaussian bound, hence the superscript $SG$; see Corollary \ref{cor:tail_behavior} below.
\end{remark}
\end{enumerate}
\ref{app:classical_CGF} contains further information on the relationships between these classical cases. 

To use any the Bennett or Hoeffding   bounds one must, at minimum, have an upper bound: $\log(dQ/dP)\leq b$, i.e., if $p(x)$ and $q(x)$ are densities for $P$ and $Q$ respectively, then one must have $q(x)\leq e^b p(x)$.   This is an untenable restriction in many cases, as it implies that $q$ decays at least as fast as $p$. For example, in analyzing the  battery-failure model introduced  in Section \ref{sec:motivation}, the  baseline model  that we fit to the data  (see Section \ref{sec:Battery}) will have power-law decay of the form $T^{\hat \beta}$ at $0$.  One would like to stress test $P$ under both faster ($\beta>\hat \beta$) and {\em slower} ($\beta<\hat \beta$) decay at $0$; the latter is not possible via the above classical bounds.

 Based on this, we are motivated to develop a significantly more general method of constructing  $\Lambda$'s  that can be used to define ambiguity sets.  This construction can be viewed as a stress-test design problem; one designs a $\Lambda$ so that $\mathcal{U}^\Lambda$ captures all $Q$'s that are tail-perturbations of $P$ of a desired form.   This construction has the added benefit of being more intuitively meaningful than either  $h$ in \req{eq:h_def} or $\Lambda$ in \req{eq:ambiguity_Lambda_Def}.

 \section{Ambiguity Sets and Tail Behavior: A Stress-Test Design Problem}\label{sec:tail_behavior}

The goal of this section is to construct the $\Lambda$ that captures tail perturbations of a desired type, i.e., that characterizes a desired stress-test scenario. The corresponding ambiguity set can then be defined via \req{eq:ambiguity_Lambda_Def}.  First we not that the reverse problem has a straightforward  solution; each choice of $\Lambda$ in \req{eq:ambiguity_Lambda_Def}  implies a particular relationship between the tail behavior of $Q\in\mathcal{U}^{\Lambda}(P)$, as compared to $P$.  This can be made concrete  via a Chernoff bound:
\begin{lemma}\label{lemma:Chernoff}
If $Q\in\mathcal{U}^\Lambda(P)$ then for any $r\in\mathbb{R}$ we have
\begin{align}
P(\log(dQ/dP)\geq r) \leq \exp\left(-\sup_{\lambda>0}\left\{(\lambda+1) r-\Lambda(\lambda)\right\}\right).
\end{align}
\end{lemma}
\begin{proof}
For $\lambda>0$, a Chernoff bound gives
\begin{align}
P(\log(dQ/dP)\geq r) \leq E_P\left[e^{(\lambda+1)\log(dQ/dP)}\right]e^{-(\lambda+1)r}=\exp\left(\Lambda_Q^{\log(dQ/dP)}(\lambda)-(\lambda+1)r\right).
\end{align}
Bounding  $\Lambda_Q^{\log(dQ/dP)}(\lambda)\leq \Lambda(\lambda)$ and then taking the infimum over $\lambda>0$ gives the result.
\end{proof}

The above lemma yields a bound on the tail behavior, given a specified $\Lambda$.  The starting point for the reverse process, that of constructing a $\Lambda$ that captures a desired tail behavior, is the following expression for the CGF of the log-likelihood:
\begin{lemma}\label{lemma:tail}
Let $Q\ll P$. For $r\geq 0$ define $G(r)=P(dQ/dP\geq r)$. Then
\begin{align}\label{eq:G_MGF_bound}
\Lambda_Q^{\log(dQ/dP)}(\lambda)= \log\left((\lambda+1)\int_{0}^\infty G(z)z^\lambda dz\right),\,\,\,\lambda\geq 0.
\end{align}
%actually, holds for $\lambda>-1$

$G$ also has the following properties:
\begin{align}\label{eq:G_properties}
\text{ $G(r)$ is non-increasing, \,left continuous, \,$G(0)=1$,\, $G(r)\to 0$ \,as\, $r\to\infty$, \,and \,$\int_0^\infty G(z)dz=1$.}
\end{align}
Moreover, one can write $G(r)=\mu([r,\infty))$ where $\mu$ is the distribution of $dQ/dP$ under $P$; $\mu$ is a probability measure on $[0,\infty)$ with mean  equal to $1$. 
\end{lemma}
\begin{remark}
Many of the above properties are  satisfied by $P(Y\geq r)$ for any random variable $Y$.  What is special here is that $Y\equiv dQ/dP$ is  non-negative and integrates to $1$ under $P$, i.e. $YdP$ is a probability measure.  These lead to the key properties we will need going forward, namely  $\int_0^\infty G(z)dz=1$ and $\mu$ has mean $1$.
\end{remark}
\begin{proof}
Using Fubini's theorem, we have
\begin{align}
\Lambda_Q^{\log(dQ/dP)}(\lambda)=\log\left(\int (dQ/dP)^{\lambda+1}dP\right)=\log\left(\int_0^\infty P\left(dQ/dP\geq s^{1/(\lambda+1)}\right)ds\right).
\end{align}
Now change variables to $z=s^{1/(\lambda+1)}$.  This proves \req{eq:G_MGF_bound}. Fubini's theorem similarly gives
\begin{align}
\int_0^\infty G(z)dz=\int \int_0^\infty 1_{dQ/dP\geq z} dzdP=\int \frac{dQ}{dP}dP=1.
\end{align}
 The remaining properties are straightforward to verify.
\end{proof}
\begin{remark}
   Lemma \ref{lemma:tail} combined with   \req{eq:renyi_formula} shows that R{\'e}nyi divergences only have access to information on $P$ and $Q$ through $P(dQ/dP\geq r)$, $r\geq 0$. If, for instance, $dQ=qdx$ and $dP=pdx$ on $\mathbb{R}^n$ then
\begin{align}
P(dQ/dP\geq r)=\int_{q\geq rp} p(x)dx.
\end{align}
As $r$ increases, the decay of this quantity describes how the  large relative-perturbations of $q$ compared to $p$ become concentrated into the tail of $P$  (i.e., where $p(x)$ is small). Intuitively, this is why control on $P(dQ/dP\geq r)$ provides one with control on rare-events, and other risk-sensitive QoIs.
\end{remark}

Lemma \ref{lemma:tail} suggests that one construct  R{\'e}nyi-divergence ambiguity sets  by specifying $P(dQ/dP\geq r)$, thus motivating the following definition:
\begin{definition}\label{def:tail_decay_ambiguity_set}
Let  $\mu$ be a probability measure on $[0,\infty)$ with mean $1$ and  define $G_\mu(r)=\mu([r,\infty))$, $r\geq 0$ and 
\begin{align}\label{eq:Lambda_mu_def}
\Lambda_\mu(\lambda)\equiv\log\left((\lambda+1)\int_{0}^\infty G_\mu(z)z^\lambda dz\right),\,\,\,\lambda\geq 0.
\end{align}
Note that  $G_\mu$ satisfies the properties in \req{eq:G_properties} and $\Lambda_\mu(0)=0$.  Moreover,  letting  $\nu(dz)=z\mu(dz)$ we have
\begin{align}%actually, holds for \lambda>-1.
\Lambda_\mu(\lambda)=\Lambda_\nu^{\log(d\nu/d\mu)}(\lambda)=\lambda(1+\lambda)R_{1+\lambda}(\nu\|\mu)\geq 0,\,\,\,\,\lambda>0.
\end{align}

If $\Lambda_\mu$ is also  finite in a neighborhood of $0$ then for any $P$ we can  define the corresponding ambiguity set, as in \req{eq:ambiguity_Lambda_Def}:
\begin{align}\label{eq:U_mu_def}
\mathcal{U}^\mu(P)\equiv\mathcal{U}^{\Lambda_\mu}(P)=\left\{Q:Q\ll P \text{ and } \Lambda_Q^{\log(dQ/dP)}(\lambda)\leq \Lambda_\mu(\lambda),\,\,\,\lambda>0\right\}.
\end{align}

\end{definition}
%$\Lambda_\mu$ is automatically finite on $(-1,0]$

The following lemma  shows that $\mathcal{U}^{\mu}(P)$ at least captures our stated motivation: it contains all alternative models for which $ P(dQ/dP\geq r)=G_\mu(r)$, $r\geq 0$. This lemma also shows that the  ambiguity sets  (\ref{eq:U_mu_def}) also imply a bound on the relative entropy:
\begin{lemma}\label{Lemma:KL_bound}
Let  $\mu$  be as in Definition \ref{def:tail_decay_ambiguity_set} (in particular, we assume $\Lambda_\mu$ is finite on a neighborhood of $0$).   Then $\mathcal{U}^{\mu}(P)$ contains all $Q$ for which $Q\ll P$ and $dQ/dP$ is distributed as $\mu$ under $P$.  Note that such $Q$'s will satisfy
\begin{align}
 &\Lambda_Q^{\log(dQ/dP)}(\lambda)=\Lambda_\mu(\lambda)\,\, \text{ for all }\,\,\lambda\geq0,\label{eq:saturate_bound}\\ %actually, holds for \lambda>-1.
&R(Q\|P)=  \frac{d}{d\lambda}\bigg|_{\lambda=0}\Lambda_\mu(\lambda)= 1+\int_{0}^\infty G_\mu(z)\log(z) dz.\label{eq:KL_equality}\notag
\end{align}
Any alternative model $Q$ satisfying \req{eq:saturate_bound} will be said to saturate the MGF bound that defines  $\mathcal{U}^{\mu}(P)$.

More generally,  
\begin{align}\label{eq:tail_set_rel_ent}
R(Q\|P)\leq  \frac{d}{d\lambda}\bigg|_{\lambda=0}\Lambda_\mu(\lambda)= 1+\int_{0}^\infty G_\mu(z)\log(z) dz,\,\,\,\,\,Q\in\mathcal{U}^{\mu}(P).
\end{align}
\end{lemma}
\begin{proof}
 From Lemma \ref{lemma:tail}, we see that   $dQ/dP \overset{P}{\sim} \mu$ implies $\Lambda_Q^{\log(dQ/dP)}=\Lambda_\mu$ and hence $Q\in\mathcal{U}^{\mu}(P)$ (see the definition  in \req{eq:U_mu_def}).

The dominated convergence theorem implies that $\lambda\to \int_{0}^\infty G_\mu(z)z^\lambda dz$ is $C^1$ on a neighborhood of $0$ and can be differentiated under the integral (here we use the assumption that $\Lambda_\mu$ is finite on a neighborhood of $0$). The bound on the relative entropy then follows from Lemma \ref{lemma:U_def}, including the case of equality (\ref{eq:KL_equality}).
\end{proof}

\begin{remark}
Attempting to generalize \req{eq:Lambda_mu_def} to 
\begin{align}\label{eq:Lambda_G_def}
\Lambda_G(\lambda)\equiv\log\left((\lambda+1)\int_{0}^\infty G(z)z^\lambda dz\right),\,\,\,\lambda\geq 0,
\end{align}%actually, holds for \lambda>-1.
where $G(r)\geq\mu([r,\infty))$, and then replacing $\Lambda_\mu$ with $\Lambda_G$ in \req{eq:U_mu_def}  does not produce any  useful new ambiguity sets. As discussed in Remark \ref{remark:U_def}, a `useful' definition requires $\Lambda_G(0)=0$. From \req{eq:Lambda_G_def}, we see that this implies  $\int_0^\infty G(z)dz=1$.  Combine this fact with $G(r)\geq\mu([r,\infty))$  and $\int_0^\infty \mu([z,\infty))dz=1$ and we see that we must have $G(r)=\mu([r,\infty))$ for a.e. $r$. Weakening the equality to a.e. equality does not modify the definition of $\Lambda_\mu$ or of $\mathcal{U}^\mu(P)$, and is therefore irrelevant.
\end{remark}

The following theorem shows how to construct an element of $\mathcal{U}^\mu (P)$ that achieves the defining bound in \req{eq:U_mu_def}, i.e., that satisfies $\Lambda_Q^{\log(dQ/dP)}(\lambda)=\Lambda_\mu(\lambda)$ for all $\lambda>0$:
\begin{theorem}\label{Thm:construct_Q}
Let $\mu$, $G_\mu$,  and $\Lambda_\mu$ be as in Definition \ref{def:tail_decay_ambiguity_set}.   Let $P$ be a probability measure and $\nu$ be a sigma-finite positive measure such that $dP=pd\nu$. Define $\psi:[0,\infty)\to[0,1]$ by $\psi(y)=P(p\leq y)$.   Assume that 
\begin{align}\label{eq:G_psi_range_condition}
\range( G_\mu)\subset \range(\psi)\cup\{1\}.
\end{align}

Define $H_\mu:[0,1]\to[0,\infty]$ by
\begin{align}\label{eq:H_def}
H_\mu(\rho)=\sup\{r\in[0,\infty):G_\mu(r)\geq \rho\}.
\end{align}
 and $\phi:[0,\infty)\to[0,\infty)$  by
\begin{align}
\phi(y)=H_\mu(\psi(y))1_{\psi(y)>0}.
\end{align}
Then $Q(dx)\equiv \phi(p(x)) P(dx)$ is a probability measure, $dQ/dP\overset{P}{\sim} \mu$, and  $Q\in\mathcal{U}^\mu(P)$.
\end{theorem}
\begin{proof}
Note that $G_\mu(0)=1$, $G_\mu$ is left continuous, and non-increasing, and $G_\mu(r)\to0$ as $r\to\infty$.  Hence $r_0=H_\mu(1)\in[0,\infty)$ is the unique real number satisfying $G_\mu|_{[0,r_0]}=1$, $G_\mu|_{(r_0,\infty)}<1$ . We also have $H_\mu(\rho)\in [r_0,\infty)$ for $\rho\in(0,1]$.  In particular, we see that 
\begin{align}\label{eq:phi_range}
\phi(y)\in[r_0,\infty)\,\text{ if }\, \psi(y)\neq 0.
\end{align}

Next, we need to show that $P(\psi(p)=0)=0$:
\begin{align}
P(\psi(p)=0)=\int 1_{\psi(y)=0} (p_*P)(dy),
\end{align}
where $p_*P$ is the distribution of $p$ under $P$.  Let $y_0=\sup\{y:\psi(y)=0\}$.  $\psi$ is non-decreasing so $\{\psi=0\}=[0,y_0)$ or $\{\psi=0\}=[0,y_0]$.  In the latter case, we have
\begin{align}
P(\psi(p)=0)=P(p\leq y_0)=\psi(y_0)=0.
\end{align}
In the former case, we have $y_0>0$ and
\begin{align}
P(\psi(p)=0)=P(p<y_0)=\lim_{n\to\infty}P(p\leq y_0-1/n)=\lim_{n\to\infty}\psi(y_0-1/n)=0.
\end{align}
Therefore we have proven that  $P(\psi(p)=0)=0$.
%psi is measurable, as is $H_\mu$ since it is non-increasing, so phi is measurable

We now show that $dQ/dP\overset{P}{\sim} \mu$: For $r\leq r_0$ we have  
\begin{align}
P(dQ/dP\geq r)=P(\phi(p)\geq r, \psi(p)> 0)=P(\psi(p)> 0)=1=G_\mu(r),
\end{align}
(here we used \req{eq:phi_range}).

 For $r>r_0$, we write
\begin{align}
P(dQ/dP\geq r)=P\left(H_\mu(\psi(p))\geq r,\psi(p)>0\right).
\end{align}
$G_\mu$ and $H_\mu$ are not inverses, but they still satisfy the following key property:
\begin{align}\label{eq:H_inequality}
H_\mu(\rho)\geq r \,\text{ iff }\, \rho\leq G_\mu(r),\,\,\,\rho\in[0,1],r\in[0,\infty).
\end{align}
The follows from the fact that $H_\mu$ equals the negative of the quantile function of $\widetilde \mu$, where $\widetilde \mu$ is the distribution of the function $r\to-r$ under $\mu$. \req{eq:H_inequality} then follows from the related property of quantile functions.

Using \req{eq:H_inequality}, we obtain
\begin{align}
P(dQ/dP\geq r)=P\left(\psi(p)\leq G_\mu(r),\psi(p)>0\right) =P\left(\psi(p)\leq G_\mu(r)\right)=\int1_{\psi(y)\leq G_\mu(r)} (p_*P)(dy).
\end{align}
$G_\mu(r)<1$, hence  the assumption \req{eq:G_psi_range_condition} implies that $G_\mu(r)\in\range(\psi)$.  Define $y_r=\sup\{y:\psi(y)\leq G_\mu(r)\}$. Similarly to the above, either  $\{\psi(y)\leq G_\mu(r)\}=[0,y_r)$ or $\{\psi(y)\leq G_\mu(r)\}=[0,y_r]$. 

 We know that $G_\mu(r)\in\range(\psi)$, hence one of the following two cases must hold:
\begin{enumerate}
\item  $\psi(y_r)=G_\mu(r)$:  In this case, $\{\psi(y)\leq G_\mu(r)\}=[0,y_r]$ and hence
\begin{align}
P(dQ/dP\geq r)=\int1_{[0,y_r]} (p_*P)(dy)=P(p\leq y_r)=\psi(y_r)=G_\mu(r).
\end{align}
\item $\psi(y_r)\neq G_\mu(r)$ and there exists $y^*<y_r$ with $\psi(y^*)=G_\mu(r)$: $\psi$ is non-decreasing so $\psi(y)=G_\mu(r)$ for all $y\in[y^*,y_r)$ and  $\{\psi(y)\leq G_\mu(r)\}=[0,y_r)$ .  Therefore, we can compute
\begin{align}
P(dQ/dP\geq r)=&\int1_{[0,y_r)} (p_*P)(dy)=P(p<y_r)=\lim_{n\to\infty}P\left(p\leq n^{-1}y^*+(1-n^{-1})y_r\right)\\
=&\lim_{n\to\infty}\psi(n^{-1}y^*+(1-n^{-1})y_r)=G_\mu(r),\notag
\end{align}
since $\psi(n^{-1}y^*+(1-n^{-1})y_r)=G_\mu(r)$ for all $n\in \mathbb{Z}^+$.
\end{enumerate}
Therefore we have proven that $P(dQ/dP\geq r)=G_\mu(r)$ for all $r\geq 0$, and hence   $dQ/dP\overset{P}{\sim} \mu$.   Since $\mu$ has mean $1$, this also implies that $Q$ is a probability measure.  Property (1) of Theorem \ref{thm:tail_decay_ambiguity_set} then gives  $Q\in \mathcal{U}^\mu(P)$.
%Borel sigma algebra on [0,inf) is generated by half-infinite closed intervals
\end{proof}

The following corollary gives a simpler expression for $\phi$, under stronger assumptions:
\begin{corollary}\label{corr:construct_Q}
As in the above proof, let $r_0\equiv H_\mu(1)\in[0,\infty)$ be the unique real number satisfying $G_\mu|_{[0,r_0]}=1$, $G_\mu|_{(r_0,\infty)}<1$.  Suppose $G_\mu$ is continuous and strictly decreasing on $[r_0,\infty)$. Then $H_\mu|_{(0,1]}= (G_\mu|_{[r_0,\infty)})^{-1}$ (in such cases, we simply write $G_\mu^{-1}$ for short).

Suppose that we also have $(0,1)\subset \range(\psi)$ (for example, if $\psi$ is continuous) and define $\phi:[0,\infty)\to[0,\infty)$  by
\begin{align}
\phi(y)= G_\mu^{-1}(\psi(y))1_{\psi(y)>0}.
\end{align}
Then $Q(dx)\equiv\phi(p(x)) P(dx)$ is a probability measure, $dQ/dP\overset{P}{\sim} \mu$, and  $Q\in\mathcal{U}^\mu(P)$.
\end{corollary}

\begin{remark}
The likelihood ratio $dQ/dP$ constructed via Theorem \ref{Thm:construct_Q} or Corollary \ref{corr:construct_Q} perturbs all regions with same density, $p$, equally; it does not preferentially perturb one tail versus another. As such, the construction is most relevant when one is concerned with risk-sensitive quantities that depend on both (all of the) tail regions, or when one first conditions on only looking at a single tail of interest.
\end{remark}

In the remainder of this subsection, we prove several results that give further information on what models are contained in $\mathcal{U}^\mu(P)$.  First, we have several  inclusions between these ambiguity sets:
\begin{lemma}\label{Lemma:U_inclusion}
Let $\mu_i$, $G_{\mu_i}$, and $\Lambda_{\mu_i}$, $i=1,2$, be as in Definition \ref{def:tail_decay_ambiguity_set}.
\begin{enumerate}
  \item Suppose we have $r_0>0$ such that
\begin{align}\label{eq:U_inclusion_bound1}
G_{\mu_1}(r)\geq G_{\mu_2}(r)\,\text{ for all $r\in[0,r_0]$  and }\, G_{\mu_1}(r)\leq G_{\mu_2}(r) \,\text{ for all  $r> r_0$}.
\end{align}
Then $\mathcal{U}^{\mu_1}(P)\subset\mathcal{U}^{\mu_2}(P)$.
\item Let $I_2=\{\lambda>0:\int_0^\infty G_{\mu_2}(z)z^\lambda dz<\infty\}$. If
\begin{align}\label{eq:U_inclusion_decay}
\lim_{R\to\infty} R^\lambda\int_R^\infty G_{\mu_2}(z)dz=0\,\,\text{ for all $\lambda\in I_2$ }
\end{align}
 and
\begin{align}\label{eq:U_inclusion_bound2}
\int_0^r G_{\mu_2}(z)dz\leq \int_0^rG_{\mu_1}(z)dz\,\, \text{ for all }\,\, r\geq 0
\end{align}
then $\mathcal{U}^{\mu_1}(P)\subset\mathcal{U}^{\mu_2}(P)$.
\end{enumerate}
\end{lemma}
\begin{proof}
\begin{enumerate}
\item The first claim follows if we can show $\int_0^\infty (G_{\mu_2}(z)-G_{\mu_1}(z))(z/r_0)^\lambda dz\geq 0$ for all $\lambda\geq 0$.  This integral equals $0$ at $\lambda=0$, as both $G_{\mu_2}$ and $G_{\mu_1}$ integrate to $1$.  So we are done if we can show
\begin{align}\label{eq:inclusion1}
\lambda\to \int_0^\infty(G_{\mu_2}(z)-G_{\mu_1}(z))(z/r_0)^\lambda dz
\end{align}
is non-decreasing.  To see this, write
\begin{align}
\int_0^\infty(G_{\mu_2}(z)-G_{\mu_1}(z))z^\lambda dz=\int_{r_0}^\infty(G_{\mu_2}(z)-G_{\mu_1}(z))(z/r_0)^\lambda dz-\int_0^{r_0}(G_{\mu_1}(z)-G_{\mu_2}(z))(z/r_0)^\lambda dz.
\end{align}
\req{eq:U_inclusion_bound1} implies that both integrands are non-negative and $(z/r_0)^\lambda$ is increasing in $\lambda$ for $z> r_0$ and decreasing in $\lambda$ for $z\in(0,r_0)$.  \req{eq:inclusion1} then follows from these facts.
\item  We need to show that $\int_0^\infty G_{\mu_1}(z)z^\lambda dz\leq \int_0^\infty  G_{\mu_2}(z)z^\lambda dz$ for all $\lambda>0$.  This is trivial if $\lambda\not\in I_2$, so suppose $\lambda\in I_2$.  Integrating by parts and  using \req{eq:U_inclusion_decay} and \req{eq:U_inclusion_bound2} gives
\begin{align}
\int_0^\infty G_{\mu_2}(z)z^\lambda dz=&\lim_{R\to \infty} \left(-R^\lambda \int_R^\infty G_{\mu_2}(z)dz+\lambda\int_0^Rz^{\lambda-1}\int_z^\infty  G_{\mu_2}(r)dr dz\right)\\
=&\lim_{R\to \infty} \lambda\int_0^Rz^{\lambda-1}\left(1-\int_0^z  G_{\mu_2}(r)dr\right) dz\notag\\
\geq&\lim_{R\to \infty} \lambda\int_0^Rz^{\lambda-1}\left(1-\int_0^z  G_{\mu_1}(r)dr\right) dz\notag\\
=&\lim_{R\to \infty}\left( R^\lambda\int_R^\infty  G_{\mu_1}(z)dz+\int _0^R G_{\mu_1}(z) z^\lambda dz\right).\notag
\end{align}
\req{eq:U_inclusion_decay} together with \req{eq:U_inclusion_bound2} implies $\lim_{R\to\infty}R^\lambda\int_R^\infty  G_{\mu_1}(z)dz=0$ and so we are done.
\end{enumerate}
\end{proof}

Next, we have several useful criteria for determining if a particular $Q$ is a member of $\mathcal{U}^{\mu}(P)$.  These conditions are significantly more intuitive than either of  \req{eq:h_def} or \req{eq:ambiguity_Lambda_Def}, which is one reason why we view Definition \ref{def:tail_decay_ambiguity_set} as a preferred way to construct R{\'e}nyi-divergence ambiguity sets. 
\begin{theorem}\label{thm:tail_decay_ambiguity_set}
Let $\mu$, $G_\mu$, and $\Lambda_\mu$ be as in Definition \ref{def:tail_decay_ambiguity_set}:
\begin{enumerate}
\item If $Q\ll P$ and we have $r_0>0$ such that
\begin{align}\label{eq:P_r0_bounds}
P(dQ/dP\geq r)\geq G_\mu(r)\,\text{ for all $r\in[0,r_0]$  and }\, P(dQ/dP\geq r)\leq G_\mu(r) \,\text{ for all $r> r_0$}
\end{align}
then $Q\in \mathcal{U}^{\mu}(P)$.
\item Suppose $Q_0\in \mathcal{U}^{\mu}(P)$ and $Q\ll P$.  If  we have $r_0>0$ such that
\begin{align}& r_01_{dQ_0/dP\leq r_0}\geq\frac{dQ}{dP} 1_{dQ_0/dP\leq r_0}\geq \frac{dQ_0}{dP} 1_{dQ_0/dP\leq r_0}\,\,\,\, P-a.s.\,\,\,\text{ and }\\
&r_01_{dQ_0/dP\geq r_0}\leq \frac{dQ}{dP} 1_{dQ_0/dP\geq r_0}\leq \frac{dQ_0}{dP} 1_{dQ_0/dP\geq r_0}\,\,\,\, P-a.s.\notag
\end{align}
then $Q\in\mathcal{U}^{\mu}(P)$.
\item Suppose $Q_0\in \mathcal{U}^{\mu}(P)$ and $\nu$ is a sigma-finite positive measure with $dQ_0=q_0 d\nu$, $dP=pd\nu$, $dQ=q d\nu$. If we have $r_0>0$ such that $q$ is between $q_0$ and $r_0p$ pointwise $\nu$-a.s. then $Q\in\mathcal{U}^{\mu}(P)$.
\end{enumerate}

\end{theorem}
\begin{remark}
Intuitively, property (1) states that any $Q$ whose likelihood decays faster than $G_\mu$ is in $\mathcal{U}^\mu(P)$.   The intuitive meaning of property  (3) is immediate; see Figure \ref{fig:U_examples}.
\end{remark}
\begin{proof}
\begin{enumerate}
\item  Let $\widetilde{\mu}$ be the distribution of $dQ/dP$ under $P$. The second bound in \req{eq:P_r0_bounds} implies that $\Lambda_{\tilde\mu}$ is finite in a neighborhood of $0$. From part (1) of Lemma \ref{Lemma:U_inclusion} we see that $\mathcal{U}^{\widetilde{\mu}}(P)\subset\mathcal{U}^\mu(P)$, and from Lemma \ref{Lemma:KL_bound} we see that $Q\in \mathcal{U}^{\widetilde{\mu}}(P)$.
\item By set inclusions, it is straightforward to see that the assumptions imply  $P(dQ/dP\geq r)\geq P(dQ_0/dP\geq r)$ for all $r\in[0,r_0]$ and $P(dQ/dP\geq r)\leq P(dQ_0/dP\geq r)$ for all $r>r_0$. Similarly to the proof of Lemma \ref{Lemma:U_inclusion},  the formula \req{eq:G_MGF_bound} then implies $\Lambda_Q^{\log(dQ/dP)}(\lambda)\leq \Lambda_{Q_0}^{\log(dQ_0/dP)}(\lambda)$ for all $\lambda>0$. This proves the claim.
\item The assumptions imply that $Q\ll P$ and $dQ/dP=\frac{q}{p} 1_{p> 0}$ and $dQ_0/dP=\frac{q_0}{p} 1_{p>0}$.  The result then follows from part (2).
\end{enumerate}
\end{proof}

Figure \ref{fig:U_examples} illustrates property (3) from Theorem \ref{thm:tail_decay_ambiguity_set}: Given a baseline model $P$, with density shown given by the black solid curve, and an alternative model $Q_0\in\mathcal{U}^\mu(P)$, with density given by the blue dashed curve, then  any other probability measure $Q$ whose densities lies in either of the gray regions (for example, either of the red dashed curves) is also in $\mathcal{U}^\mu(P)$. Since $Q_0$ (perhaps being constructed via Theorem \ref{Thm:construct_Q}) can has a different (slower) decay than $P$, so also can $Q$. This is in stark contrast to the classical examples from Section \ref{sec:classical_MGF_ex}.

\begin{figure}[h]
\minipage{0.5\textwidth}
  \includegraphics[width=\linewidth]{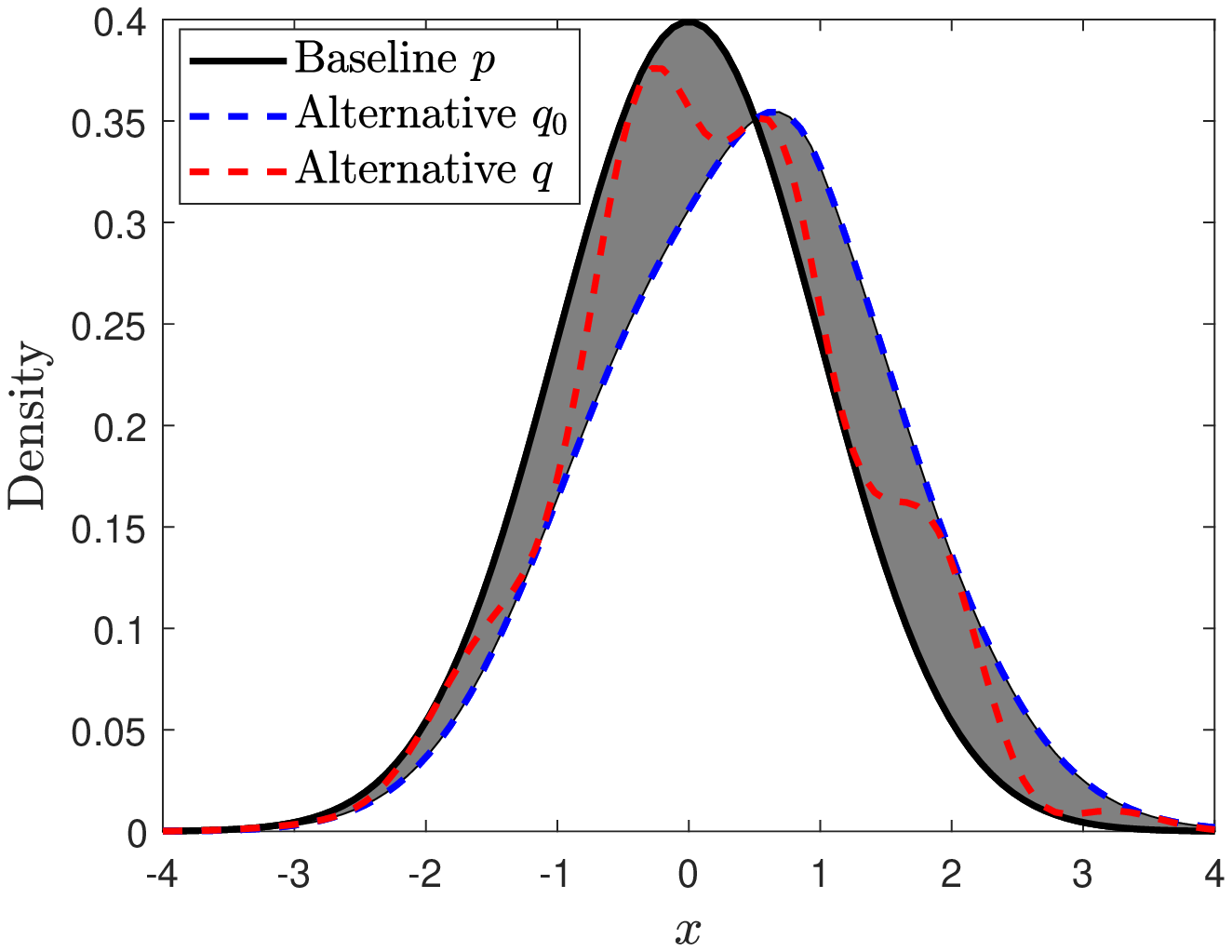}
\endminipage\hfill
\minipage{0.5\textwidth}%
  \includegraphics[width=\linewidth]{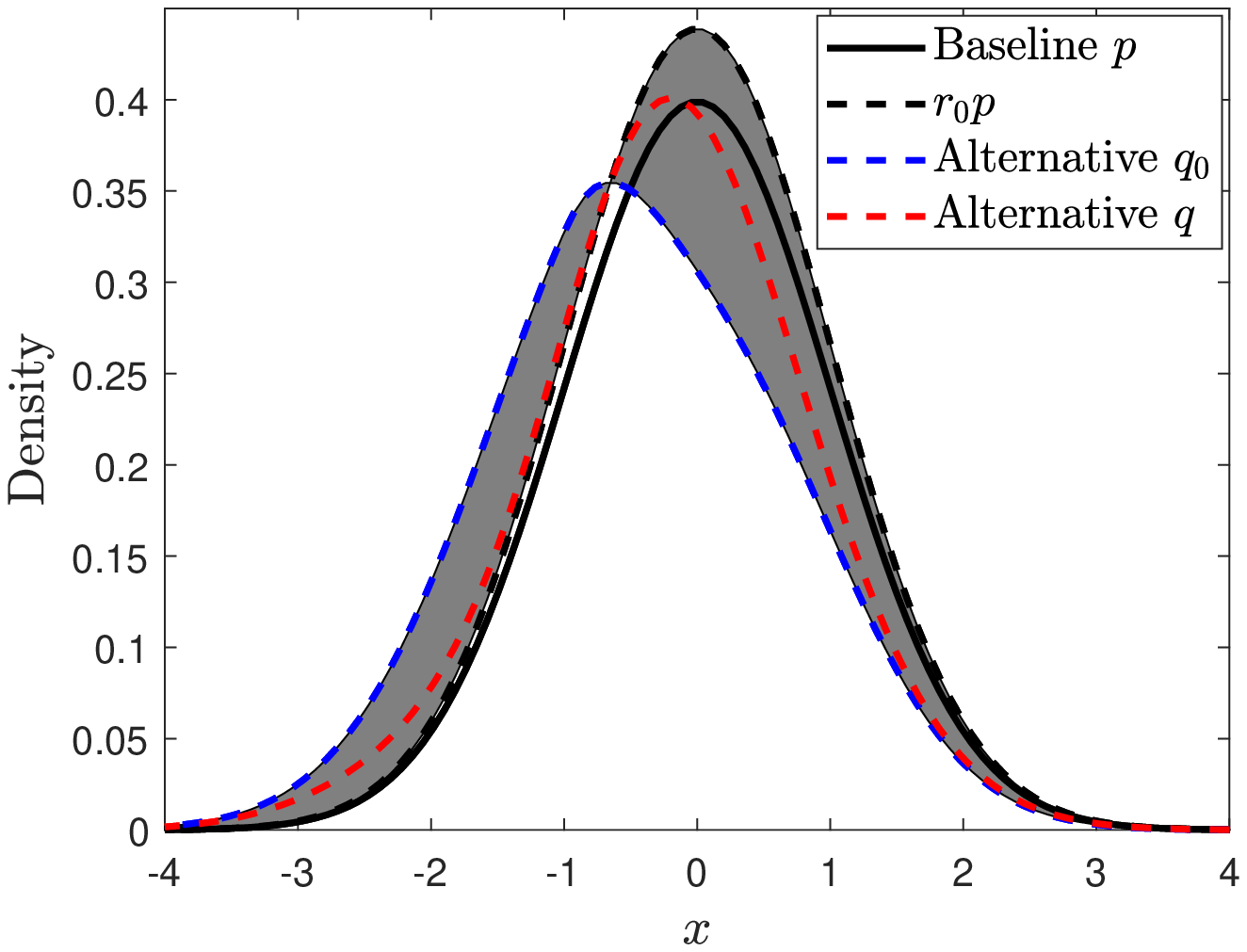}
\endminipage
\caption{If $Q_0\in\mathcal{U}^{\mu}(P)$  and  $P$ and $Q_0$ have densities $p$ (black solid curve) and $q_0$ (blue dashed curve) respectively, then, for any $r_0>0$, any probability measure whose density lies between $r_0p$ (black dashed curve) and $q_0$ (red dashed curve) is also in $\mathcal{U}^\mu(P)$. The left plot is for $r_0=1$ and the right for $r_0=1.1$. Note that the R{\'e}nyi-divergences only depend on the decay (in probability) of the likelihood-ratio: $P(dQ/dP\geq r)$. In particular, they cannot tell {\em which} tail of $P$ is perturbed; reflecting either plot around the $y$ axis doesn't change any of the R{\'e}nyi divergences. } \label{fig:U_examples}
\end{figure}

 Note that R{\'e}nyi divergences can only access information on  $P(dQ/dP\geq r)$ (see Lemma \ref{lemma:tail}); in particular, they cannot distinguish  {\em which} tail is perturbed.  For instance, reflecting either plot from Figure \ref{fig:U_examples}  about the $y$-axis doesn't change the corresponding R{\'e}nyi divergences.

Finally, it is also possible to construct ambiguity sets from two-sided bounds on the deviation of the log-likelihood from its mean.  This method appears both less fundamental and less useful, so we relegate discussion of it to \ref{sec:2_sided_tail}.

\section{Examples}\label{sec:U_examples}
Here we provide several examples that illustrate the ambiguity-set definition (\ref{eq:U_mu_def}) as well as the method of constructing a member of $\mathcal{U}^\mu(P)$ that was developed in Theorem \ref{Thm:construct_Q}. More substantial applications of our methods can be found in Section \ref{sec:applications}.
\subsection{Constructing Ambiguity Sets from Tail Behavior}\label{sec:ex_ambiguity_set}
Each choice of $\mu$ in Theorem \ref{thm:tail_decay_ambiguity_set} represents a different type of perturbation of $P$; intuitively, a slower decay in $r$ implies a more extreme perturbation to tail-probabilities and hence to risk-sensitive quantities. We provide several classes of probability measures, $\mu$, on $[0,\infty)$ with mean $1$ that can be used in Definition 
\ref{def:tail_decay_ambiguity_set} to capture various important forms of decay behavior.  Each of these families contain a parameter that can be fixed by specifying a relative-entropy bound via Lemma \ref{Lemma:KL_bound}.

As motivation, we note that the first of the  families introduced below will have the following effect when applied to the battery example discussed in Section \ref{sec:motivation} (see Section \ref{sec:Battery} below for details): the resulting ambiguity sets will contain perturbations of the base model, $P$,  with  power-law decay at $0$ of a {\em slower} decay rate than $P$.  UQ bounds over one of these ambiguity sets will then constitute a stress-test of the model under this form of  decay-rate perturbation; again, a particular member of the family will be singled out via a relative-entropy bound.

\begin{enumerate}
\item Power-law decay:  Let $0< r_0<1$ and consider the  family of Pareto distributions
\begin{align}
d\mu_{r_0}=\frac{\eta-1}{r_0}(x/r_0)^{-\eta}1_{x\geq r_0}dx,\,\,\,\eta\equiv (2-r_0)/(1-r_0)
\end{align}
 (note that $r_0=(\eta-2)/(\eta-1)$ and $\eta>2$).  The corresponding $G_{r_0}(r)\equiv\mu_{r_0}([r,\infty))$ and its inverse are
\begin{align}\label{eq:G_power_law}
G_{r_0}(r)=&1_{r\leq r_0}+(r/r_0)^{-(\eta-1)}1_{r>r_0},\\
G_{r_0}^{-1}(y)=& r_0y^{-1/(\eta-1)}
\end{align}
(recall that $G_{r_0}^{-1}$ denotes the inverse on $[r_0,\infty)$; see Corollary \ref{corr:construct_Q}). We also have
\begin{align}\label{eq:power_law_lambda}
\Lambda_{{r_0}}(\lambda)=&\log\left((\lambda+1)\int_0^\infty G_{r_0}(z)z^\lambda dz\right)\\
=&\lambda\log(r_0)+\log\left(1+\frac{\lambda}{r_0/(1-r_0)-\lambda}\right),\,\,\,-1<\lambda<r_0/(1-r_0).\notag
\end{align}
We denote the resulting ambiguity sets by $\mathcal{U}^{PL}_{r_0}(P)$. Finally, we can relate $r_0$ to the relative entropy bound by computing the derivative at $0$:
\begin{align}\label{eq:R_x0}
R(Q\|P)\leq \log(r_0)+1/r_0-1,\,\,\,Q\in\mathcal{U}^{PL}_{r_0}(P).
\end{align}
%Note that \req{eq:R_x0} is decreasing in $r_0\in(0,1)$, so there is a unique $r_0$ corresponding to any $R(Q\|P)$.
%range is $(-\infty,0)$

\item Sub/Super-exponential decay: Define  $\mathcal{U}^{Exp}_{r_0}(P)$ using
\begin{align}
d\mu_{r_0}=&\kappa\eta (x-r_0)^{\kappa-1}\exp\left(-\eta(x-r_0)^\kappa\right)1_{x\geq r_0}dx,\,\,\,\,0\leq r_0<1,\,\,\,\kappa>0,\,\,\,\eta=\left(\frac{\Gamma(1+1/\kappa)}{1-r_0}\right)^\kappa,\label{eq:exp_mu}\\
G_{r_0}(r)=&1_{r\leq r_0}+e^{-\eta(r-r_0)^\kappa}1_{r>r_0},\label{eq:G_sub_exp}\\
G^{-1}_{r_0}(y)=&r_0+(\eta^{-1}\log(1/y))^{1/\kappa},\\
\Lambda_{r_0}(\lambda)=&\log\left(r_0^{\lambda+1}+(\lambda+1)\int_{r_0}^\infty \exp\left(-\eta(z-r_0)^\kappa\right)z^\lambda dz\right),\,\,\,\lambda\geq 0,\label{eq:Lambda_sub_exp}\\ %actually, holds for \lambda>-1.
R(Q\|P)\leq & 1+r_0\log(r_0)-r_0+ \int_{r_0}^{\infty} \exp\left(-\eta(z-r_0)^\kappa\right)\log(z) dz,\,\,\,Q\in\mathcal{U}^{Exp}_{r_0}(P).\label{eq:R_x0_2}
\end{align}
For a given value of $\kappa$, $r_0$ can be fixed by specifying a desired relative-entropy bound. Note that \req{eq:exp_mu} is a family of shifted Weibull distributions. We will refer to the case $\kappa=1$  as exponential  decay.

\item Perturbation of a Gaussian:  Let $P_2= N(\mu,\sigma_2^2)$, $P_1= N(\mu,\sigma_1^2)$ be two normal distributions on $\mathbb{R}$ with $\sigma_2>\sigma_1>0$.  Define $r_0=\sigma_1/\sigma_2$ and let $\mu_{r_0}$ be the distribution of $dP_2/dP_1$ under $P_1$. Denote the resulting ambiguity set by $\mathcal{U}^{G}_{r_0}(P)$. We have:
\begin{align}
G_{r_0}(r)=&1_{r\leq r_0}+\erfc\left(C(r)\right)1_{r>r_0},\,\,\,\,r_0\equiv \sigma_1/\sigma_2,\,\,\,\,C(r)\equiv\left(1-r_0^2\right)^{-1/2}\log\left(r/r_0\right)^{1/2},\label{eq:G_gaussian}\\
G_{r_0}^{-1}(y)=&r_0\exp\left((1-r_0^2)(\erfc^{-1}(y))^2\right),\\
\Lambda_{r_0}(\lambda)=&\lambda(\lambda+1)R_{\lambda+1}(P_2\|P_1)=\lambda\log(r_0)-\frac{1}{2}\log(1-\lambda(r_0^{-2}-1)),\,\,\,0\leq \lambda <(r_0^{-2}-1)^{-1},\\
R(Q\|P)\leq& R(P_2\|P_1)= \log(r_0)+ \frac{1}{2}\left(r_0^{-2}-1\right),\,\,\,Q\in\mathcal{U}^G_{r_0}(P).\label{eq:R_x0_3}
\end{align}
As our notation suggests, the ambiguity set is determined by the value of the single parameter $r_0=\sigma_1/\sigma_2\in(0,1)$, which can again be fixed by specifying a relative-entropy bound. We will refer to ambiguity sets constructed in this manner as Gaussian ambiguity sets. 
\begin{remark}\label{remark:G_PL_similar}
The Gaussian ambiguity sets are similar to the power-law decay family. This can be seen by approximating $\erfc(x)\approx e^{-x^2}/(\sqrt{\pi}x)$, which implies that $G_\mu(r)\approx Cr^{-\alpha}/\sqrt{\log(r/r_0)}$ for some $C,\alpha>0$.  In practice, we find the densities constructed via Theorem \ref{Thm:construct_Q} using a Gaussian ambiguity set to be smoother than those constructed via a power-law ambiguity set (\ref{eq:G_power_law}), and so the Gaussian case can be viewed as a reasonable replacement for (\ref{eq:G_power_law}) when this smoothness is desired.
\end{remark}
\end{enumerate}

\begin{figure}[h]
\minipage{0.5\textwidth}
  \includegraphics[width=\linewidth]{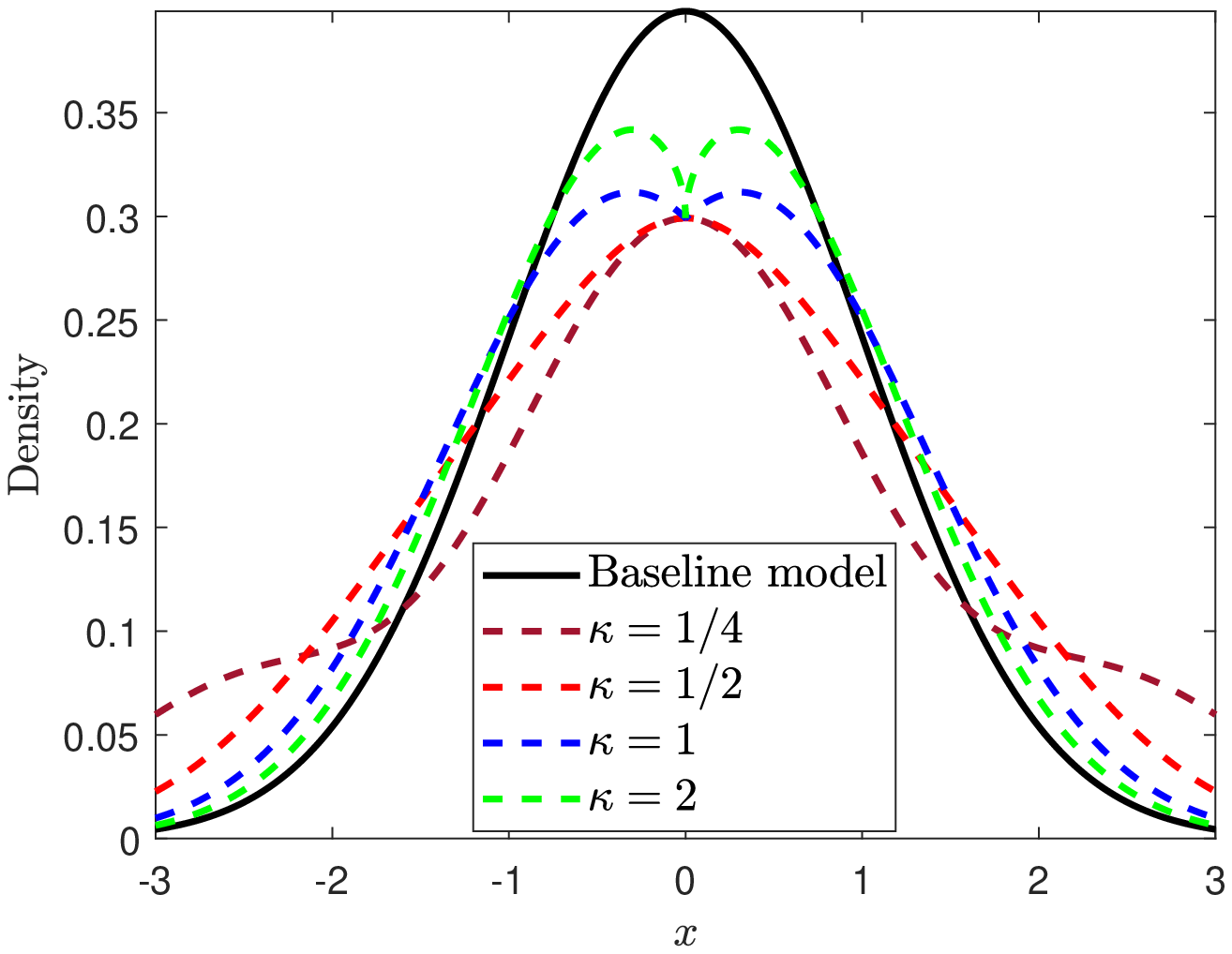}
\endminipage\hfill
\minipage{0.5\textwidth}%
  \includegraphics[width=\linewidth]{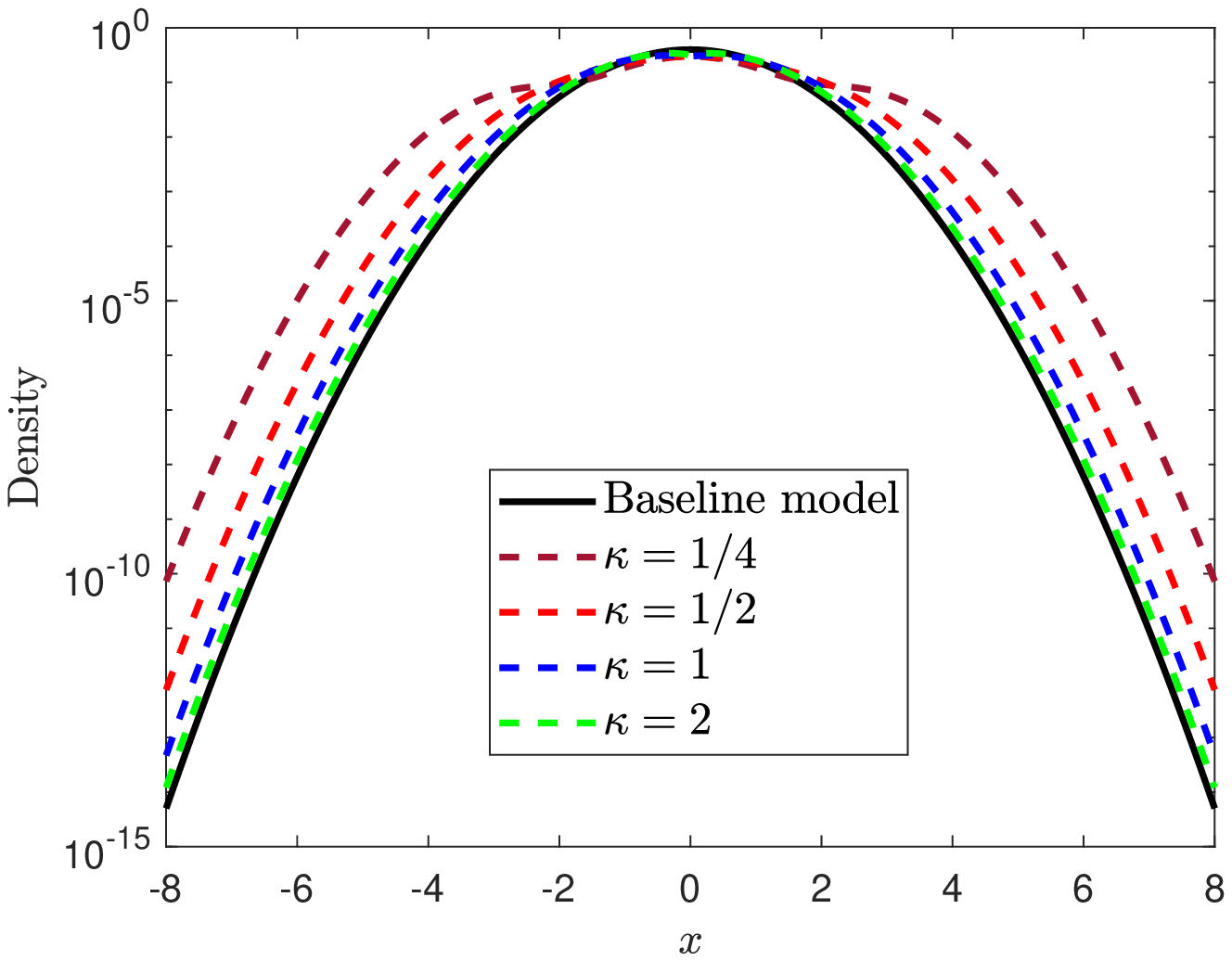}
\endminipage
\caption{ Alternative models that saturate the MGF bound (i.e., that satisfy \req{eq:saturate_bound}) for the sub/super-exponential ambiguity sets with $r_0=3/4$ and $\kappa=1/2,1,2$; see \req{eq:Q_gaussian}.  The alternative models (densities, $q$, shown in dashed lines) are constructed using the procedure from Theorem \ref{Thm:construct_Q}.  The baseline model (density, $p$, given  by the black solid line) is a Gaussian with mean $0$ and variance $1$. The value of $r_0$ gives the minimum of $q/p$ and this always occurs at the maximum of $p$.  Each ambiguity set contains all probability measures whose densities are between $p$ and the corresponding $q$ (see Theorem \ref{thm:tail_decay_ambiguity_set} for this, and other such criteria).  } \label{fig:normal_perturb}
\end{figure}

\subsection{Constructing a Member of $\mathcal{U}^\mu(P)$}
We now give several examples where the construction from Theorem \ref{Thm:construct_Q}  of an alternative measure $Q$ that  saturates the defining MGF bound of $\mathcal{U}^\mu(P)$ can be carried out explicitly; we use the ambiguity sets from Section \ref{sec:ex_ambiguity_set}. Recall that each of these involves a parameter, $r_0$, that can be set by imposing a desired relative-entropy bound; see \req{eq:R_x0}, \req{eq:R_x0_2}, and \req{eq:R_x0_3}.
\begin{enumerate}
\item Let $\lambda>0$ and $P(dt)=\lambda e^{-\lambda t}dt$ on $[0,\infty)$. Then, letting $p(t)=\lambda e^{-\lambda t}$ be the density with respect to Lebesgue measure, we have
\begin{align}
\psi(y)=P(p\leq y)=\min\{y/\lambda,1\},\,\,\,\,\,\psi(p(t))=e^{-\lambda t}.
\end{align}
 Imposing power-law decay of the likelihood-ratio, with some choice of $0<r_0<1$ (see example 1 in Section \ref{sec:ex_ambiguity_set}),  results in
\begin{align}
&Q(dt)=\phi(p(t))P(dt),\,\,\,\,\,\phi(p(t))=G_{r_0}^{-1}(\psi(p(t)))=r_0\exp\left(\frac{\lambda}{\eta-1}t\right),\,\,\,\eta=(2-r_0)/(1-r_0),
\end{align}
i.e., $Q$ is exponentially distributed with the slower rate $r_0\lambda$.
\item Again, let  $P(dt)=\lambda e^{-\lambda t}dt$.  This time we impose sub/super-exponential decay of the likelihood ratio for some choice of $r_0\in[0,1)$ and $\kappa>0$  (see example 2 in Section \ref{sec:ex_ambiguity_set}), and  find
\begin{align}
&Q(dt)=\phi(p(t))P(dt),\,\,\,\,\,\phi(p(t))=r_0+\left(\lambda t/\eta\right)^{1/\kappa},\,\,\,\,\eta=\left(\frac{\Gamma(1+1/\kappa)}{1-r_0}\right)^\kappa,
\end{align}
i.e., the tail of $dQ/dt$ is a power-law perturbation of $dP/dt$.
\item Let $P=N(\mu,\sigma^2)$ be a normal distribution on $\mathbb{R}$ with $\sigma>0$. Then, letting $p$ be the density with respect to Lebesgue measure, we have
\begin{align}
\psi(y)=&\erfc\left(\log\left(\frac{1}{(2\pi\sigma^2)^{1/2}y}\right)^{1/2}\right) 1_{y\leq (2\pi\sigma^2)^{-1/2}}+1_{y> (2\pi\sigma^2)^{-1/2}},\,\,\,\,\psi(p)=\erfc\left(\frac{|x-\mu|}{\sqrt{2}\sigma}\right).
\end{align}
Imposing the sub/super exponential decay \req{eq:G_sub_exp} for some $r_0\in[0,1)$ and $\kappa>0$ we have
\begin{align}\label{eq:Q_gaussian}
Q(dx)=&\phi(p(x))P(dx),\,\,\,\,\,\phi(p(x))=r_0+\left(\eta^{-1}\log\left(1/\erfc\left(\frac{|x-\mu|}{\sqrt{2}\sigma}\right)\right)\right)^{1/\kappa},\,\,\,\eta=\left(\frac{\Gamma(1+1/\kappa)}{1-r_0}\right)^\kappa.
\end{align}
Several different $Q$'s obtained in this manner are shown in Figure \ref{fig:normal_perturb}.
\end{enumerate}
Again, we emphasize that once one has a single member of $\mathcal{U}^\mu(P)$, other members can be obtained by using Theorem \ref{thm:tail_decay_ambiguity_set} (see Figure \ref{fig:U_examples} for illustrative examples).

\begin{remark}
Note that we have named the general classes of ambiguity sets constructed via \req{eq:power_law_lambda} or \req{eq:Lambda_sub_exp} after the decay of the likelihood ratio  under $P$: $P(dQ/dP\geq r)$.  For any particular instance of one of these constructions, the tail behavior  with respect to (for example) Lebesgue measure ($dQ/dx$ as compared to $dP/dx$) will be highly dependent on the tail behavior of the chosen baseline model $P$; it will often be quite different from the decay behavior of $P(dQ/dP\geq r)$ (as was illustrated by the above examples) and  must be investigated on a case-by-case basis.
\end{remark}

\section{Uncertainty Quantification for Risk-Sensitive QoIs}\label{sec:Renyi_UQ}
Having explored  the structure of R{\'e}nyi-divergence ambiguity sets, we now move towards the application of these ideas to uncertainty quantification for risk-sensitive QoIs.  The techniques we develop here will apply to non-negative random variables, $\tau$, and from this point on we will formally call any quantity of the form $\log(E_P[\tau])$ (or simply $\tau$ by itself) a risk-sensitive QoI ($E_P$ denotes the expectation under $P$), though our results are most relevant for those $\tau$'s that are risk-sensitive in the intuitive sense (i.e., sensitive to the tail(s) of $P$), e.g., rare-events, partition functions, moment generating functions, exit times, or  random variables whose MGF is not finite on any neighborhood of $0$ (i.e., heavy tailed).

 In this section we discuss the general theory of R{\'e}nyi-divergence-based UQ, both summarizing the relevant established literature as well as proving some new additions.   In particular, we derive a  Donsker-Varadhan variational formula for R{\'e}nyi divergences, a generalized version of the UQ bound \req{prop:UQ1}, and prove several new tightness and divergence properties. Applications and examples will be given in Section \ref{sec:applications}.

\subsection{Variational Principles and R{\'e}nyi Divergences}

 UQ bounds for non-risk-sensitive QoIs can be derived from variational formulas involving relative entropy  \cite{dupuis2011uq,BREUER20131552,glasserman2014,lam:robustsens,GKRW,PK2013,dupuis2015path,doi:10.1137/15M1047271,KRW,2018arXiv181205174B}. For motivation, we recall the relevant variational formulas:
\begin{enumerate}
\item The Donsker-Varadhan variational formula:
\begin{proposition}\label{prop:DV_var}
Let  $P,Q$ be probability measures on $(\Omega,\mathcal{M})$.  Then
\begin{align}\label{eq:DV_var}
R(Q\|P)=\sup_{g\in B({\Omega})}\left\{E_Q[g]-\log\left[\int e^g dP\right]\right\},
\end{align}
where $B({\Omega})$ denotes the set of bounded  measurable real-valued functions on $(\Omega,\mathcal{M})$.
\end{proposition}
\item Its dual relation, the Gibbs variational principle:
\begin{proposition}
Let  $P$ be a probability measure on $(\Omega,\mathcal{M})$ and $g\in B({\Omega})$. Then
\begin{align}\label{eq:Gibbs}
\log\left[\int e^g dP\right]=\sup_{Q\in\mathcal{P}(\Omega)}\{E_Q[g]-R(Q\|P)\},
\end{align}
where $\mathcal{P}(\Omega)$ denotes the set of probability measures on $(\Omega,\mathcal{M})$.
\end{proposition}
\end{enumerate}

The variational principles \req{eq:DV_var} and \req{eq:Gibbs} have counterparts that involve R{\'e}nyi divergences.  A R{\'e}nyi-based variant of \req{eq:Gibbs} was derived in  \cite{atar2015robust}:
\begin{proposition}
\label{prop:duality}
Let  $P$ be a probability measure on $(\Omega,\mathcal{M})$, $g\in B({\Omega})$ and $\beta,\gamma\in\mathbb{R}\setminus\{0\}$ with $\beta<\gamma$.  Then
\begin{align}\label{eq:Renyi_MGF_inf}
\frac{1}{\beta}\log\left[\int e^{\beta g} dP\right]=\inf_{Q\in\mathcal{P}(\Omega)}\left\{\frac{1}{\gamma} \log\left[\int e^{\gamma g}dQ \right]+\frac{1}{\gamma-\beta}R_{\gamma/(\gamma-\beta)}(P\|Q)\right\}
\end{align}
and
\begin{align}\label{eq:Renyi_MGF_sup}
\frac{1}{\gamma}\log\left[\int e^{\gamma g} dP\right]=\sup_{Q\in\mathcal{P}(\Omega)}\left\{\frac{1}{\beta}\log\left[\int e^{\beta g}dQ \right]-\frac{1}{\gamma-\beta}R_{\gamma/(\gamma-\beta)}(Q\|P)\right\}.
\end{align}
\end{proposition}

The R{\'e}nyi divergences also have a variational characterization, generalizing the  Donsker-Varadhan variational formula and complementing Proposition \ref{prop:duality}. This is a new result; the  proof can be found in  \ref{app:gen_DV}:
\begin{theorem}[R{\'e}nyi-Donsker-Varadhan Variational Formula]\label{thm:gen_DV}
Let  $P,Q$ be probability measures on $(\Omega,\mathcal{M})$ and $\alpha\in\mathbb{R}\setminus\{0,1\}$. Then
\begin{align}\label{eq:gen_DV}
R_\alpha(Q\|P)=\sup_{g\in B(\Omega)}\left\{\frac{1}{\alpha-1}\log\left[\int e^{(\alpha-1)g}dQ\right]-\frac{1}{\alpha}\log\left[\int e^{\alpha g}dP\right]\right\}.
\end{align}
\end{theorem}
\begin{remark}
Formally taking $\alpha\searrow 1$ in \req{eq:gen_DV}, one recovers the Donsker-Varadahn variational formula for relative entropy, \req{eq:DV_var}.
\end{remark}

\subsection{Non-Perturbative Robustness Bounds for Risk-Sensitive QoIs}
The Gibbs variational principle (\ref{eq:Gibbs}) leads to the  UQ bounds for non-risk-sensitive QoIs that were quoted in the introduction; see Proposition \ref{prop:info_ineq0}. This result, and the corresponding linearized theory (i.e., sensitivity analysis), were generalized to risk-sensitive QoIs in \cite{atar2015robust,2018arXiv180506917D}; specifically, Proposition \ref{prop:duality} leads to the UQ bounds from Proposition \ref{prop:UQ1}. Combining this with our discussion of R{\'e}nyi-ambiguity-sets in Section \ref{sec:MGF_ambiguity}, we will obtain useful tool for risk-sensitive UQ.  Before moving on to applications, we first extend the results of \cite{atar2015robust,2018arXiv180506917D} to apply to a more general class of QoIs, as well as prove some additional tightness and divergence properties.

The following  bound for risk-sensitive QoIs  is a  generalization of the result from \cite{atar2015robust}, (quoted above in Proposition \ref{prop:UQ1}), now allowing  $g$ to take the values $\pm\infty$. The proof can be found in \ref{app:UQ_proof}, though if $\tau=e^g$ for $g$ bounded, the result is an immediate consequence of Proposition \ref{prop:duality}.
\begin{lemma}\label{thm:Renyi_UQ}
Let $P,Q$ be probability measures on $(\Omega,\mathcal{M})$ and $\tau:\Omega\to[0,\infty]$ be measurable.  Then: 
\begin{enumerate}
\item
\begin{align}\label{eq:UQ_upper}
&\log\left[\int \tau dQ\right]\leq\inf_{c>1}\left\{\frac{1}{c}\log\left[ \int \tau^c dP\right]+\frac{1}{c-1}R_{c/(c-1)}(Q\|P)\right\},
\end{align}
where we define $-\infty+\infty\equiv\infty$.
\item \begin{align}\label{eq:UQ_lower}
\log\left[\int \tau dQ\right]\geq& \sup_{c<1,c\neq 0}\left\{\frac{1}{c}\log\left[ \int \tau^cdP\right]-\frac{1}{1-c}R_{1/(1-c)}(P\|Q)\right\},
\end{align}
where we define  $\infty-\infty\equiv-\infty$.  

\end{enumerate}
Here and in the following, powers of non-negative extended reals are defined by $\tau^c=\exp(c\log(\tau))$, along with the use of the continuous extensions of $\exp$ and $\log$ to the extended reals. 
\end{lemma}
\begin{remark}
  The above generalization is necessary if one wishes to treat event probabilities (i.e., $\tau=1_A$) within the same unified framework as positive QoIs, $\tau=e^g$.  Event probabilities were studied in \cite{atar2015robust} by a specialized limiting argument, but we find it convenient to have a single result that encompasses the previously studied cases. 
\end{remark}

To understand the utility of Lemma \ref{thm:Renyi_UQ} for risk-sensitive robustness, we  first contrast it with the following robustness bounds for non-risk-sensitive QoIs, derived in \cite{dupuis2011uq,dupuis2015path,BLM,BREUER20131552,glasserman2014}. These bound the $Q$-expectation of $f$ for all $Q$ in a relative entropy (i.e., KL-divergence) neighborhood of $P$:
\begin{proposition}[\bf Gibbs information inequality]\label{prop:info_ineq0}
Let  $f:\Omega\to\overline{\mathbb{R}}$, $f\in L^1(P)\cap L^1(Q)$, and define $\widehat{f}=f-E_P[f]$.  Then 
\begin{align}\label{goal_oriented_bound}
\pm E_{Q}[f]\leq \inf_{c>0}\left\{\frac{1}{c} \Lambda_{P}^{f}(\pm c)+\frac{1}{c}R(Q||P)\right\},
\end{align}
 where $\overline{\mathbb{R}}$ denotes the extended reals and $\Lambda_P^f(c)=\log E_P[e^{cf}]$ is the cumulant generating function (CGF) of $f$ under $P$.
\end{proposition}
The bound \req{goal_oriented_bound} is nontrivial only when the MGF is finite in a neighborhood of $0$; this is one sign that Proposition \ref{prop:info_ineq0} is inappropriate for risk-sensitive QoIs. The bounds in Lemma \ref{thm:Renyi_UQ} only require the existence of $c$'th moments of the QoI for an appropriate interval of $c$'s.

 Lemma \ref{thm:Renyi_UQ} and Proposition \ref{prop:info_ineq0} together illustrate an important message: UQ for non-risk-sensitive quantities only requires control of the mean of the log-likelihood, i.e., the relative entropy
\begin{align}
R(Q\|P)=E_Q[\log(dQ/dP)],
\end{align}
while UQ for risk-sensitive QoIs requires control of the MGF of the log-likelihood (i.e., control of all moments), which is equivalent to having control on the R{\'e}nyi-divergences (see \req{eq:Renyi_MGF}).  This  is why the method of constructing ambiguity sets in terms of R{\'e}nyi-divergences ambiguity sets from Section \ref{sec:background} is appropriate for deriving UQ bounds on risk-sensitive QoIs.

Combining these ideas   results in a powerful tool for  addressing the distributional robustness problem (\ref{eq:general_goal}):
\begin{theorem}\label{thm:main_UQ_result}
Let $P$ be a probability measures on $(\Omega,\mathcal{M})$, $\tau:\Omega\to[0,\infty]$ be measurable, and $Q\in \mathcal{U}^\Lambda(P)$ for some choice of $\Lambda$.  Then
\begin{align}\label{eq:UQ_upper}
&\log\left[\int \tau dQ\right]\leq\inf_{c>1}\left\{\frac{1}{c}\log\left[ \int \tau^c dP\right]+\frac{c-1}{c}\Lambda(1/(c-1))\right\},
\end{align}
where we define $-\infty+\infty\equiv\infty$.
\end{theorem}

\subsection{UQ for Rare Events}\label{ex:rare_event}
To gain some intuition on the bound \req{eq:UQ_upper}, we specialize to (rare) events, $\tau=1_A$, $A\in\mathcal{M}$. For $Q\in \mathcal{U}^\Lambda(P)$ we then have
\begin{align}\label{eq:rare_event_UB}
&\log\left[Q(A)\right]\leq\inf_{c>1}\left\{\frac{1}{c}\log\left[P(A)\right]+\frac{c-1}{c}\Lambda\left(1/(c-1)\right)\right\}.
\end{align}
For comparison, the non-risk-sensitive bound from Proposition \ref{prop:info_ineq0} yields
\begin{align}\label{eq:ex_KL_bound}
Q(A)\leq \inf_{c>0}\left\{\frac{1}{c}\log\left[1+ P(A)(e^c -1)\right]+\frac{1}{c}R(Q\|P)\right\}.
\end{align}

We illustrate these bounds using  ambiguity sets constructed from classical MGF-bounds; see Section \ref{sec:MGF_ambiguity} and \ref{app:classical_CGF}. In particular, Figure \ref{fig:bound_comparison}   demonstrates the ambiguity set inclusions from  Lemmas \ref{lemma:inclusion1} - \ref{lemma:inclusion4}, and also compares the risk-sensitive and non-risk-sensitive bounds.  In the non-risk-sensitive bound (\ref{eq:ex_KL_bound}) we   use the bound on $R(Q\|P)$ that is implied by $Q\in \mathcal{U}^\Lambda(P)$, as discussed in Lemma \ref{lemma:U_def}.  The Bennett-bound, which utilizes an upper bound on $\log(dQ/dP)$ as well as its mean and variance under $Q$, is the tightest among the risk-sensitive bounds; this is a consequence of the Lemmas in \ref{app:classical_CGF}. For each risk-sensitive bound there is a crossover level of rarity, below which it becomes  tighter than the non-risk-sensitive bound.  This is a commonly observed occurrence when comparing the methods: the R{\'e}nyi-divergence based methods generally perform better for sufficiently rare events, but for more typical events, the relative-entropy-based method from Proposition \ref{prop:info_ineq0} is  preferable.  In practice, it is of little additional cost to compute both bounds (\ref{eq:rare_event_UB}) and (\ref{eq:ex_KL_bound}) and use the minimum of the two.

\begin{figure}

\minipage{0.33\textwidth}
  \includegraphics[width=\linewidth]{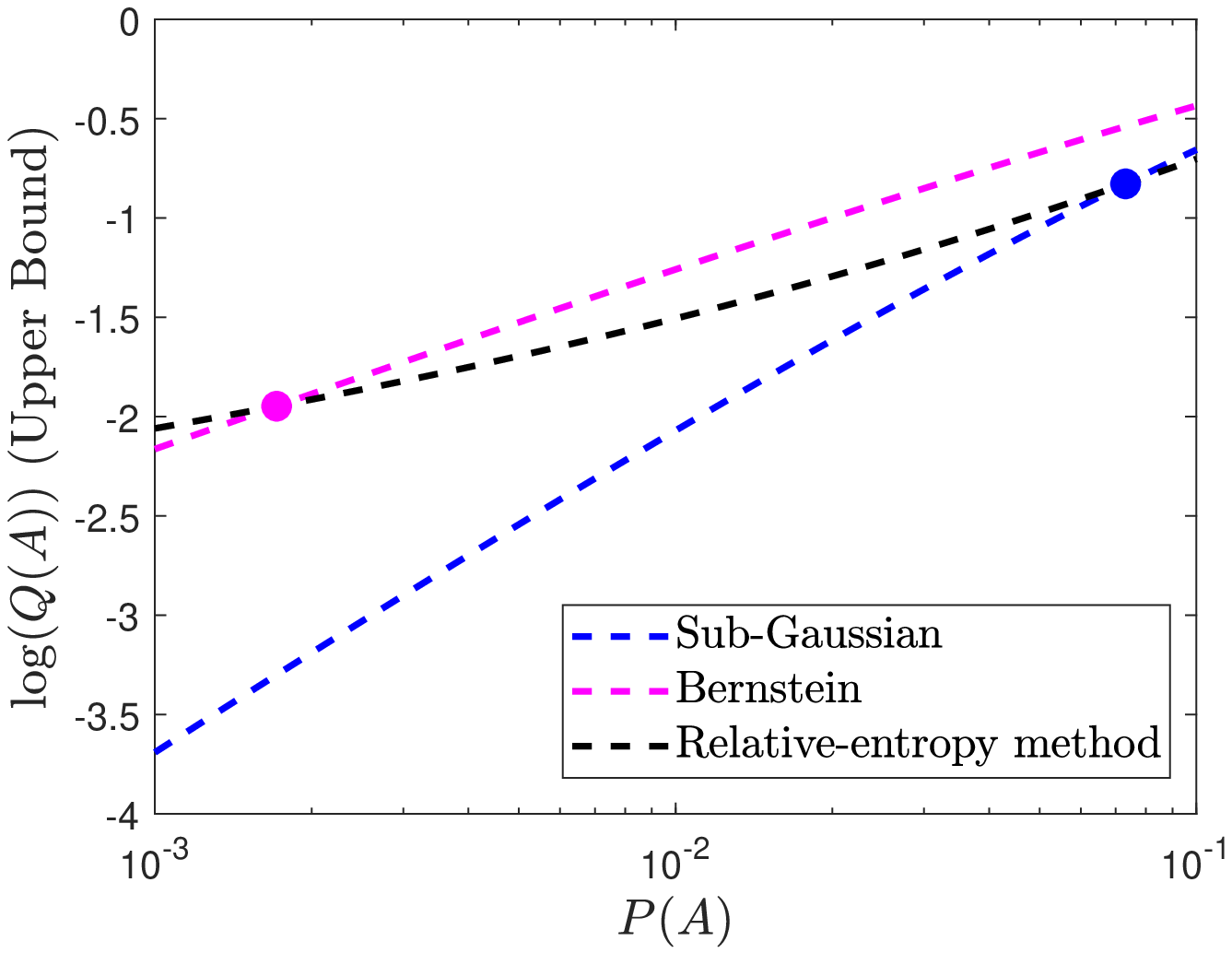}
\endminipage\hfill
\minipage{0.33\textwidth}
  \includegraphics[width=\linewidth]{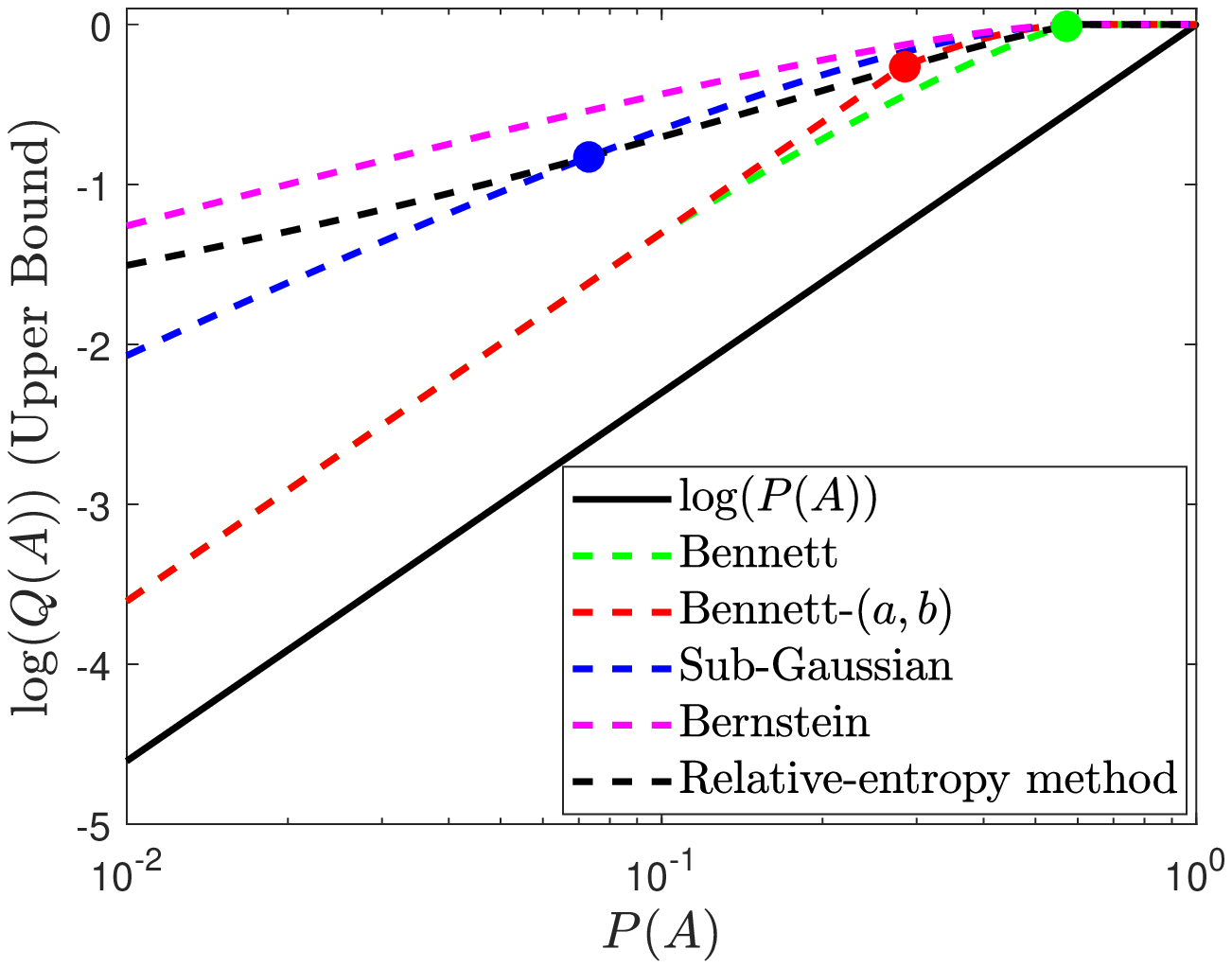}
\endminipage\hfill
\minipage{0.33\textwidth}%
  \includegraphics[width=\linewidth]{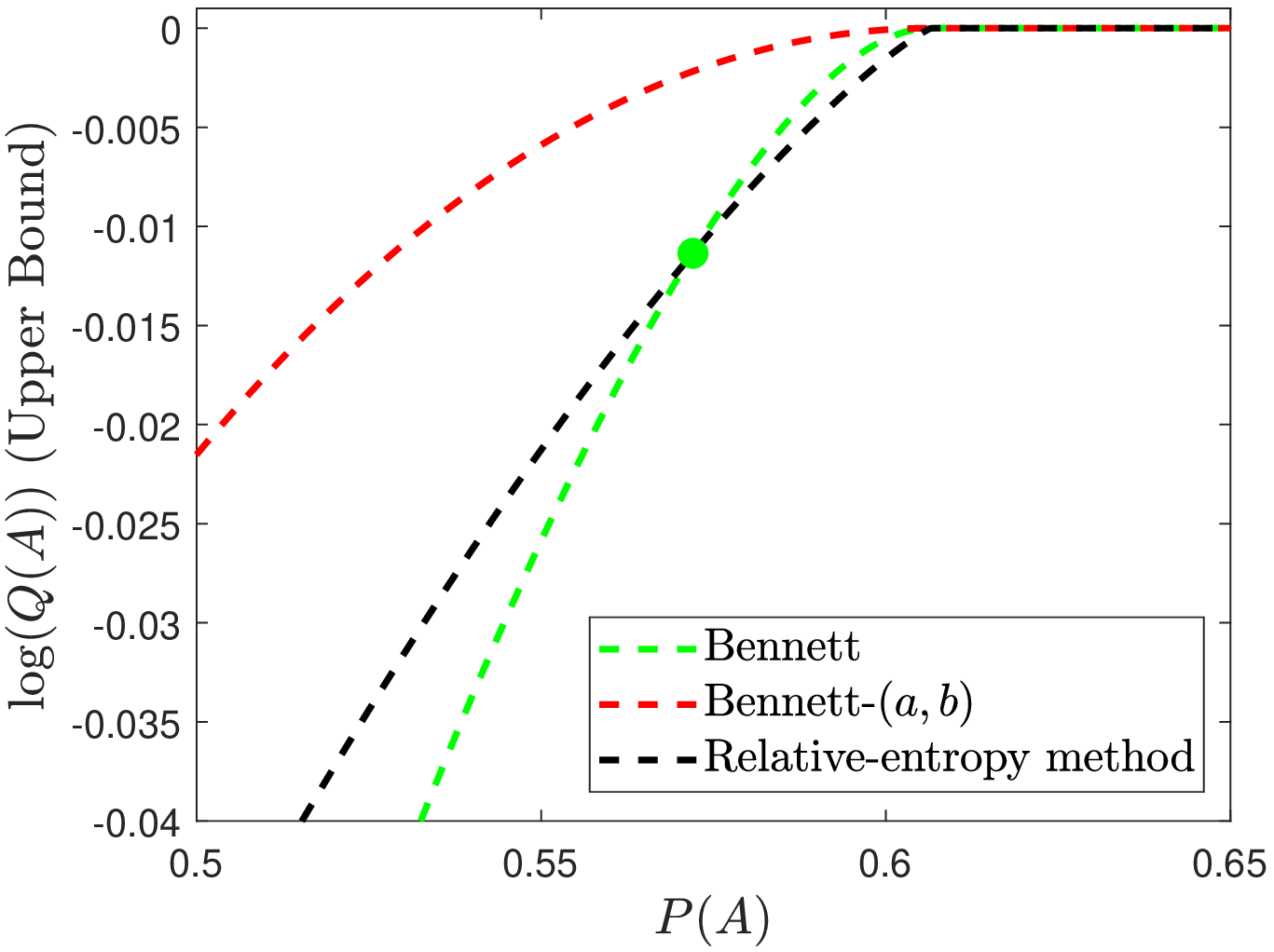}
\endminipage
\caption{Comparison of the non-risk-sensitive UQ upper bound (i.e., relative-entropy-based bound) from \req{goal_oriented_bound}  (black dashed curves) with the risk-sensitive UQ upper bounds from \req{eq:UQ_upper}. We show plots over several different ranges of $P(A)$   and comparison of various  ambiguity-set types defined in  \ref{app:classical_CGF}. These plots  illustrate the following: for each risk-sensitive bound there is a crossover level of rarity (shown by the circles of the corresponding color), below which it becomes  tighter than the non-risk-sensitive bound.   Intuitively, the relative-entropy-based method is superior for `typical events' ($P(A)$ close to $1$), but the R{\'e}nyi-based bounds are tighter for sufficiently rare events.  The parameter values used are as follows: $b=1$, $a=-1$, $\eta=1/2$, $\sigma^2=(b-\eta)(\eta-a)/8$ (for the Bennett ambiguity set), $M=1$, with the variance for the sub-Gaussian ambiguity set chosen as in the inclusion from Lemma \ref{lemma:inclusion3}.}\label{fig:bound_comparison}
\end{figure}

\section{Properties of the Risk-Sensitive UQ Bounds} 
This section contains  several  properties new properties, of both theoretical and computational interest, regarding to the UQ bounds in Lemma \ref{thm:Renyi_UQ}. First, we state the divergence property and then we give various tightness results; proofs can be found in the appendices.

\subsection{Divergence Property}
The  UQ bound from Lemma \ref{thm:Renyi_UQ} can alternatively be formulated in terms of what we call the goal-oriented R{\'e}nyi divergences.  For simplicity, in this subsection we focus on the case $\tau=e^g$ with $g\in B(\Omega)$.
\begin{theorem}[Goal-Oriented R{\'e}nyi Divergence]\label{thm:div_property}
Let $P,Q$ be probability measures on $(\Omega,\mathcal{M})$ and $g\in B(\Omega)$.  Then for $\gamma\in\mathbb{R}\setminus\{0\}$ we have
\begin{align}\label{eq:div_def_plus}
&\frac{1}{\gamma}\log\left[\int e^{\gamma g} dQ\right]-\frac{1}{\gamma}\log\left[\int e^{\gamma g} dP\right]\\
\leq& \inf_{\beta>\gamma,\beta\neq 0}\left\{\frac{1}{\beta} \log\left[\int e^{\beta g}dP \right]-\frac{1}{\gamma}\log\left[\int e^{\gamma g} dP\right]+\frac{1}{\beta-\gamma}R_{\beta/(\beta-\gamma)}(Q\|P)\right\}\equiv \Xi^\gamma_+(Q\|P,g)\notag
\end{align}
and 
\begin{align}\label{eq:div_def_minus}
&\frac{1}{\gamma}\log\left[\int e^{\gamma g} dQ\right]-\frac{1}{\gamma}\log\left[\int e^{\gamma g} dP\right]\\
\geq& \sup_{\beta<\gamma,\beta\neq 0}\left\{ \frac{1}{\beta}\log\left[\int e^{\beta g}dP \right]-\frac{1}{\gamma}\log\left[\int e^{\gamma g} dP\right]-\frac{1}{\gamma-\beta}R_{\gamma/(\gamma-\beta)}(P\|Q)\right\}\equiv \Xi^\gamma_-(Q\|P,g).\notag
\end{align}

$\Xi_\pm^\gamma$ have the following properties:
\begin{enumerate}
\item \begin{equation}\label{eq:Xi_pm_relation}
\Xi_-^\gamma(Q\|P,g)=-\Xi_+^{-\gamma}(Q\|P,-g).
\end{equation}
\item Divergence property: 
\begin{equation}\label{eq:Xi_sign}
\pm\Xi_\pm^\gamma(Q\|P,g)\geq 0,
\end{equation}
and if $g$ is not $P$-a.s. constant then equality holds in \req{eq:Xi_sign} iff $Q=P$. (Note that only one of $\Xi_\pm^\gamma(Q\|P,g)$ must be zero to guarantee that $Q=P$.)
\end{enumerate}

\end{theorem}

\begin{remark}
We call $\Xi_\pm^\gamma$ the goal-oriented R{\'e}nyi divergences.  The term goal-oriented refers to the fact that the divergences incorporate information about the QoI, $g$.
\end{remark}

\subsection{Tightness Properties}\label{sec:tightness}

Next, we show that the bounds in Lemma \ref{thm:Renyi_UQ} satisfy a pair of tightness properties.  First, we show that for a given $\tau$,  we have tightness over a particular range of $Q$'s.  The proof follows by combining the UQ bound from Lemma \ref{thm:Renyi_UQ} with explicit calculations using \req{eq:Renyi_formula}; see  \ref{app:tightness} for a detailed proof.
\begin{theorem}\label{thm:tightness1}
Let  $P$ be a probability measure on $(\Omega,\mathcal{M})$, $\tau:\Omega\to[0,\infty]$ be measurable, and $\gamma\neq 0,-1$ with $0<E_P[\tau^\gamma]<\infty$ ($\tau^\gamma$ denotes the $\gamma$'th power). Define the  probability measure
\begin{equation}
dQ_\gamma=\tau^\gamma dP/E_P[\tau^\gamma]
\end{equation}
and, for any $\beta\in \mathbb{R}$, $I\subset\mathbb{R}$, define the ambiguity sets
\begin{align}\label{eq:U_def}
&\mathcal{U}^\gamma_\beta(P)=\{Q:R_\beta(Q\|P)\leq R_\beta(Q_\gamma\|P)\},\\
&\mathcal{U}_I^\gamma(P)=\{Q:R_\alpha(Q\|P)\leq R_\alpha(Q_\gamma\|P)\,\text{ for all }\,\alpha\in I\},\notag
\end{align}
 i.e., ambiguity sets of the form  \req{eq:h_def} with $h(\alpha)= R_\alpha(Q_\gamma\|P)$, either just for $\alpha=\beta$ or for all $\alpha\in I$.

Then 
\begin{align}\label{eq:U_inclusion}
Q_\gamma\in \mathcal{U}_I^\gamma(P)\subset \mathcal{U}^\gamma_\beta(P)\,\text{  for any $\beta\in I$ }
\end{align}
and:
\begin{enumerate}
\item If $\gamma>0$ we have
\begin{align}\label{eq:tightness1_1}
\sup_{Q\in \mathcal{U}^\gamma_{(1,\infty)}(P)} \log\left[\int \tau dQ\right]=\sup_{Q\in \mathcal{U}^\gamma_{1+\gamma^{-1}}(P)} \log\left[\int \tau dQ\right]=& \log\left[\int \tau dQ_\gamma\right]\\
=& \inf_{c>1}\left\{\frac{1}{c} \log\left[\int \tau^cdP \right]+\frac{1}{c-1}R_{c/(c-1)}(Q_\gamma\|P)\right\}\notag,
\end{align}
and the infimum is achieved at $c=\gamma+1$.
\item If $\gamma\in(-1,0)$,  $\tau<\infty$ $P$-a.s., and  $E_P[ \tau^{\gamma+1}]<\infty$ then
\begin{align}\label{eq:tightness1_2}
\inf_{Q\in \mathcal{U}^\gamma_{(-\infty,0)}(P)} \log\left[\int \tau dQ\right]=\inf_{Q\in \mathcal{U}^\gamma_{1+\gamma^{-1}}(P)} \log\left[\int \tau dQ\right]=&\log\left[\int \tau dQ_\gamma\right]\\
=& \sup_{c<1,c\neq 0}\left\{\frac{1}{c}\log\left[ \int \tau^cdP\right]-\frac{1}{1-c}R_{1/(1-c)}(P||Q_\gamma)\right\}\notag,
\end{align}
and the supremum is achieved at $c=\gamma+1$.
\item If $\gamma<-1$ we have
\begin{align}\label{eq:tightness1_3}
\inf_{Q\in \mathcal{U}^\gamma_{(0,1)}(P)} \log\left[\int \tau dQ\right]=\inf_{Q\in \mathcal{U}^\gamma_{1+\gamma^{-1}}(P)} \log\left[\int \tau dQ\right]=&\log\left[\int \tau dQ_\gamma\right]\\
=& \sup_{c<1,c\neq 0}\left\{\frac{1}{c}\log\left[ \int \tau^cdP\right]-\frac{1}{1-c}R_{1/(1-c)}(P||Q_\gamma)\right\}\notag,
\end{align}
and the supremum is achieved at $c=\gamma+1$.
\end{enumerate}
\end{theorem}

\begin{remark}\label{remark:U_def2}
Theorem \ref{thm:tightness1} shows that, for fixed $P$ and $\tau$, tightness of the UQ bounds in Lemma \ref{thm:Renyi_UQ} holds over the ambiguity set of the form \req{eq:h_def} that is constructed from $h(\alpha)= R_\alpha(Q_\gamma\|P)$ for an appropriate range of $\alpha\in I$ ($I$ being related to $\gamma$ as stated in the above theorem).
\end{remark}

We also obtain a tightness result over a range of QoIs, for  fixed $P$ and $Q$.  The method of proof is very similar to the proof of Theorem \ref{thm:tightness1}; see \ref{app:tightness} for details.
\begin{theorem}\label{thm:tightness2}
Let $P$ and $Q$ be probability measures on $(\Omega,\mathcal{M})$. For $\rho:\Omega\to[0,\infty]$ measurable, define
\begin{align}
&\mathcal{T}^{+}_\beta(\rho)=\left\{\tau:\Omega\to[0,\infty]|\tau\text{ is measurable and }E_P[ \tau^\beta]\leq E_P[\rho^\beta ]\right\},\,\,\, \beta\neq 0,\\
&\mathcal{T}^{-}_\beta(\rho)=\left\{\tau:\Omega\to[0,\infty]|\tau\text{ is measurable and }  E_P[\tau^\beta]\geq E_P[\rho^\beta]\right\},\,\,\,\beta\neq 0.\notag
\end{align}
\begin{enumerate}
\item Suppose   $Q\ll P$. For $\nu>0$ define $\phi_\nu=(dQ/dP)^\nu$.

 Then
\begin{align}\label{eq:tightness2_1}
\sup_{\tau\in \mathcal{T}^{+}_{1+\nu^{-1}}(\phi_\nu)}\log\left[\int \tau dQ\right]=&\log\left[\int \phi_\nu dQ\right]\\
=&\inf_{c>1}\left\{\frac{1}{c}\log\left[\int \phi_\nu^c dP\right]+\frac{1}{c-1}R_{c/(c-1)}(Q\|P)\right\},\notag
\end{align}
 the infimum is achieved at $c=1+\nu^{-1}$, and
\begin{align}
\log\left[\int \phi_\nu dQ\right]=\nu(1+\nu)R_{1+\nu}(Q\|P).
\end{align}

\item Suppose $P\ll Q$.  For $\nu> 1$ define $\rho_\nu=(dP/dQ)^\nu$.
If $\int (dP/dQ)^{\nu-1} dP<\infty$ then 
\begin{align}\label{eq:tightness2_2}
\inf_{\tau\in\mathcal{T}^{-}_{1-\nu^{-1}}(\rho_\nu)}\log\left[\int \tau dQ\right]=&\log\left[\int \rho_\nu dQ\right]\\
=&\sup_{c<1,c\neq 0}\left\{\frac{1}{c}\log\left[\int \rho_\nu^c dP\right]-\frac{1}{1-c}R_{1/(1-c)}(P\|Q)\right\},\notag
\end{align}
 the supremum is achieved at $c=1-\nu^{-1}$, and
\begin{align}\label{eq:int_rho_nu}
\log\left[\int \rho_{\nu}dQ\right]=\nu(\nu-1)R_\nu(P\|Q).
\end{align}
\item Suppose $P\ll Q$.  For $\nu\in(0,1)$ define $\rho_{\nu}=(dP/dQ)^\nu$.
Then 
\begin{align}\label{eq:tightness2_3}
\inf_{\tau\in\mathcal{T}^{+}_{1-\nu^{-1}}(\rho_\nu)}\log\left[\int \tau dQ\right]=&\log\left[\int \rho_{\nu}dQ\right]\\
=&\sup_{c<1,c\neq 0}\left\{\frac{1}{c}\log\left[\int \rho_\nu^c dP\right]-\frac{1}{1-c}R_{1/(1-c)}(P\|Q)\right\},\notag
\end{align}
 the supremum is achieved at $c=1-\nu^{-1}$, and \req{eq:int_rho_nu} holds.

\end{enumerate}

\end{theorem}
\begin{remark}
One can prove a related tightness property for the KL-divergences: If $R(Q\|P)<\infty$, $R(P\|Q)<\infty$, and for $\nu>0$ we define $f^\pm_\nu=\pm\nu\log(dQ/dP)$ then
\begin{align}
\pm\inf_{c>0}\left\{\frac{1}{c}\log\left(E_P\left[e^{\pm c f^\pm_\nu}\right]\right)+\frac{1}{c}R(Q\|P) \right\}=E_Q[f^\pm_\nu]=\pm\nu R(Q\|P).
\end{align}
\end{remark}

\section{Applications}\label{sec:applications}

\subsection{Battery Failure Probability}\label{sec:Battery}

As our first application of the methods developed above, we illustrate the steps one might take in analyzing risk-sensitivity for a model obtained by fitting  to a data-set.  Specifically, we will analyze the  life-time,  $T$,   of lithium-ion batteries using the data set of $N=124$ battery failure times from \cite{Severson_battery} ($T$ is the  number of charge-discharge cycles over the battery's lifetime); see the left pane of  Figure \ref{fig:hist} for a histogram. 

\begin{figure}[h]
\minipage{0.5\textwidth}
  \includegraphics[width=\linewidth]{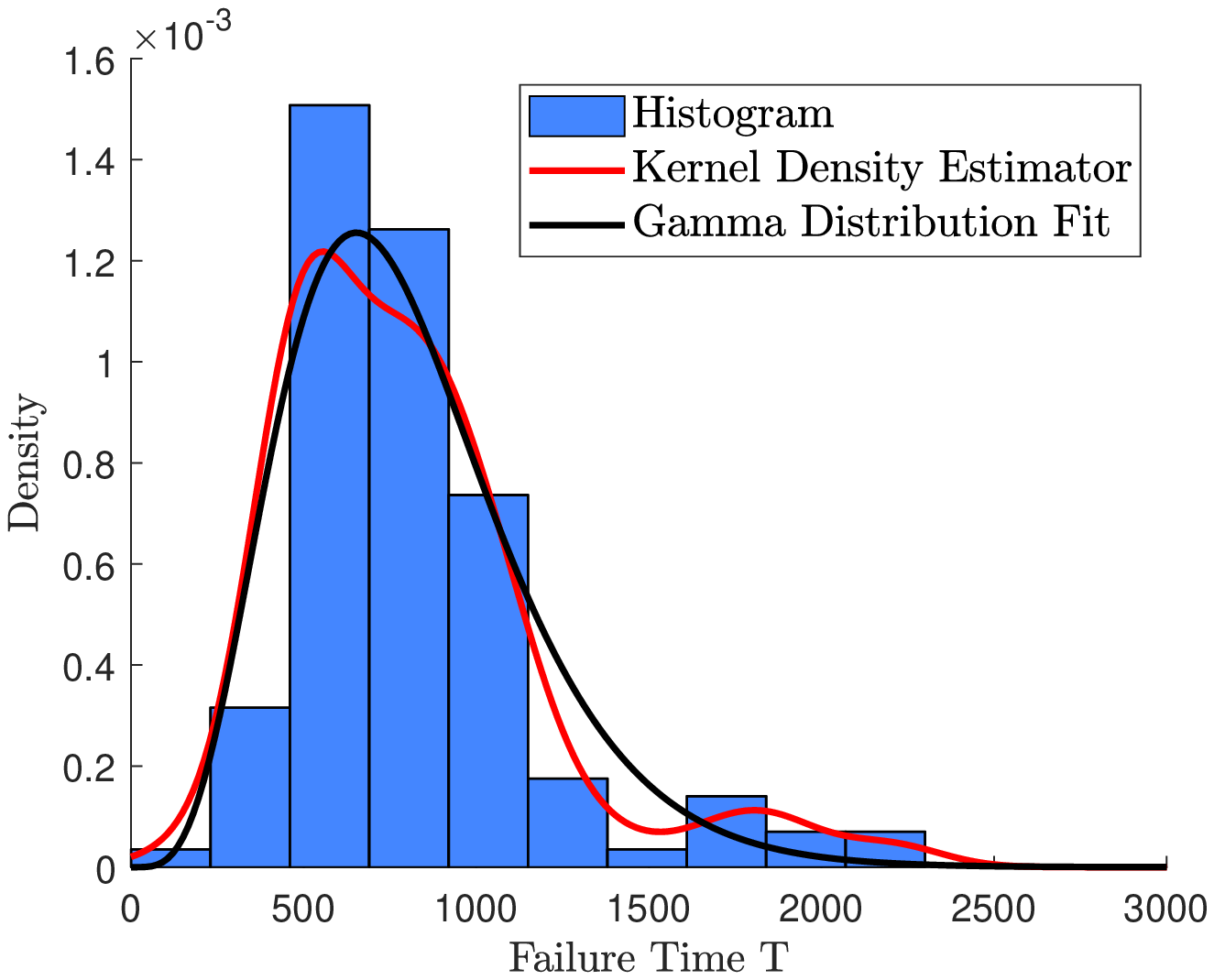}
\endminipage\hfill
\minipage{0.5\textwidth}%
  \includegraphics[width=\linewidth]{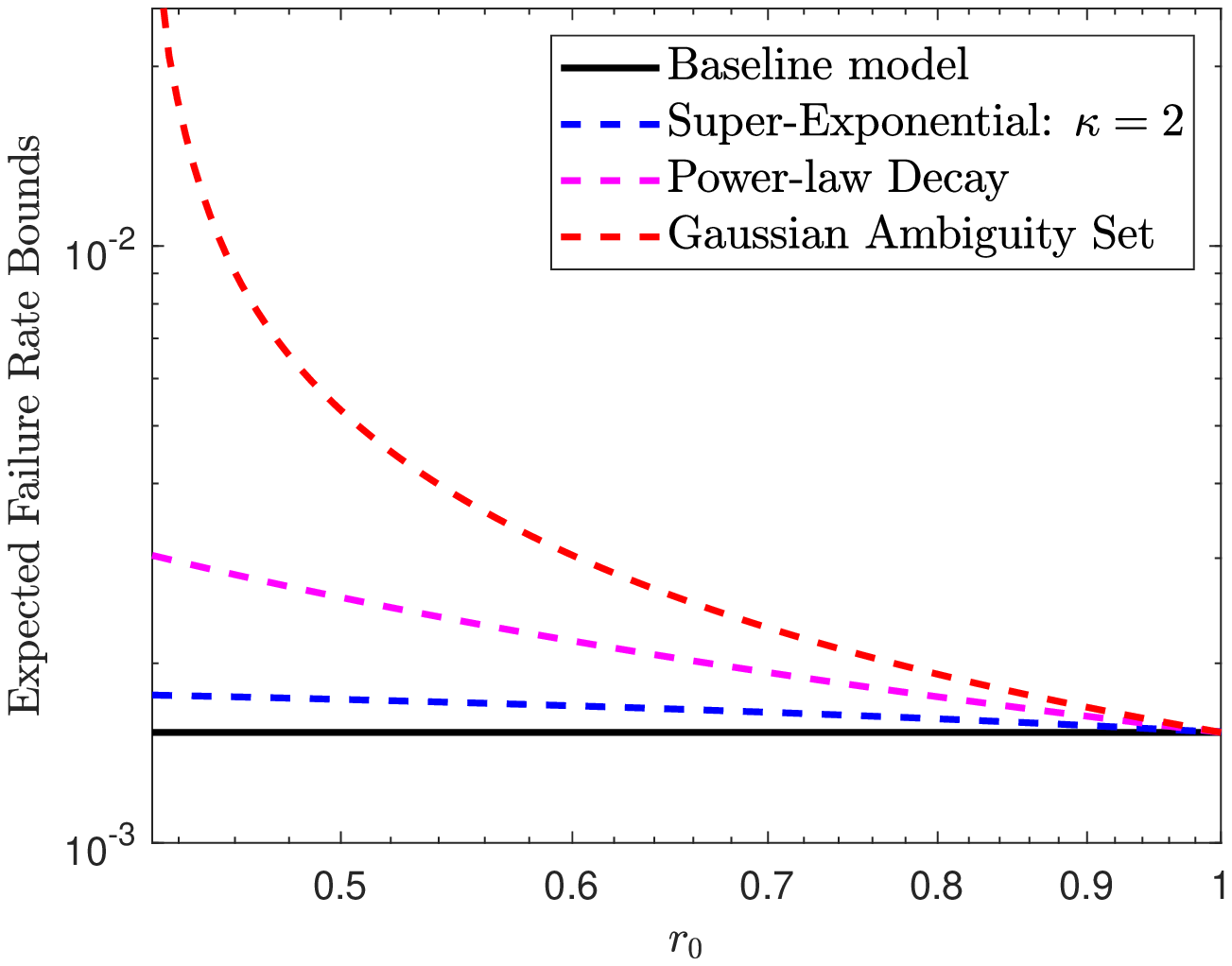}
\endminipage
\caption{Left: Histogram created from the 124 battery lifetimes, $T$, in the data-set from \cite{Severson_battery}  (see the supplementary material), together with the gamma-distribution fit and a kernel-density estimator. Right: Upper bounds on the expected battery failure rate, $E[1/T]$, computed using the risk-sensitive UQ bound \req{eq:rare_event_UB}, under the scenarios of  super-exponential decay (blue curve; see \req{eq:G_sub_exp} with $\kappa=2$), power-law decay (magenta curve; see \req{eq:G_power_law}),  and Gaussian ambiguity set (red curve; see \req{eq:G_gaussian}). The parameter $r_0$ is related to the relative entropy of the alternative models with respect to the baseline (see Section \ref{sec:ex_ambiguity_set} as well as the discussion below).   The non-risk-sensitive bound (\ref{eq:ex_KL_bound}), which is based on relative entropy, is not applicable here as the MGF of the QoI does not exist on any neighborhood of $0$.  The baseline model QoI (black solid line) was computed from a MLE fit to a gamma distribution; see \req{eq:gamma_density} - (\ref{eq:base_gamma}).}\label{fig:hist}
\end{figure}

 The steps suggested by the framework developed above are as follows:
\begin{enumerate}

\item{\bf Pick a QoI: }  If one is interested in risk-sensitive QoIs   then the R{\'e}nyi-based methodology  developed in this paper is appropriate (see the discussion in Section \ref{ex:rare_event}). Here we will study the battery failure rate: $\tau=1/T$, which is very sensitive to the tail of the distribution near $T=0$.

\item {\bf Construct the Baseline Model from the Data:}   The chosen QoI is most sensitive to the tail near zero (the `early' failure probability), and is relatively insensitive to the large $T$ tail. We choose to fit to a   gamma distribution, with probability density
\begin{equation}\label{eq:gamma_density}
  p_{a,b}(T)=\frac{1}{b^a\Gamma(a)}T^{a-1}e^{-T/b},\,\,\, t>0,
\end{equation}
as it provides a dedicated parameter, $a$, to fit the decay at $0$ (unlike, for example, the Weibull family, where the decay as $T\to0$ and $T\to\infty$ are linked).  

The maximum likelihood estimators (MLE) obtained from the data in \cite{Severson_battery} are used to define the baseline model:
\begin{align}\label{eq:base_gamma}
P(dT)=p_{a,b}(T)dT,\,\,\,\,\,a\equiv\widehat{a}=5.38,\,\,\,b\equiv\widehat{b}=149.
\end{align}

\item{\bf Build Ambiguity Set(s):} As is apparent form the frequency histogram, Figure \ref{fig:hist_frequency}, the data set contains very little information about the tails of the distribution.  The chosen QoI, $\tau=1/T$, is risk sensitive and so one should have a healthy skepticism regarding the value of $E_P[1/T]$ that is computed form the model constructed in step (2).  It is therefore important to test the robustness of the chosen QoI under  perturbations, $Q$, of the baseline model $P$. Performing a non-parametric robustness test using our methods requires the choice of an ambiguity set, defined via \req{eq:ambiguity_Lambda_Def} (or, more concretely, via \req{eq:U_mu_def}). The choice of ambiguity set should reflect the type of model-form  uncertainty one wishes to test  (i.e., the type of stress test). We discuss this choice in more detail below.

\item{\bf  Compute Robustness Bounds:} Now that we have a QoI and  ambiguity set(s), we can use Theorem \ref{thm:main_UQ_result} to obtain UQ bounds.  
\begin{remark}
Note that the MGF of the QoI $\tau=1/T$ under $P$ does not exist. This is another indication that the QoI is risk-sensitive and also implies that the non-risk-sensitive bound from Proposition \req{prop:info_ineq0} is not applicable in this case.  In contrast, the risk-sensitive bound (\ref{eq:UQ_upper}) only requires existence of $c$'th moments for $c$ in a neighborhood of $1$, and so is applicable in this case.  This is an extreme form of the phenomenon seen in Figure \ref{fig:bound_comparison}, where the risk-sensitive UQ bounds were seen to be tighter than the non-risk-sensitive bounds when the degree of risk-sensitivity (rarity) became sufficiently high.
\end{remark}

\end{enumerate}

For illustrative purposes, we characterize the system under  three tail-decay scenarios, using the framework of Section \ref{sec:tail_behavior}: power-law decay (\ref{eq:G_power_law}), super-exponential decay (\req{eq:G_sub_exp} with $\kappa=2$), and the Gaussian ambiguity set (\req{eq:G_gaussian}).   UQ bounds are shown in the right pane of Figure \ref{fig:hist} as a function of $r_0$. Each of the corresponding ambiguity sets are fixed once one specifies an upper bound on the relative entropy, as that fixes $r_0$ in  \req{eq:R_x0}, \req{eq:R_x0_2}, or \req{eq:R_x0_3}. There is a large body of literature on estimating relative entropy, and other density functionals \cite{10.2307/4616387,leonenko:hal-00322783,NIPS2010_4072,Gupta2010,Poczos:2011:NDE:3020548.3020618,NIPS2014_5505,UM2019658}, though none of these  works appear to rigorously apply to the task of estimating $R(Q\|P)$ in this situation (continuous distributions of unbounded support).  Here, we use a bootstrap method, where we take $N$ samples (with replacement) from the data set, use the sample to form a kernel density estimator and use that as a surrogate for $Q$ to compute $R(Q\|P)$.  Averaging the result over many different samples results in  $R(Q\|P)\approx 0.074$.

Selecting {\em which} ambiguity set to use, out of the three types shown in Figure \ref{fig:hist} (or out of the infinitely many other possibilities), is a modeling choice; each such choice represents a different type of stress-test of the model.   Given that we have little data regarding the tail at $0$, it is reasonable to stress-test the model via alternative models that decay  polynomially  at $0$ with a slower rate $\beta$, than the MLE fit, i.e., $\beta<a$.  Thinking of the tail near zero, such alternative models with have log-likelihood decay of the form:
\begin{align}
P(dQ/dP\geq r)\approx P( CT^{\beta-1}/T^{a-1}\geq r)\approx D r^{-a/(a-\beta)}
\end{align}
for $r$ large enough ($C$, $D$ are some constants).  This is of the power-law type (\ref{eq:G_power_law}).

\begin{remark}
Note that here, power-law decay of the likelihood under $P$ does correspond to power-law perturbation of the density near $0$.  This is because $P$ also has a power-law form near $0$.
\end{remark}

 Therefore, we are motivated to stress-test the model under the power-law family of ambiguity sets (\ref{eq:G_power_law}).  This, together with the estimate $R(Q\|P)\approx 0.074$, fixes $r_0\approx0.70$.  The resulting UQ bound obtained from Theorem \ref{thm:main_UQ_result} (see Figure \ref{fig:hist}) is $E_Q[1/T]\leq 0.00194$ for all $Q\in\mathcal{U}^{PL}_{r_0}(P)$, as compared to  the base-model expectation $E_P[1/T]=0.00153$. We emphasize that  we do {\em not} simply compute the QoI under the gamma distribution  \req{eq:gamma_density} with perturbed values of $a$; our goal with these methods is to  obtain robust UQ bounds over {\em non-parametric} model neighborhoods. 
\begin{remark}
We obtain a very similar bound $E_Q[1/T]\leq 0.00198$, if we stress-test via the Gaussian ambiguity set  (see \req{eq:G_gaussian}). Based on the discussion in Remark \ref{remark:G_PL_similar}, this is unsurprising.
\end{remark}

\subsection{Distributional Robustness for Large Deviations Rate Functions}\label{sec:rate_func}
Next, we show how our results can yield UQ bounds on large deviations rate functions. Specifically, we let $X_n$, $n\in\mathbb{Z}^+$ be IID $\mathbb{R}^d$-valued random variables with distribution $P$ and finite MGF everywhere, i.e.,
\begin{align}\label{eq:cramer_assump}
\int e^{v\cdot x}P(dx)<\infty,\,\,\,v\in\mathbb{R}^d.
\end{align}
 Let $P_n$ be the distribution of $\overline{X}_n\equiv \frac{1}{n}\sum_{i=1}^n X_i$.  Cramer's Theorem (see, for example, Section 2.2  in \cite{dembo2009large}) implies that  the $P_n$ satisfy a large deviations principle with rate function, $I_P$, given by
 \begin{align}\label{eq:rate_def}
 I_P(x)=\sup_{v\in\mathbb{R}^d}\left\{v\cdot x-\Lambda_P(v)\right\},\,\,\,\Lambda_P(v)\equiv\log\left[\int e^{v\cdot y}P(dy)\right].
 \end{align}
In the following result, we obtain UQ bounds on this type of  rate function:
\begin{theorem}
Let $P,Q$ be  probability measures  on $\mathbb{R}^d$.  Then, using the definitions from \req{eq:rate_def}, we have
  \begin{align}\label{eq:rate_function_bounds}
\sup_{c>1}\left\{\frac{1}{c}I_P(x)-\frac{1}{c-1}R_{c/(c-1)}(Q\|P)\right\} \leq I_Q(x)\leq \inf_{0<c<1}\left\{\frac{1}{c}I_P(x)+\frac{1}{1-c}R_{1/(1-c)}(P\|Q)\right\}
 \end{align}
 ($\infty-\infty\equiv-\infty$ in the lower bound).

\end{theorem} 
\begin{proof}
Lemma \ref{thm:Renyi_UQ} with $\tau(x)=e^{v\cdot x}$ gives
\begin{align}
\Lambda_Q(v)\leq \inf_{c>1}\left\{\frac{1}{c}\log\left[\int e^{cv\cdot y}P(dy)\right]+\frac{1}{c-1}R_{c/(c-1)}(Q\|P)\right\}.
\end{align}
Therefore, for $c>1$,
 \begin{align}
 I_Q(x)\geq&\frac{1}{c}\sup_{v\in\mathbb{R}^d}\left\{ cv\cdot x-\log\left[\int e^{cv\cdot y}P(dy)\right]\right\}-\frac{1}{c-1}R_{c/(c-1)}(Q\|P)\\
  =&\frac{1}{c}I_P(x)-\frac{1}{c-1}R_{c/(c-1)}(Q\|P).\notag
 \end{align}
 Taking the supremum over $c>1$ gives the lower bound.

 The derivation of the upper bound proceeds similarly:
  \begin{align}
\Lambda_Q(v)\geq& \sup_{c<1,c\neq 0}\left\{\frac{1}{c}\log\left[ \int  e^{cv\cdot y}P(dy)\right]-\frac{1}{1-c}R_{1/(1-c)}(P\|Q)\right\}
\end{align}
( $\infty-\infty\equiv-\infty$).   So for $0<c<1$ we have
 \begin{align}
 I_Q(x)\leq& \frac{1}{c}\sup_{v\in\mathbb{R}^d}\left\{ v\cdot x- \log\left[ \int  e^{v\cdot y}P(dy)\right]\right\}+\frac{1}{1-c}R_{1/(1-c)}(P\|Q)\\
 =&\frac{1}{c}I_P(x)+\frac{1}{1-c}R_{1/(1-c)}(P\|Q).\notag
 \end{align}
 (Note that we must restrict the supremum to $c>0$ so that we can pull the $1/c$ factor out of the supremum.) Taking the infimum over $c\in(0,1)$ gives the upper bound.
\end{proof}

As a concrete example, let $P$ be a normal distribution on $\mathbb{R}^d$ with mean $0$ and invertible covariance matrix $\Sigma$; the associated rate function   is $I_P(x)=\frac{1}{2}x\Sigma^{-1}x$.  Suppose we are considering a specific alternative model 
\begin{align}
Q(dx)=Z^{-1}(1+\phi(x))P(dx),
\end{align}
e.g., a `rough' non-Gaussian perturbation $\phi$, where  $Z$ is the normalization factor and we assume known bounds $0<m\leq  1+\phi(x)\leq M$ for all $x$.

Here we illustrate the use of the classical MGF bounds (in particular, the Bennett bound) that were discussed at the end of Section \ref{sec:MGF_ambiguity}: The assumed bounds on $1+\phi$ imply
\begin{align}\label{eq:Z_bound}
a\equiv& \log(m/M)\leq \log(m/Z) \leq \log(dQ/dP)\leq\log( M/Z)\leq \log( M/m)\equiv b,
\end{align}
(note that $a=-b$) and so the Bennett bound  \req{eq:bennett_bound} combined with  \req{eq:Renyi_MGF} gives
\begin{align}\label{eq:Renyi_bennett_bound}
R_\alpha(Q\|P)\leq \frac{1}{\alpha(\alpha-1)} \left((\alpha-1) b+\log\left( \frac{(b-\eta)^2}{(b-\eta)^2+\sigma^2}e^{-(\alpha-1)(\sigma^2/(b-\eta)+(b-\eta))}+\frac{\sigma^2}{(b-\eta)^2+\sigma^2}\right)\right)\equiv h(\alpha),\,\,\,  \alpha>1,
\end{align}
where $\eta\leq b$ is an upper bound on $R(Q\|P)$ and $\sigma^2$ is an upper bound on $\Var_Q[\log(dQ/dP)]$.  \req{eq:Renyi_bennett_bound} can be used in conjunction with \req{eq:rate_function_bounds} to obtain a lower bound on the rate function.   Results are shown in  Figure \ref{fig:rate_function_plots} for several values of $\eta$ and $\sigma^2$.   For comparison, we also show the  naive bound  (black dashed curve) $I_Q(x)\geq \max\{I_P(x)-\log(M/m),0\}$, obtained from using the upper bound (\ref{eq:Z_bound}) on $dQ/dP$ in the definition \req{eq:rate_def}, together with the fact that rate functions are non-negative. 
\begin{wrapfigure}{r}{.5\textwidth}
\vspace{-1.1cm}
\begin{center}
  \includegraphics[width=.5\textwidth]{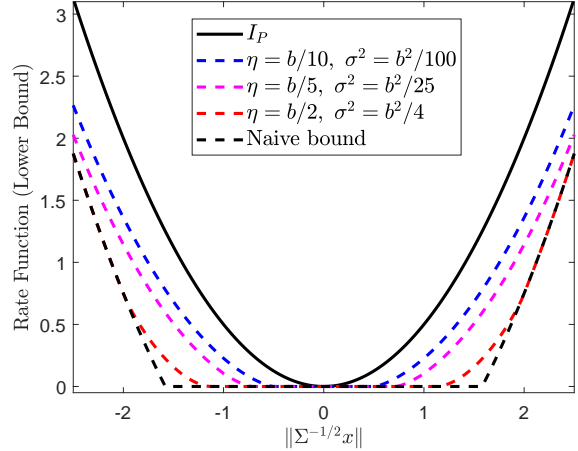}
\end{center}
\vspace{-.5cm}
\caption{ Lower bounds on the rate function for $Q$, computed from \req{eq:rate_def}. The base model is normal, $P=N(0,\Sigma)$, with rate function shown in black. We assume a bounded perturbation with $-b\leq\log(dQ/dP)\leq b=2$ and several values of $\eta=R(Q\|P)$ and $\sigma^2= \Var_Q[\log(dQ/dP)]$. }\label{fig:rate_function_plots}
\end{wrapfigure}

To understand the behavior in this  figure, note that when $\|\Sigma^{-1/2}x\|$ is large, the optimum value of $c$ in the lower bound from \req{eq:rate_function_bounds} will occur for $c\approx 1$, and so $\alpha\gg  1$ in \req{eq:Renyi_bennett_bound}.  Therefore the upper bound on $R_\alpha(Q\|P)$ will approach $b/\alpha$, and  the optimum value is approximately
\begin{align}
&\frac{1}{c^*}I_P(x)-\frac{1}{c^*-1}R_{c^*/(c^*-1)}(Q\|P)\\
\approx& \frac{1}{c^*}(I_P(x)-b)\approx I_P(x)-b,\notag
\end{align}
 the same result as the naive bound. Hence, all of the bounds converge to the naive bound for $\|\Sigma^{-1}x\|$ sufficiently large.  However, for intermediate values of  $\|\Sigma^{-1}x\|$ our bounds are much more informative; they do require the additional effort of computing (estimating)  $R(Q\|P)$ and  $\Var_Q[\log(dQ/dP)]$. Note that if one can also more accurately estimate $Z$, then a further improvement is be obtained by using  $b=\log(M/Z)$.

\subsection{Risk-Sensitive Distributionally Robust Optimization with an Application to Option Pricing}\label{ex:DRO}
Finally, as an application of the tightness results from Section \ref{sec:tightness},  we study a distributionally robust optimization (DRO) problem. DRO is an area with much recent work \cite{doi:10.1287/opre.1090.0741,doi:10.1287/opre.1090.0795,doi:10.1287/opre.2014.1314,Jiang2016,2016arXiv160402199G,2016arXiv160509349L,MohajerinEsfahani2018,doi:10.1137/16M1094725,doi:10.1287/moor.2018.0936}, wherein both information-theoretic and optimal-transport (Wasserstein-metric) based methods are widely used. In this section, we develop a DRO result applicable to risk-sensitive QoIs, using R{\'e}nyi divergences.

The general setting is as follows: Fix a probability measure $P$, a collection of QoIs $\tau_y:\Omega\to[0,\infty]$, $y\in Y$ (the parameter-space to be optimized over), and ambiguity sets $\mathcal{U}^y\subset\mathcal{P}(\Omega)$, $y\in Y$.   The goal in risk-sensitive distributionally robust optimization (DRO) is to compute one of the following:
 \begin{align}\label{eq:DRO_def}
 \argmin_{y\in Y} \sup_{Q\in\mathcal{U}^y} \log E_Q[\tau_y]\,\,\,\, \text{ or }\,\,\,\, \argmax_{y\in Y} \inf_{Q\in\mathcal{U}^y} \log E_Q[\tau_y].
 \end{align}
In the former, $\tau_y$ is often thought of as a cost, and the latter as a profit; in either case, $y$ is a design or control variable.  The problem \req{eq:DRO_def} then corresponds to finding the value of the control variable that minimizes (resp. maximizes) the worst case cost (resp. profit) over the ambiguity sets  $\mathcal{U}^y$.

\begin{remark}
The inclusion of the log in \req{eq:DRO_def} doesn't impact the minimizer/maximizer, but formulating the problem this way better connects with the results developed above. We need to treat the min-max and max-min problems separately, as our risk-sensitive UQ results only apply to non-negative QoIs, and hence we cannot use the standard technique of converting between the problem types simply by changing the sign of the QoI.
\end{remark}

Theorem \ref{thm:tightness1} gives conditions under which one can explicitly evaluate the inner maximization or minimization in \req{eq:DRO_def}, for specific classes of ambiguity sets. As a simple, concrete example, we consider the pricing of perpetual American put options.  For background see  Chapter 8 of \cite{shreve2004stochastic} or Chapter 8 of \cite{glasserman2013monte}. 

Under the baseline model, the asset price evolves according to geometric Brownian motion in an interest rate-$r$ environment and with volatility $\sigma$,
\begin{align}\label{eq:options_base}
 dX_t=rX_rdt+\sigma X_tdW_t,&
\end{align}
where  $r,\sigma>0$ are constants.  This has the explicit solution
\begin{align}\label{eq:X_sol}
X_t=X_0 \exp\left(\sigma W(t)+(r-\sigma^2/2)t\right).
\end{align} 
The baseline measure, $P$, is then the distribution of $X$ on path space, $C([0,\infty),\mathbb{R})$.

The quantity of interest we consider is defined as follows:  Let $K>0$ be the option strike price.   The  payout if the option is exercised at time $t\geq 0$ is $K-x_t$, where $x_t$ is the asset price at time $t$. The relevant QoI is then the value, discounted to the present time:
\begin{align}
V[x]_t=e^{-rt}(K-x_t) 1_{t<\infty}.
\end{align}

We assume that the option-holder's strategy is to exercise the option when the stock price hits some level $L$, assumed to satisfy  $0<L<K<X_0$, i.e.,  we consider the   stopping times $\tau_L[x]=\inf\{t\geq 0: x_t\leq L\}$.  The option holder's goal is to choose $L$ so as to maximize the expected option value at the stopping time, i.e., maximize $E_P[V_{\tau_L}]$.  However, due to lack of confidence in the model, one might desire to build in a safety margin by first maximizing over some ambiguity set of alternative models. With this as motivation, our goal here is  to compute the following:
\begin{align}\label{eq:V_L}
\argmax_{L}\inf_{Q\in\mathcal{U}} \log(E_Q[V_L]),\,\,\,\,\,V_L[x]\equiv e^{-r\tau_L}(K-x_{\tau_L}) 1_{\tau_L<\infty}=(K-L)e^{-r\tau_L}1_{\tau_L<\infty}
\end{align}
for an appropriate ambiguity set $\mathcal{U}$ (to be discussed below). 

From \req{eq:X_sol}, one can see that $\tau_L$ is the level $a_L$  hitting time of a Brownian motion with constant drift $\mu$, where
\begin{align}\label{eq:option_a_mu}
a_L\equiv-\sigma^{-1}\log(X_0/L),\,\,\,\,\,\,\mu\equiv r/\sigma-\sigma/2.
\end{align}
 From \req{eq:tightness1_2}, we see that if we choose $\gamma<0,\gamma\neq-1$ and 
\begin{align}\label{eq:put_ambiguity}
\mathcal{U}=\mathcal{U}^\gamma_{1+\gamma^{-1}}(P)=\{Q:R_{1+\gamma^{-1}}(Q\|P)\leq R_{1+\gamma^{-1}}(Q_\gamma\|P)\},\,\,\,dQ_\gamma\equiv  V_L^\gamma dP/E_P[V_L^\gamma]
\end{align}
($V_L^\gamma$ denotes the $\gamma$'th power) then the inner minimization is computable, and we obtain
\begin{align}\label{eq:option_optimization}
&\argmax_{L}\inf_{Q\in U^\gamma_{1+\gamma^{-1}}(P)} \log(E_Q[V_L])=\argmax_{L} \log(E_{Q_\gamma}[V_L])=\argmax_{L} E_{Q_\gamma}[V_L].
\end{align}
To make the above calculation valid, we must restrict to parameter values where $\mu<0$, so that $\tau_L<\infty$ $P$-a.s. (see Chapter 8 of \cite{shreve2004stochastic} and page 196 in \cite{karatzas2014brownian}), otherwise we will have $E_P[V_L^\gamma]=\infty$.

The expectation  can be evaluated using the formula for the MGF of $\tau_L$ under $P$ (again, see  Chapter 8 of \cite{shreve2004stochastic}):
\begin{align}
E_P[e^{-\lambda \tau_L}]=\exp(a_L\mu-|a_L|\sqrt{\mu^2+2\lambda}),\,\,\,\,\lambda>-\mu^2/2.
\end{align}
Using this we find for $0>\gamma>-\mu^2/(2r)$, $\gamma\neq -1$:
\begin{align}\label{eq:E_VL}
E_{Q_\gamma}[V_L]=&\int V_L^{1+\gamma} dP/E_P[V_L^\gamma]=  (K-L) \int e^{-r(1+\gamma)\tau_L}dP/\int e^{-r\gamma \tau_L}dP\\
=&(K-L) \exp(-|a_L|\sqrt{\mu^2+2r(1+\gamma)}+|a_L|\sqrt{\mu^2+2r\gamma})\notag\\
=&(K-L) (L/X_0)^{d_\gamma/\sigma}, \,\,\,\,\,d_\gamma\equiv\sqrt{\mu^2+2r(1+\gamma)}-\sqrt{\mu^2+2r\gamma}.\notag
\end{align}
One can calculate $R_{1+\gamma^{-1}}(Q_\gamma\|P)$ explictly, and thereby write the ambiguity set as
\begin{align}\label{eq:U_def}
\mathcal{U}^\gamma_{1+\gamma^{-1}}(P)=\left\{Q:R_{1+\gamma^{-1}}(Q\|P)\leq -\frac{\gamma}{\gamma+1}a_L\mu+\frac{\gamma^2}{\gamma+1}|a_L|\sqrt{\mu^2+2(\gamma+1)r}+\gamma|a_L|\sqrt{\mu^2+2r\gamma}\right\}.
\end{align}
Note that this ambiguity set only constrains $R_\alpha(Q\|P)$ at the single (but $\gamma$-dependent) value of $\alpha=1+\gamma^{-1}$; in the language of Definition \ref{eq:h_def}, the right-hand-side of \req{eq:U_def} defines $h(1+\gamma^{-1})$, with  $h$ being undefined otherwise.

For a simple example of a class of alternative models that are contained in $\mathcal{U}^\gamma_{1+\gamma^{-1}}(P)$, one can take $Q$ to be the distribution of $Y$ in path space, where $Y_t$ solves an SDE of the form
\begin{align}\label{eq:options_alt}
 dY_t=(r+\Delta r_t)Y_tdt+\sigma Y_tdW_t,&
\end{align}
and the process $\Delta r$ is bounded on $[0,T]$ and equals $0$ for $t>T$. A straightforward application Girsanov's theorem can then be used to show that
\begin{align}
R_\alpha(Q\|P)\leq \frac{\|\Delta r\|_\infty^2}{2\sigma^2}T,\,\,\,\text{ for all $\alpha>1$.}
\end{align}
Combining this with the right-hand-side of \req{eq:U_def}, we see that
\begin{align}\label{eq:U_def}
Q\in \mathcal{U}^\gamma_{1+\gamma^{-1}}(P)\,\,\,\text{ if }\, \,\, \frac{\|\Delta r\|_\infty^2}{2\sigma^2}T\leq -\frac{\gamma}{\gamma+1}a_L\mu+\frac{\gamma^2}{\gamma+1}|a_L|\sqrt{\mu^2+2(\gamma+1)r}+\gamma|a_L|\sqrt{\mu^2+2r\gamma}.
\end{align}

\begin{remark}
More general alternative models that (\ref{eq:options_alt}) are possible, including ones where $\Delta r$ has unbounded support, but  bounding the R{\'e}nyi divergences in such cases is a  more involved task; we do not enter into the techniques for doing so here, as they are tangential to our primary goal in this paper.

 It should also be noted that in more realistic option-pricing models, the option itself has a finite expiration date, $T$, and so the problem is naturally posed on  a compact time interval. However, analyzing such models is more complex and we do not discuss them further here; see \cite{shreve2004stochastic} for details.
\end{remark}

From \req{eq:E_VL}, one can compute the  maximizer in \req{eq:option_optimization}:
\begin{align}\label{eq:L_star}
L^*_\gamma\equiv\argmax_{L} E_{Q_\gamma}[V_L]=\frac{d_\gamma K}{d_\gamma+\sigma},\,\,\,\,d_\gamma=\sqrt{\mu^2+2r(1+\gamma)}-\sqrt{\mu^2+2r\gamma}.
\end{align}
It is important to note that $d_\gamma$ is decreasing and
\begin{align}
\lim_{\gamma\nearrow 0} L^*_\gamma=\frac{2rK}{2r+\sigma^2},
\end{align}
which equals the optimal $L$ in the non-risk-sensitive formulation of the problem (see 8.3.12 in \cite{shreve2004stochastic}).  Therefore we can interpret $\gamma<0$ as the degree of risk-aversion; when $\gamma$ decreases below $0$, $L^*_\gamma$ increases, meaning that a more conservative choice of stopping level is considered optimal when one has  uncertainty about one's model.

\appendix

\setcounter{theorem}{0}
    \renewcommand{\thetheorem}{\Alph{section}\arabic{theorem}}

\section{Proof of the R{\'e}nyi-Donsker-Varadhan Variational Formula}\label{app:gen_DV}
Here we prove the generalization of the Donsker-Varadhan variational formula from Theorem \ref{thm:gen_DV} and restated below:
\begin{theorem}[R{\'e}nyi-Donsker-Varadhan Variational Formula]
Let  $P,Q$ be probability measures on $(\Omega,\mathcal{M})$ and $\alpha\in\mathbb{R}\setminus\{0,1\}$. Then
\begin{align}\label{eq:Renyi_var}
R_\alpha(Q\|P)=\sup_{g\in B(\Omega)}\left\{\frac{1}{\alpha-1}\log\left[\int e^{(\alpha-1)g}dQ\right]-\frac{1}{\alpha}\log\left[\int e^{\alpha g}dP\right]\right\}.
\end{align}
\end{theorem}
\begin{proof}
If one can show \req{eq:Renyi_var} for all $\alpha>1$ and all $P,Q$, then, using the relation \req{eq:R_neg_alpha} and reindexing $g\to -g$ in the supremum, one find that \req{eq:Renyi_var} also holds for all $\alpha<0$.  So we only need to consider the cases $\alpha\in(0,1)$ and $\alpha>1$.

Using \req{eq:Renyi_MGF_sup} with $\beta=\alpha-1$, $\gamma=\alpha$, for all $g\in B(\Omega)$ we find
\begin{align}
\frac{1}{\alpha} \log\left[\int e^{\alpha g}dP \right]\geq 
\frac{1}{\alpha-1}\log\left[\int e^{(\alpha-1) g} dQ\right]-R_{\alpha}(Q\|P),
\end{align}
and hence
\begin{align}
R_\alpha(Q\|P)\geq\sup_{g\in B(\Omega)}\left\{\frac{1}{\alpha-1}\log\left[\int e^{(\alpha-1)g}dQ\right]-\frac{1}{\alpha}\log\left[\int e^{\alpha g}dP\right]\right\}\equiv H(Q\|P).
\end{align}

To show the reverse inequality, we separate into three cases:
\begin{enumerate}
\item $\alpha>1$ and $Q\not\ll P$: We will show $H(Q\|P)=\infty$, which will prove the desired inequality.  To do this, take a measurable set $A$ with $P(A)=0$ but $Q(A)\neq 0$.  Let $g=N1_A$.  Then
\begin{align}
H(Q\|P)\geq& \frac{1}{\alpha-1}\log\int e^{(\alpha-1)g}dQ-\frac{1}{\alpha}\log\int e^{\alpha g}dP\\
=&\frac{1}{\alpha-1}\log\left( e^{(\alpha-1)N}Q(A)+Q(A^c)\right)-\frac{1}{\alpha}\log P(A^c)\to \infty\notag
\end{align}
as $N\to\infty$, since $\alpha>1$.  
\item $\alpha>1$ and $Q\ll P$: In this case we have
\begin{align}
R_\alpha(Q\|P)=\frac{1}{\alpha(\alpha-1)}\log\left(\int (dQ/dP)^\alpha dP\right).
\end{align}
Let $f_{n,m}(x)=x1_{1/m<x<n}+n1_{x\geq n}+1/m1_{x\leq 1/m}$ and $g_{n,m}=\log(f_{n,m}(dQ/dP))$.  These are bounded and measurable and so
\begin{align}
H(Q\|P)\geq& \frac{1}{\alpha-1}\log\int e^{(\alpha-1)g_{n,m}}dQ-\frac{1}{\alpha}\log\int e^{\alpha g_{n,m}}dP\\
=&\frac{1}{\alpha-1}\log\int f_{n,m}(dQ/dP)^{(\alpha-1)}\frac{dQ}{dP}dP-\frac{1}{\alpha}\log\int f_{n,m}(dQ/dP)^\alpha dP.\notag
\end{align}
Define  $f_n(x)=x1_{x<n}+n1_{x\geq n}$. Using the dominated convergence theorem to take $m\to\infty$, we find 
\begin{align}
H(Q\|P)\geq&\frac{1}{\alpha-1}\log\int f_{n}(dQ/dP)^{(\alpha-1)}\frac{dQ}{dP}dP-\frac{1}{\alpha}\log\int f_{n}(dQ/dP)^\alpha dP\\
\geq&\frac{1}{\alpha(\alpha-1)}\log\int f_{n}(dQ/dP)^\alpha dP.\notag
\end{align}

We have $0\leq f_n(dQ/dP)\nearrow dQ/dP$, and so the monotone convergence theorem  gives
\begin{align}
H(Q\|P)\geq&\frac{1}{\alpha(\alpha-1)}\log\int(dQ/dP)^\alpha dP=R_\alpha(Q\|P).
\end{align}
This proves the result in this case.

\item $\alpha\in(0,1)$: Here, we have
\begin{align}
R_\alpha(Q\|P)=\frac{1}{\alpha(\alpha-1)}\log\left(\int_{p>0} q^\alpha p^{1-\alpha}d\nu\right),
\end{align}
where $\nu$ is any sigma-finite positive measure for which $dQ=qd\nu$ and $dP=pd\nu$. Define $f_{n,m}(x)$ as in the previous case and let $g_{n,m}=\log(f_{n,m}(q/p))$, where $q/p$ is defined to be $0$ if $q=0$ and $+\infty$ if $p=0$ and $q\neq 0$.

$g_{n,m}$ are bounded and measurable, hence
\begin{align}
H(Q\|P)\geq& -\frac{1}{1-\alpha}\log\int e^{(\alpha-1)g_{n,m}}dQ-\frac{1}{\alpha}\log\int e^{\alpha g_{n,m}}dP\\
=&-\frac{1}{1-\alpha}\log\int f_{n,m}(q/p)^{\alpha-1}qd\nu-\frac{1}{\alpha}\log\int f_{n,m}(q/p)^\alpha pd\nu.\notag
\end{align}
We can take $n\to\infty$ using the dominated convergence theorem (here it is critical that $\alpha\in(0,1)$) to get
\begin{align}\label{eq:H_lb}
H(Q\|P)\geq&-\frac{1}{1-\alpha}\log\int \frac{1}{\max(1/m,q/p)^{1-\alpha}}qd\nu-\frac{1}{\alpha}\log\int \max(1/m,q/p)^\alpha pd\nu\\
\geq&-\frac{1}{1-\alpha}\log\int_{q>0,p>0} q^\alpha p^{1-\alpha}d\nu-\frac{1}{\alpha}\log\int_{p>0} \max(1/m,q/p)^\alpha pd\nu.\notag
\end{align}
(Note that the second term is always finite.)

We can use the dominated convergence theorem on the second term to obtain
\begin{align}
H(Q\|P)\geq&-\frac{1}{1-\alpha}\log\int_{p>0} q^\alpha p^{1-\alpha}d\nu-\frac{1}{\alpha}\log\int_{p>0} q^\alpha p^{1-\alpha}d\nu\\
=&\frac{1}{\alpha(\alpha-1)}\log\int_{p>0} q^\alpha p^{1-\alpha}d\nu=R_\alpha(Q\|P). \notag
\end{align}
\end{enumerate}
This completes the proof.
\end{proof}

\section{Proof of UQ Bounds for Risk-Sensitive QoIs}\label{app:UQ_proof}

Here we prove the UQ bound for risk-sensitive QoIs, Lemma \ref{thm:Renyi_UQ}, extending Proposition \ref{prop:UQ1} to $g$ that are valued in the extended reals:\\
\begin{proof}
 First suppose $\tau:\Omega\to[0,\infty)$.    If $\int \tau^c dP=\infty$ or $R_{c/(c-1)}(Q\|P)=\infty$ then \req{eq:UQ_upper} is trivial, so suppose not.
 
   Define $\tau_n=1/n 1_{\tau<1/n}+\tau 1_{\tau\geq 1/n}$.  These are positive and decrease pointwise to $\tau$.  Therefore,  Proposition \ref{prop:UQ1} with $g=\log(\tau_n)$ gives
  \begin{align}
  \log\int \tau_n dQ\leq \frac{1}{c}\log \int \tau_n^c dP+\frac{1}{c-1}R_{c/(c-1)}(Q\|P).
  \end{align}
  By Fatou's Lemma we have $\int \tau dQ\leq \lim_n \int\tau_n dQ$, therefore
  \begin{align}
  &\log\int \tau dQ\leq \log  \lim_n \int\tau_n dQ  \leq\liminf_n \left(\frac{1}{c}\log \int \tau_n^c dP+\frac{1}{c-1}R_{c/(c-1)}(Q\|P)\right).
  \end{align}
  We are in the case where $R_{c/(c-1)}(Q\|P)<\infty$ and $\tau^c\in L^1(P)$.  Therefore $\tau_n^c=1/n^c1_{\tau<1/n}+\tau^c1_{\tau\geq 1/n}\leq 1+\tau^c\in L^1(P)$. The dominated convergence theorem then implies $\int \tau_n^c dP\to \int \tau^c dP$ and hence
    \begin{align}
\log\int \tau dQ  \leq&  \frac{1}{c}\log \int \tau^c dP+\frac{1}{c-1}R_{c/(c-1)}(Q\|P).
  \end{align}
This proves \req{eq:UQ_upper} when  $\tau:\Omega\to[0,\infty)$.

To prove \req{eq:UQ_lower} for  $\tau$ valued in $[0,\infty)$ we need to show that
\begin{align}
\log\left[\int \tau dQ\right]\geq\frac{1}{c}\log \int \tau^cdP-\frac{1}{1-c}R_{1/(1-c)}(P\|Q)
\end{align}
for all $c<1$, $c\neq 0$. If $R_{1/(1-c)}(P\|Q)=\infty$ or $\int \tau dQ=\infty$ this is trivial, so suppose they are both finite.

Define $\tau_n$ as above. First let $c\in(0,1)$:\\
   $0\leq \tau_n\leq 1+\tau\in L^1(Q)$ (in the case we are considering) so the dominated convergence theorem, combined with Proposition \ref{prop:UQ1} and Fatou's Lemma, give
\begin{align}
\log\left[\int \tau dQ\right]=&\lim_n \log\left[\int \tau_n dQ\right]\\
\geq& \liminf_n\left( c^{-1}\log\int \tau_n^c dP-1/(1-c)R_{1/(1-c)}(P\|Q)\right)\notag\\
\geq&  c^{-1}\log \int\tau^c dP -1/(1-c)R_{1/(1-c)}(P\|Q).\notag
\end{align}
Similarly, for $c<0$ we have
\begin{align}
\log\left[\int \tau dQ\right]=&\lim_n \log\left[\int \tau_n dQ\right]\\
\geq&\limsup_n \left(\frac{1}{c}\log \int \tau_n^c dP-\frac{1}{1-c}R_{1/(1-c)}(P\|Q)\right)\notag\\
\geq&\frac{1}{c}\log \int \tau^c dP-\frac{1}{1-c}R_{1/(1-c)}(P\|Q),\notag
\end{align}
where we used $\tau_n^c\leq \tau^c$.

Now consider the general case $\tau:\Omega\to[0,\infty]$. Applying the above calculations to $\tau_m=\tau 1_{\tau\leq m}+m1_{\tau>m}$  (which are valued in $[0,\infty)$) gives
\begin{align}
&\log\left[\int \tau_m dQ\right]\leq\frac{1}{c}\log\left[ \int \tau_m^c dP\right]+\frac{1}{c-1}R_{c/(c-1)}(Q\|P),\,\,\,c>1,
\end{align}
where we interpret $-\infty+\infty=\infty$, and 
 \begin{align}
&\log\left[\int \tau_m dQ\right]\geq \frac{1}{c}\log\left[ \int \tau_m^cdP\right]-\frac{1}{1-c}R_{1/(1-c)}(P\|Q),\,\,\,c<1,\,\,c\neq 0,
\end{align}
where $\infty-\infty=-\infty$.

Taking $m\to\infty$ via the monotone convergence theorem gives
\begin{align}
&\log\left[\int \tau dQ\right]\leq\frac{1}{c}\log\left[ \int \tau^c dP\right]+\frac{1}{c-1}R_{c/(c-1)}(Q\|P)
\end{align}
for $c>1$ and 
\begin{align}
&\log\left[\int \tau dQ\right]\geq \frac{1}{c}\log\left[ \int \tau^cdP\right]-\frac{1}{1-c}R_{1/(1-c)}(P\|Q)
\end{align}
for $0<c<1$.

For $c<0$, the monotone convergence theorem gives 
\begin{align}
&\log\left[\int \tau dQ\right]\geq\limsup_m\left( \frac{1}{c}\log\left[ \int \tau_m^cdP\right]-\frac{1}{1-c}R_{1/(1-c)}(P\|Q)\right),
\end{align}
where  we use the convention $\infty-\infty\equiv-\infty$.  We have $\tau_m^c\searrow \tau^c$, hence  if $\int \tau^cdP<\infty$ then  $\tau_m^c\leq 1+\tau^c\in L^1(P)$ and so the dominated convergence theorem implies
\begin{align}
&\log\left[\int \tau dQ\right]\geq\frac{1}{c}\log\left[ \int \tau^cdP\right]-\frac{1}{1-c}R_{1/(1-c)}(P\|Q).
\end{align}
If  $\int \tau^cdP=\infty$ then the above bound trivially holds. This completes the proof.
\end{proof}

\section{Proofs of the Divergence and Tightness Properties}\label{app:tightness}
\noindent The following is a proof of the divergence property, Theorem \ref{thm:div_property}:\\
\begin{proof}
\req{eq:div_def_plus} and \req{eq:div_def_minus} follow directly from Proposition \ref{prop:duality} and \req{eq:Xi_pm_relation} follows from reindexing $\beta\to-\beta$ in the infimum.

To prove non-negativity of $\Xi_+^\gamma$, note that the cumulant generating function $\Lambda(c)\equiv\log\int e^{cg}dP$ is smooth and  convex (strictly convex if $g$ is not $P$-a.s. constant), hence its derivative is non-decreasing.  Therefore
\begin{align}
\frac{1}{c}\Lambda(c)=\int_0^1\Lambda^\prime(cr)dr.
\end{align}
is non-decreasing (strictly increasing if  $g$ is not $P$-a.s. constant)  in $c$. This, together with non-negativity of the R{\'e}nyi divergences, implies $\Xi_+^\gamma(Q\|P,g)\geq 0$. 

If $Q=P$ then $R_\alpha(Q\|P)=0$ for all $\alpha$, and so taking $\beta\searrow\gamma$ gives $\Xi_+^\gamma(Q\|P,g)=0$.

To complete the proof of the divergence property for $\Xi^\gamma_+$,  suppose $g$ is not $P$-a.s. constant and  $\Xi_+^\gamma(Q\|P,g)=0$.  By the definition of $\Xi_+^\gamma$, there must  exist $\beta_n>\gamma$, $\beta_n\neq 0$ with
\begin{align}
 \frac{1}{\beta_n} \log\left[\int e^{\beta_n g}dP \right]-\frac{1}{\gamma}\log\left[\int e^{\gamma g} dP\right]+\beta_n^{-1}\frac{\beta_n}{\beta_n-\gamma}R_{\beta_n/(\beta_n-\gamma)}(Q\|P)\to 0.
\end{align}
The fact that $c\to\frac{1}{c}\Lambda(c)$ is strictly increasing implies that, after taking a subsequence, we have $\beta_{n}\searrow \gamma$.  Therefore
\begin{align}\label{eq:div_property_temp1}
0=&\lim_{n\to\infty}\frac{1}{\beta_{n}} \log\left[\int e^{\beta_{n} g}dP \right]-\frac{1}{\gamma}\log\left[\int e^{\gamma g} dP\right]+\beta_{n}^{-1}\frac{\beta_{n}}{\beta_{n}-\gamma}R_{\beta_{n}/(\beta_{n}-\gamma)}(Q\|P)\\
=&\lim_{n\to\infty} \frac{\beta_{n}}{\beta_{n}-\gamma}R_{\beta_{n}/(\beta_{n}-\gamma)}(Q\|P).\notag
\end{align}
By \req{eq:Renyi_nondec} we see that, if $\gamma>0$, the sequence  $\frac{\beta_{n}}{\beta_{n}-\gamma}R_{\beta_{n}/(\beta_{n}-\gamma)}(Q\|P)$ is nondecreasing.  Combined with \req{eq:div_property_temp1} and non-negativity, we find  $R_{\beta_{1}/(\beta_{1}-\gamma)}(Q\|P)=0$.  The divergence property for R{\'e}nyi, \req{eq:Renyi_div_property}, then gives $Q=P$.

If  $\gamma<0$, first use \req{eq:R_neg_alpha} to rewrite  \req{eq:div_property_temp1} as
\begin{align}
0=&\lim_{n\to\infty} \frac{\beta_{n}}{\beta_{n}-\gamma}R_{\beta_{n}/(\beta_{n}-\gamma)}(Q\|P)=-\lim_{n\to\infty} \frac{1}{1-\beta_{n}/\gamma}R_{1/(1-\beta_{n}/\gamma)}(P\|Q).
\end{align}
Taking $n$ large enough that $\beta_{n}<0$ we see that $n\to \frac{1}{1-\beta_{n}/\gamma}R_{1/(1-\beta_{n}/\gamma)}(P\|Q)$ is nondecreasing, and so
\begin{align}
0=\lim_{n\to\infty} \frac{1}{1-\beta_{n_k}/\gamma}R_{1/(1-\beta_{n_k}/\gamma)}(P\|Q)\geq \frac{1}{1-\beta_{n_1}/\gamma}R_{1/(1-\beta_{n_1}/\gamma)}(P\|Q)\geq 0.
\end{align}
Hence $R_{1/(1-\beta_{n_1}/\gamma)}(P\|Q)=0$ and therefore $P=Q$. Finally, the divergence property for $\Xi_-^\gamma$  follows from \req{eq:Xi_pm_relation}.

\end{proof}

\noindent We now prove the tightness results.  These calculations are similar to proving the equality case from  Proposition \ref{prop:UQ1}. First we give a proof of Theorem \ref{thm:tightness1}:\\
\begin{proof}

\begin{enumerate}
\item  Let $\gamma>0$. Using \req{eq:UQ_upper} from Lemma \ref{thm:Renyi_UQ}, bounding the infimum by the value at $c=1+\gamma$, and then using the definition of $\mathcal{U}^\gamma_{1+\gamma^{-1}}(P)$ we have
\begin{align}
\sup_{Q\in \mathcal{U}^\gamma_{1+\gamma^{-1}}(P)} \log\left[\int \tau dQ\right]
\leq&\sup_{Q\in \mathcal{U}^\gamma_{1+\gamma^{-1}}(P)} \inf_{c>1}\left\{\frac{1}{c}\log\left[ \int \tau^c dP\right]+\frac{1}{c-1}R_{c/(c-1)}(Q||P)\right\}\\
\leq&\sup_{Q\in \mathcal{U}^\gamma_{1+\gamma^{-1}}(P)}\left\{ \frac{1}{1+\gamma}\log\left[ \int \tau^{1+\gamma} dP\right]+\frac{1}{\gamma}R_{1+1/\gamma}(Q||P)\right\}\notag\\
\leq& \frac{1}{1+\gamma}\log\left[ \int \tau^{1+\gamma} dP\right]+\frac{1}{\gamma}R_{1+1/\gamma}(Q_\gamma||P).\notag
\end{align}
Direct calculation, using \req{eq:Renyi_formula}, gives
\begin{align}\label{eq:tight_temp1}
&\frac{1}{1+\gamma} \log\left[\int \tau^{1+\gamma} dP \right]+\frac{1}{\gamma}R_{1+1/\gamma}(Q_\gamma|P)=\log\left[\int\tau dQ_\gamma\right]\leq \sup_{Q\in \mathcal{U}^\gamma_{1+\gamma^{-1}}(P)} \log\left[\int \tau dQ\right].
\end{align}

Using Corollary  \ref{thm:Renyi_UQ} again yields
\begin{align}
\log\left[\int\tau dQ_\gamma\right]\leq&\inf_{c>1}\left\{\frac{1}{c}\log\left[ \int \tau^c dP\right]+\frac{1}{c-1}R_{c/(c-1)}(Q_\gamma||P)\right\}\\
\leq&\frac{1}{1+\gamma}\log\left[ \int \tau^{1+\gamma} dP\right]+\frac{1}{\gamma}R_{1+1/\gamma}(Q_\gamma||P).\notag
\end{align}
Combining these with \req{eq:U_inclusion} gives the claimed result, \req{eq:tightness1_1}.

\item  Let $\gamma\in(-1,0)$. Similarly to the previous case, using \req{eq:UQ_lower} from Lemma \ref{thm:Renyi_UQ}, bounding the supremum below by the value at $c=1+\gamma$, and using the definition of $\mathcal{U}^\gamma_{1+\gamma^{-1}}(P)$ along with \req{eq:R_neg_alpha}, we find
\begin{align}\label{eq:tight_temp2}
\inf_{Q\in \mathcal{U}^\gamma_{1+\gamma^{-1}}(P)} \log\left[\int \tau dQ\right]\geq&\inf_{Q\in \mathcal{U}^\gamma_{1+\gamma^{-1}}(P)} \sup_{c<1,c\neq 0}\left\{\frac{1}{c}\log\left[ \int \tau^cdP\right]-\frac{1}{1-c}R_{1/(1-c)}(P||Q)\right\}\\
\geq&\frac{1}{1+\gamma}\log\left[ \int \tau^{1+\gamma}dP\right]-\frac{1}{-\gamma}R_{1+1/\gamma}(Q_\gamma||P)\notag\\
=&\frac{1}{1+\gamma}\log\left[ \int \tau^{1+\gamma}dP\right]-\frac{1}{-\gamma}R_{-\gamma^{-1}}(P\|Q_\gamma).\notag
\end{align}

We assumed $E_P[\tau^\gamma]<\infty$ so $P(\tau=0)=0$.  We also assumed $P(\tau=\infty)=0$.  Therefore $P\ll Q_\gamma$ and we can use Lemma \ref{thm:Renyi_UQ} and the R{\'e}nyi entropy formula \ref{eq:Renyi_formula} to compute
\begin{align}\label{eq:tight_temp3}
&\log\left[\int \tau dQ_\gamma\right]\geq \sup_{c<1,c\neq 0}\left\{\frac{1}{c}\log\left[ \int \tau^cdP\right]-\frac{1}{1-c}R_{1/(1-c)}(P||Q_\gamma)\right\}\\
\geq& \frac{1}{\gamma+1}\log\left[ \int \tau^{\gamma+1}dP\right]-\frac{1}{-\gamma}R_{-\gamma^{-1}}(P||Q_\gamma)\notag\\
=& \frac{1}{\gamma+1}\log\left[ \int \tau^{\gamma+1}dP\right]+\frac{\gamma}{1+\gamma}\log\left[\int\tau^{\gamma+1}dP\right]-\log E_P[\tau^\gamma]\notag\\
=& \log\left[\int\tau dQ_\gamma\right].\notag
\end{align}
Here, it was critical that we assumed $\int \tau^{\gamma+1}dP<\infty$, so as to avoid the $\infty-\infty\equiv-\infty$ case.

Combining \req{eq:tight_temp2} with  \req{eq:tight_temp3} and  \req{eq:U_inclusion}   yields the claimed result, \req{eq:tightness1_2}.

\item The proof of \req{eq:tightness1_3} is very similar to the above proof of \req{eq:tightness1_2}; we omit the details.
\end{enumerate}
\end{proof}

\noindent  Finally, we give a proof of Theorem \ref{thm:tightness2}:\\
\begin{proof}
\begin{enumerate}
\item Let $\nu>0$. Applying \req{eq:UQ_upper} from Lemma \ref{thm:Renyi_UQ} and bounding the infimum above by the value at $c=1+\nu^{-1}>1$ gives
\begin{align}
\log\left[\int \phi_\nu dQ\right]\leq& \inf_{c>1}\left\{\frac{1}{c}\log\left[\int \phi_\nu^cdP\right]+\frac{1}{c-1}R_{c/(c-1)}(Q\|P)\right\}\\
\leq&\frac{1}{1+\nu^{-1}}\log\left[\int \phi_\nu^{1+\nu^{-1}}dP\right]+\frac{1}{\nu^{-1}}R_{(1+\nu^{-1})\nu}(Q\|P)\notag\\
=&\frac{\nu}{\nu+1}\log\left[\int \phi_\nu^{1+\nu^{-1}}dP\right]+\frac{1}{\nu+1}\log\left[\int (dQ/dP)^{\nu+1}dP\right]\notag\\
=&\log\left[\int\phi_\nu dQ\right].\notag
\end{align} 
Note that this holds even if $\int (dQ/dP)^{\nu}dQ=\infty$.

We similarly have
\begin{align}
\sup_{\tau\in \mathcal{T}^{+}_{1+\nu^{-1}}(\phi_\nu)}\log\left[\int \tau dQ\right]\leq&\sup_{\tau\in \mathcal{T}^{+}_{1+\nu^{-1}}(\phi_\nu)} \inf_{c>1}\left\{\frac{1}{c}\log\left[\int\tau^c dP\right]+\frac{1}{c-1}R_{c/(c-1)}(Q\|P)\right\}\notag\\
\leq&\sup_{\tau\in \mathcal{T}^{+}_{1+\nu^{-1}}(\phi_\nu)}\left\{\frac{1}{1+\nu^{-1}}\log\left[\int\tau^{1+\nu^{-1}} dP\right]+\frac{1}{\nu^{-1}}R_{(1+\nu^{-1})\nu}(Q\|P)\right\}\notag\\
\leq& \frac{1}{1+\nu^{-1}}\log\left[\int \phi_\nu^{1+\nu^{-1}}dP\right]+\frac{1}{\nu^{-1}}R_{(1+\nu^{-1})\nu}(Q\|P)\notag\\
=&\log\left[\int \phi_\nu dQ\right]\leq \sup_{\tau\in\mathcal{T}^{+}_{1+\nu^{-1}}(\phi_\nu)}\log\left[\int \tau dQ\right].\notag
\end{align}
Combining these proves \req{eq:tightness2_1}.

\item  Now let $\nu>1$, so that $0<1-1/\nu<1$. Using \req{eq:UQ_lower} from Lemma \ref{thm:Renyi_UQ} and then proceeding similarly to the previous case, we find
\begin{align}
\log\left[\int \rho_\nu dQ\right]\geq&\inf_{\tau\in\mathcal{T}^{-}_{1-1/\nu}(\rho_\nu)}\log\left[\int \tau dQ\right]\\
\geq&\inf_{\tau\in\mathcal{T}^{-}_{1-1/\nu}(\rho_\nu)}\sup_{c<1,c\neq 0}\left\{\frac{1}{c}\log\left[ \int \tau^cdP\right]-\frac{1}{1-c}R_{1/(1-c)}(P||Q)\right\}\notag\\
\geq&\inf_{\tau\in\mathcal{T}^{-}_{1-1/\nu}(\rho_\nu)}\left\{\frac{1}{1-1/\nu}\log\left[ \int \tau^{1-1/\nu}dP\right]-\frac{1}{1-(1-1/\nu)}R_{1/(1-(1-1/\nu))}(P||Q)\right\}\notag\\
\geq&\frac{1}{1-1/\nu}\log\left[ \int \rho_\nu^{1-1/\nu}dP\right]-\frac{1}{1-(1-1/\nu)}R_{1/(1-(1-1/\nu))}(P||Q)\notag\\
=&\frac{1}{1-1/\nu}\log\left[ \int (dP/dQ)^{\nu-1} dP\right]-\frac{1}{\nu-1}\log\int (dP/dQ)^{\nu-1} dP\notag\\
=&\log\left[ \int \rho_\nu dQ\right].
 \end{align}
In obtaining the last line, we used the assumption that $\int (dP/dQ)^{\nu-1} dP<\infty$.

Again using \req{eq:UQ_lower} from Lemma \ref{thm:Renyi_UQ} we find
\begin{align}
\log\left[ \int \rho_\nu dQ\right]\geq&\sup_{c<1,c\neq 0}\left\{\frac{1}{c}\log\left[\int\rho_\nu^cdP\right]-\frac{1}{1-c}R_{1/(1-c)}(P\|Q)\right\}\\
\geq&\frac{1}{1-1/\nu}\log\left[\int\rho_\nu^{1-1/\nu}dP\right]-\frac{1}{1-(1-1/\nu)}R_{1/(1-(1-1/\nu))}(P\|Q).\notag
 \end{align}
 
 Combining these proves \req{eq:tightness2_2}.
 \item  The proof of \req{eq:tightness2_3} is very similar to the above proofs of  \req{eq:tightness2_1} and \req{eq:tightness2_2}; we omit the details.

\end{enumerate}
\end{proof}

\section{Ambiguity Sets Arising from Classical MGF Bounds}\label{app:classical_CGF}
In Section \ref{sec:MGF_ambiguity} we saw that there is correspondence between tail behavior and the choice of $\Lambda$ defining the ambiguity set \ref{eq:ambiguity_Lambda_Def}. Here we give examples of different $\Lambda$'s, coming from classical MGF bounds and discuss the relations between the resulting ambiguity sets.
\begin{enumerate}

\item Sub-Gaussian:  Define the ambiguity set $\mathcal{U}_{\eta,\sigma}^{SG}(P)$ via:
\begin{align}\label{eq:sub_gauss}
\Lambda^{SG}_{\eta,\sigma}(\lambda)= \eta \lambda+\frac{\sigma^2}{2}\lambda^2,\,\,\,\lambda\geq 0.
\end{align}

\begin{remark}
Note that $\Lambda^{SG}_{\eta,\sigma}$ is the CGF of a normal distribution with mean $\eta$ and variance $\sigma^2$.
\end{remark}

\item Sub-gamma (i.e., Bernstein): Define $\mathcal{U}^{Bern}_{\eta,\sigma,M}(P)$ via
\begin{align}\label{eq:Bernstein}
&\Lambda^{Bern}_{\eta,\sigma,M}(\lambda)= \eta\lambda+\frac{\sigma^2/2}{1-M\lambda}\lambda^2,\,\,\,0\leq \lambda<1/M.
\end{align}

\item  Worst Case Regret: Define $\mathcal{U}_{b}^{WCR}(P)$ via:
\begin{align}
&\Lambda^{WCR}_b(\lambda)=  \lambda b,\,\,\,\lambda\geq 0.
\end{align}

\item Bennett-$(a,b)$:  Define $\Lambda^{B(a,b)}_{\eta}(\lambda)$ via:
\begin{align}\label{eq:Bennett_ab_CGF}
&\Lambda^{B(a,b)}_{\eta}(\lambda)= \lambda b+\log\left( \frac{b-\eta}{b-a}e^{-\lambda (b-a)}+\frac{\eta-a}{b-a}\right),\,\,\,\lambda\geq 0,\,\,a<b,\,\,\eta\in[a,b].
\end{align}

\begin{remark}
$\Lambda^{B(a,b)}_\eta$ is the CGF of  the  following probability measure, which has mean $\eta$  and is supported on $[a,b]$:
\begin{equation}\label{eq:nu_def}
\nu=\frac{b-\eta}{b-a}\delta_a+\frac{\eta-a}{b-a}\delta_b.
\end{equation}

\end{remark}

\item Bennett: Define $\mathcal{U}_{b,\eta,\sigma}^B(P)$ via:
\begin{align}\label{eq:Bennet}
&\Lambda^B_{b,\eta,\sigma}(\lambda)= \lambda b+\log\left( \frac{(b-\eta)^2}{(b-\eta)^2+\sigma^2}e^{-\lambda(\sigma^2/(b-\eta)+(b-\eta))}+\frac{\sigma^2}{(b-\eta)^2+\sigma^2}\right),\,\,\,\lambda\geq 0,\,\,\sigma>0,\,\,\eta< b.
\end{align}

\begin{remark}
$\Lambda^B_{b,\eta,\sigma}$ is the CGF the following probability measure, which has mean $\eta$,  variance $\sigma^2$, and is supported on $(-\infty,b]$:
\begin{align}
\nu=\frac{\sigma^2}{(b-\eta)^2+\sigma^2}\delta_b+\frac{(b-\eta)^2}{(b-\eta)^2+\sigma^2}\delta_{\eta-\sigma^2/(b-\eta)}.
\end{align}

\end{remark}
\end{enumerate}
  For further background on these CGF bounds and their applications, see \cite{BLM,dembo2009large,wainwright2019high}.

We have the following inclusions between these ambiguity sets:
\begin{lemma}\label{lemma:inclusion1} $\mathcal{U}_{\eta,\sigma}^{SG}(P)\subset \mathcal{U}^{Bern}_{\eta,\sigma,M}(P)$.
\end{lemma}
\begin{lemma}\label{lemma:inclusion2}
For $a,b,\eta$ with $b\geq 0$, $a<b$, $a\leq \eta\leq b$, we have $\mathcal{U}^{B(a,b)}_{\eta}(P)\subset \mathcal{U}^{WCR}_{b}(P)$.
\end{lemma}
\begin{proof}
This follows from 
\begin{align}
\frac{b-\eta}{b-a}e^{-(  b-a)(\alpha-1)}+\frac{\eta-a}{b-a}\leq 1.
\end{align}
\end{proof}

\begin{lemma}\label{lemma:inclusion3}
For $a,b,\eta$ with $a<b$, $a\leq \eta\leq b$ we  have $\mathcal{U}^{B(a,b)}_{\eta}(P)\subset \mathcal{U}^{SG}_{\eta,(b-a)/2}(P)$.
\end{lemma}
\begin{proof} 
For $\alpha>1$ we have
\begin{align}
&\frac{b}{\alpha}   +\frac{1}{\alpha(\alpha-1)}\log\left(\frac{b-\eta}{b-a}e^{-(  b-a)(\alpha-1)}+\frac{\eta-a}{b-a}\right)=\frac{1}{\alpha(\alpha-1)}  \log\left( \int e^{(\alpha-1)f}d\nu \right),
\end{align}
where $f=a1_a+b1_b$ and $\nu$ is defined by \req{eq:nu_def}. $f$ is a bounded measurable function with $a\leq X\leq b$ and $E_\nu[X]=\eta$, so Hoeffding's Lemma gives
\begin{align}
&\frac{b}{\alpha}   +\frac{1}{\alpha(\alpha-1)}\log\left(\frac{b-\eta}{b-a}e^{-(  b-a)(\alpha-1)}+\frac{\eta-a}{b-a}\right)\leq\frac{1}{\alpha} (\eta+(\alpha-1)(b-a)^2/8).
\end{align}
This proves the claimed inclusion.
\end{proof}

\begin{lemma}\label{lemma:inclusion4}
If $a,b,\eta,\sigma$ satisfy $a<\eta<b$, $0<\sigma^2\leq (b-\eta)(\eta-a)$ then we have $\mathcal{U}^{B}_{b,\eta,\sigma}(P)\subset\mathcal{U}^{B(a,b)}_{\eta}(P)$.
\end{lemma}
\begin{proof}
This follows from  the fact that the following function is  non-decreasing:
\begin{align}
&\sigma^2\to\frac{(b-\eta)^2}{(b-\eta)^2+\sigma^2}e^{-(\alpha-1)(\sigma^2/(b-\eta)+(b-\eta))}+\frac{\sigma^2}{(b-\eta)^2+\sigma^2}.
\end{align}

\end{proof}
\begin{remark}
The bound $\sigma^2\leq (b-\eta)(\eta-a)$ is motivated by the fact that the maximum variance of a random variable which is bounded between $a$ and $b$ and has mean $\eta$ is $(b-\eta)(\eta-a)$.
\end{remark}

\section{Ambiguity Sets from Two-Sided Tail Bound}\label{sec:2_sided_tail}

An alternative to defining ambiguity sets via Definition \ref{def:tail_decay_ambiguity_set} is to specify a two-sided bound on the deviation of the log-likelihood from the mean.  This leads leads to bounds on the moments and hence on the MGF; the proof of the following is similar to that of Lemma \ref{lemma:tail}:
\begin{lemma}\label{lemma:2-sided-tail}
Suppose $R(Q\|P)<\infty$ and we have  a measurable $G:[0,\infty)\to[0,\infty)$ with
\begin{align}\label{eq:2-sided-tail}
Q\left(|\log(dQ/dP)-R(Q\|P)|\geq r\right)\leq G(r),\,\,r\geq 0.
\end{align}
Then
\begin{align}\label{eq:moment_bound}
E\left[\left|\log(dQ/dP)-R(Q\|P)\right|^k\right]\leq \int_0^\infty G(z) kz^{k-1}dz, \,\,\,k\in\mathbb{Z},\,k\geq 2.
\end{align}
\end{lemma}
The moment bounds \req{eq:moment_bound} can  be used to bound the MGF via its Taylor series; see  Chapter 2 in \cite{wainwright2019high}.  The following corollary gives two classical examples:
\begin{corollary}\label{cor:tail_behavior}
\begin{enumerate}
\item Sub-Exponential Tail (i.e., sub-gamma or Bernstein bound):

Suppose \req{eq:2-sided-tail} holds with $G(r)=Ce^{-\beta r}$ for some $C>0$, $\beta>0$.  

Then, for $0\leq\lambda<\beta$, we have
\begin{align}
\Lambda_Q^{\log(dQ/dP)}(\lambda)\leq& \lambda R(Q\|P) +\log\left(1+\frac{C\lambda^2/\beta^2}{1-\lambda/\beta}\right)\label{eq:Lambda2_1}\\
\leq&  \lambda R(Q\|P) +\frac{C\lambda^2/\beta^2}{1-\lambda/\beta}. \label{eq:Lambda2}
\end{align}
\item Sub-Gaussian Tail:
 
 A sub-Gaussian bound,
 \begin{align} \label{eq:Lambda3}
\Lambda_Q^{\log(dQ/dP)}(\lambda)\leq  R(Q\|P)\lambda+\beta^2\lambda^2,\,\,\,\lambda\geq 0,
\end{align}
holds under any one of the following conditions:
\begin{enumerate}
\item If \req{eq:2-sided-tail} holds with $G(r)=2e^{-cr^2}$ then the sub-Gaussian bound \req{eq:Lambda3} holds with $\beta=3/\sqrt{2c}$.
\item Let  $Y\sim N(0,\tau^2)$ be a Gaussian random variable such that \req{eq:2-sided-tail} holds with $G(r)= c \,\mathbb{P}(|Y|\geq r)$ for some   $c\geq 1$. Then \req{eq:Lambda3} holds with $\beta=c\tau$.
\item If we have $a>0$ such that
\begin{align}
E_Q\left[\exp\left( a(\log(dQ/dP)-R(Q\|P))^2\right)\right]\leq 2
\end{align}
then \req{eq:Lambda3} holds with $\beta=\sqrt{3/(2a)}$.
\end{enumerate}

\end{enumerate}
\end{corollary}
\begin{remark}
We  call CGF bounds with the general forms \req{eq:Lambda2_1} (or \req{eq:Lambda2}) and \req{eq:Lambda3} sub-gamma (or Bernstein) and sub-Gaussian respectively. The classical form of the sub-gamma bound is \req{eq:Lambda2}, which often leads to simpler analytical formulas, but for practical purposes there is little reason to not use the tighter bound \req{eq:Lambda2_1}. 
\end{remark}
\begin{proof}
\begin{enumerate}
\item  First use Lemma \ref{lemma:2-sided-tail} to obtain 
\begin{align}
&E_Q\left[\left|\log(dQ/dP)-R(Q\|P)\right|^k\right]\leq \int_0^\infty Ce^{-\beta z} kz^{k-1}dz=Ck!\beta^{-k},\,\,\,k\geq 2.
\end{align}
(i.e., the Bernstein condition holds; see equation 2.16 in \cite{wainwright2019high}.)   Using this, we can bound the Taylor series
\begin{align}
\Lambda_Q^{\log(dQ/dP)}(\lambda)\leq \log\left[ e^{\lambda R(Q\|P)}\left(1+\sum_{k=2}^\infty\frac{\lambda^k}{k!}Ck!\beta^{-k}\right)\right]\leq \lambda R(Q\|P) +\log\left(1+\frac{C\lambda^2/\beta^2}{1-\lambda/\beta}\right).
\end{align}
\item The proof uses similar ideas to the sub-exponential case and the proof of Lemma \ref{lemma:tail}; see the proof of Theorem 2.1 in \cite{wainwright2019high} as well as \cite{Rivasplata2012SubgaussianRV}.
\end{enumerate}
\end{proof}

We end this section with several examples that illustrate the sub-Gamma and sub-Gaussian cases:
 \begin{enumerate}
\item Sub-gamma (i.e., Bernstein) bounds: 
\begin{enumerate}

\item  Let $P=N(\mu_1,\sigma_1)$, $Q=N(\mu_2,\sigma_2)$ be Gaussian distributions on $\mathbb{R}$ with means $\mu_i$ and variances $\sigma_i^2$, $i=1,2$, with $\sigma_2>\sigma_1$.  The R{\'e}nyi divergence (see \cite{GIL2013124}) is
\begin{align}\label{eq:gauss_Renyi}
R_\alpha(Q\|P)=\frac{1}{\alpha}\log(\sigma_1/\sigma_2)+\frac{1}{2}\frac{1}{\alpha(\alpha-1)}\log(\sigma_1^2/\sigma_\alpha^2)+\frac{(\mu_1-\mu_2)^2}{2\sigma_\alpha^2},\,\,\,\sigma_\alpha^2=\sigma_1^2-(\alpha-1)(\sigma_2^2-\sigma_1^2)
\end{align}
for all $\alpha>1$ that satisfy $\sigma_\alpha^2>0$. 

 One can bound
\begin{align}
\log(\sigma_1^2/\sigma_\alpha^2)=\log\left(1+\frac{(\alpha-1)(\sigma_2^2/\sigma_1^2-1)}{1-(\alpha-1)(\sigma_2^2/\sigma_1^2-1)}\right)\leq \frac{(\alpha-1)(\sigma_2^2/\sigma_1^2-1)}{1-(\alpha-1)(\sigma_2^2/\sigma_1^2-1)},
\end{align}
and hence
\begin{align}\label{eq:Gauss_Lambda_bound}
\lambda_Q^{\log(dQ/dP)}(\lambda)\leq \lambda\log(\sigma_1/\sigma_2)+\frac{1}{2}\frac{\lambda(\sigma_2^2/\sigma_1^2-1)}{1-\lambda(\sigma_2^2/\sigma_1^2-1)}+\frac{(\mu_1-\mu_2)^2}{2(\sigma_1^2-\lambda(\sigma_2^2-\sigma_1^2))}(1+\lambda)\lambda.
\end{align}
This can be bounded above by a Bernstein-type CGF (\ref{eq:Lambda2}) with $\beta^{-1}= \sigma_2^2/\sigma_1^2-1$. It should be emphasized that the upper bound \req{eq:Gauss_Lambda_bound} is only used to show that this family is  of Bernstein type. If one is explicitly interested in alternative models consisting  of Gaussians, one should simply use the exact formula \req{eq:gauss_Renyi} to define the corresponding ambiguity set.

\item A Bernstein bound can also be obtained directly from \req{eq:moment_bound} in the case where $\log(dQ/dP)$ is bounded, $|\log(dQ/dP)-R(Q\|P)|\leq M$ and $\Var_Q[\log(dQ/dP)]\leq\sigma^2$.  Using the Taylor series of the MGF, we obtain
\begin{align}
\Lambda_Q^{\log(dQ/dP)}(\lambda)=&\lambda R(Q\|P)+\log E_Q\left[\sum_{n=0}^\infty\frac{\lambda^n}{n!}\left(\log(dQ/dP)-R(Q\|P)\right)^n\right]\\
\leq&\lambda R(Q\|P)+\log\left(1+\frac{\sigma^2}{2M^2}\sum_{n=2}^\infty (\lambda M)^n\right)\notag\\
=&\lambda R(Q\|P)+\log\left(1+\frac{\sigma^2\lambda^2/2}{1-M\lambda}\right),\,\,\, 0\leq \lambda<1/M.\notag
\end{align}

\begin{comment}
\item  Let $dP=\mu e^{-\mu x}$, $x\geq 0$ and $dQ=e^{V(x)} P(dx)$ where $-a_- x-b_-\leq V(x)\leq a_+ x+b_+$ with $0<a_\pm<\mu$ (normalization factor is absorbed into $V$, and hence into the  constants $b_\pm$).  Then one can compute
\begin{align}
Q\left(\pm (\log(dQ/dP)-R(Q\|P))\geq r\right)\leq Q(x\geq (r\pm R(Q\|P)-b_\pm)/a_\pm),
\end{align}
which leads to 
\begin{align}
&Q\left(|\log(dQ/dP)-R(Q\|P)|\geq r\right)\leq Ce^{-\beta r},\,\,\,r\geq 0,\\
&\beta=\min\{\mu/a_\pm -1\},\,\,\,C=\frac{\mu e^{b_+}}{\mu-a_+}e^{(\mu/a_+-1)(b_+-R(Q\|P))}+\frac{\mu e^{b_-}}{\mu-a_-}e^{(\mu/a_--1)(R(Q\|P)+b_-)}
\end{align}
Therefore, the conditions of part 1 of Corollary \ref{cor:tail_behavior} are satisfied, and we have the sub-gamma bound
\begin{align}
\Lambda_Q^{\log(dQ/dP)}(\lambda)\leq  \lambda R(Q\|P) +\log\left(1+\frac{C\lambda^2/\beta^2}{1-\lambda/\beta}\right),\,\, 0<\lambda<\beta.
\end{align}
\end{comment}

\end{enumerate}

\item Sub-Gaussian bounds:
 \begin{enumerate}
\item Let $P=N(\mu_1,\Sigma)$, $Q=N(\mu_2,\Sigma)$ be Gaussian distributions on $\mathbb{R}^d$ with the same covariance matrices.  Then  
\begin{align}
R_\alpha(Q\|P)=\frac{1}{2}(\mu_1-\mu_2)^T\Sigma(\mu_1-\mu_2)=R(Q\|P)
\end{align}
for all $\alpha\geq 1$ (see \cite{GIL2013124}). Hence
\begin{align}
\Lambda^{\log(dQ/dP)}_Q(\lambda)=R(Q\|P)\lambda+R(Q\|P)\lambda^2,
\end{align}
which is of sub-Gaussian form (\ref{eq:Lambda3}).
\item The Hoeffding bound, \req{eq:Hoeff_bound}, is also of sub-Gaussian type.
\end{enumerate}
 \end{enumerate}

\section*{Acknowledgments}
The research of J.B., M.K., and L. R.-B.  was partially supported by NSF TRIPODS  CISE-1934846.  
The research of M. K. and L. R.-B.   was partially supported by the National Science Foundation (NSF) under the grant DMS-1515712 and by the Air Force Office of Scientific Research (AFOSR) under the grant FA-9550-18-1-0214. 
The research of P.D. was supported in part by the National Science Foundation (NSF) under the grant DMS-1904992 and by the Air Force Office of Scientific Research (AFOSR) under the grant FA-9550-18-1-0214.
The research of J.W. was partially supported by the Defense Advanced Research Projects Agency (DARPA) EQUiPS program under the grant W911NF1520122.

\section*{References}
%\printbibliography 
\bibliographystyle{elsarticle-num.bst}
\bibliography{ref}

\end{document}